\documentclass{amsart}
\usepackage[latin1]{inputenc}
\usepackage{amssymb}
\usepackage{amsmath}
\usepackage{enumerate}
\usepackage{epsfig,color,xcolor}

\usepackage[english]{babel}
\usepackage[all]{xy}

\usepackage{fullpage}
\usepackage{pgf,tikz}
\usepackage{mathrsfs}
\usetikzlibrary{arrows}
\usepackage{hyperref}
\usepackage{graphicx}
\newtheorem{theorem}{Theorem}[section]
\newtheorem{lemma}[theorem]{Lemma}
\newtheorem{prop}[theorem]{Proposition}
\newtheorem{proposition}[theorem]{Proposition}
\newtheorem{corollary}[theorem]{Corollary}

\newtheorem{definition}[theorem]{Definition}

\newtheorem{construction}[theorem]{Construction}

\newtheorem{rem}[theorem]{Remark}
\newtheorem{remark}[theorem]{Remark}

\numberwithin{equation}{section}
\newcommand{\rk}{{\rm rank}}
\newcommand{\id}{\mbox{id}}
\newcommand{\NS}{\mbox{NS}}
\newcommand{\Pic}{\mbox{Pic}}
\newcommand{\T}{\mbox{T}}

\newcommand{\ra}{\rightarrow}

\newcommand{\PP}{ \mathbb{P}}

\newcommand{\Z}{\mathbb{Z}}
\newcommand{\Q}{\mathbb{Q}}

\newcommand{\Aut}{\mbox{Aut}}
\newcommand{\ZZ}{\mathbb{Z}}
\def\ZZ{\mathbb{Z}}
\def\QQ{\mathbb{Q}}

\def\CC{\mathbb{C}}

\def\QQl{\QQ_\ell}

\def\HB{H}

\def\Spec{\textnormal{Spec}}
\def\Mot{\textnormal{Mot}}
\def\AbMot{\textnormal{AbMot}}
\def\HH{\mathrm{H}_{\textnormal{mot}}}
\newcommand{\oo}{\mathcal{O}}
\def\ChMot{\mathcal{M}_{\textnormal{rat}}}
\def\SmPr{\textnormal{SmPr}}
\def\opp{\textnormal{op}}

\def\AJ{\textnormal{AJ}}
\def\Spec{\textnormal{Spec}}
\def\Mot{\textnormal{Mot}}
\def\AbMot{\textnormal{AbMot}}
\def\HH{\mathrm{H}_{\textnormal{mot}}}
\def\GB{G_{\textnormal{MT}}}
\def\Gl{G_\ell}
\def\Glc{\Gl^\circ}

\def\HB{H}
\def\Hl{H_\ell}

\def\tra{\textnormal{tra}}
\def\alg{\textnormal{alg}}

\usepackage{imakeidx}

\newtheoremstyle{dico}
{\baselineskip}   
{\topsep}   
{}  
{0pt}       
{} 
{.}         
{5pt plus 1pt minus 1pt} 
{}          
\theoremstyle{dico}

\makeatletter
\def\blfootnote{\xdef\@thefnmark{}\@footnotetext}

\title{Hodge Structures of K3 type of bidouble covers\\ of rational surfaces}
\author{Alice Garbagnati and Matteo Penegini}
\begin{document}

\subjclass[2010]{14E20, 14J28, 14J29 (14C30, 14C34, 14D)}
\keywords{Bidouble covers, K3 surfaces, surfaces of general type, Hodge structures, Mumford--Tate conjecture, Torelli problems} 

			\maketitle

	\begin{abstract}  A bidouble cover  between complex algebraic varieties is a flat $G:=\left(\ZZ/2\ZZ\right)^2$-Galois cover $X \ra Y$. In this situation there exist three intermediate quotients $Y_1,Y_2$ and $Y_3$ which correspond to the three subgroups $\ZZ/2\ZZ \leq G$. We consider the following situation: $Y$ will be a rational surface and $Y_i$ will be either a surface with $p_g=0$ or a K3 surface. These assumptions will enable us to have a strong control on the weight 2 Hodge structure of the covering surface $X$. In particular, we classify all covers with these properties if $Y$ is minimal, obtaining surfaces $X$ with $p_g(X)=1,2,3$. Moreover, we will discuss the Infinitesimal Torelli Property, the Chow groups and Chow
		motive, and the Tate and Mumford-Tate conjectures for $X$. 
		We also introduce another construction, called \emph{iterated bidouble cover}, which allows us to obtain surfaces with higher value of $p_g$ for which we still have a strong control on the weight 2 Hodge structure.
	\end{abstract}

	\section{Introduction}	
	
	A bidouble cover between complex algebraic varieties is a $G:=\left(\ZZ/2\ZZ\right)^2$-Galois flat cover $X \ra Y$. F. Catanese in \cite[pp. 491--493]{CatM} has developed the theory of this kind of covers giving the structure theorem in the smooth case. The main feature of this construction is that there exist three intermediate quotients $Y_1,Y_2$ and $Y_3$ which correspond to the quotients given by the  three subgroups $\ZZ/2\ZZ \leq G$. More precisely, given a bidouble cover $\pi\colon X \ra Y$ we have a commutative diagram
	\begin{align}\label{eq: diagram bidouble cover}
		\xymatrix{
			&X\ar[dr]^{p_1}\ar[dl]_{p_3}\ar[d]^{p_2}\\
			Y_3\ar[dr]_{\pi_3}&Y_2\ar[d]_{\pi_2}&Y_1\ar[dl]^{\pi_1}\\
			&Y}
	\end{align}
	where all the arrows are double covers.

	We will consider the following situation: $Y$ will be a minimal rational surface, $Y_i$ will be either a surface with geometric genus zero or a surface K3, and for at least one $i$, $Y_i$ is a K3 surface. 
	This situation has been considered previously by several authors e.g. \cite{CatM, Cat99, G,Lat19,Lat21,Lat22,PZ19, FP} who studied some specific properties of these coverings, but no systematic study of this construction has been undertaken. 
	
	In this work, first of all we classify the bidouble covers $X\ra Y$ whose intermediate double covers $Y_i$ are as stated above and $Y=\PP^2,\PP^1 \times \PP^1$ or the Hirzebruch surface $\mathbb{F}_n$ with $n\geq 2$. This gives surfaces $X$ with $p_g(X)\leq 3$. After that, we discuss the weight 2 Hodge structure of $H^2(X)$. It decomposes into Hodge substructures and we show that, on a Hodge-theoretic level, the cohomology is described only by the cohomology of the intermediate quotients $Y_i$ which are K3 surfaces. This remarkable property will lead us to more and deeper conclusions. Namely, we give conditions which assure that the Infinitesimal Torelli Property holds for certain bidouble covers and we apply this criterion to some of the surfaces constructed.
	In particular, we can deduce the Infinitesimal Torelli Property for GS2 surfaces (see Remark \ref{rem GS2}) in a similar manner as it was done for Special Horikawa Surfaces in \cite{PZ19}.  Both the GS2 and the Special Horikawa Surfaces are surfaces of general type and bidouble covers of $\mathbb{P}^2$. It is worth to notice here that the Infinitesimal Torelli Property fails for special Kunev surfaces that can be obtain as a bidouble cover of $\PP^2$ as we shall see. Furthermore, we study the Chow Groups and Chow Motives for these surfaces as in \cite{Lat21, Lat22, Lat19}. Finally, we prove the Mumford--Tate conjecture and consequently the Tate conjecture for almost all the surfaces $X$ we constructed such that $p_g(X)=2,3$ following an argument similar to that in \cite{CP} and \cite{Com19}. 
	
	We can summarize our main results in following theorem.

	\begin{theorem}\label{thm main}  Let $X\ra Y$ be a smooth bidouble whose branch divisor is simple normal crossing. Let us assume that each intermediate double cover $Y_i$ is either a K3 surface or a surface with $p_g=0$ and that at least one of them is a K3 surface. The following holds.\begin{enumerate}  
			\item If $Y=\PP^2$, there are five possibilities for $X$, listed in Table \ref{Table1}. Using the enumeration in the Table \ref{Table1} we prove that the Mumford--Tate conjecture and the Tate conjecture hold in the cases $a), \ldots ,e)$, while the Infinitesimal Torelli Property holds in the cases $c)$, $d)$ and $e)$ and it does not hold in case $a)$ (see Theorem \ref{theorem: bidouble P2}). 
			
			\medskip
			
			\item If $Y=\PP^1\times\mathbb{P}^1$, there are eight possibilities for $X$, listed in Table \ref{TableP1xP1 general}, when  $X$ is of general type. If $K_X^2=0$ the surface $X$ is a minimal surface admitting a genus 1 fibration and there are fifteen possibilities, given in Table \ref{TableP1xP1 elliptic}. The Mumford--Tate conjecture and the Tate conjecture hold in the cases $a), \ldots ,h)$, the Infinitesimal Torelli Property holds in cases $f)$, $p)$, $r)$ and $w)$ and does not hold in cases $a)$, $b)$ and $c)$ (see Theorem \ref{theor: Y=P1xP1}).
			
			\medskip
			
			\item  If $Y$ is the Hirzebruch-Segre surface $\mathbb{F}_n$ with $n\geq 2$ then $n\leq 4$ and there are seventeen possibilities for $X$, listed in Tables \ref{TableF2 irrid}, \ref{TableF2 rid} \ref{TableF3},\ref{TableF4}.
The Mumford--Tate conjecture and the Tate conjecture hold in the cases $b)$, $f)$, $h)$, $i)$, $j)$, $m)$, $o)$, $p)$. The Infinitesimal Torelli Property holds in cases $a)$, $b)$, $f)$, $j)$ and $p)$  (see Theorem \ref{theor: Y=Fn}). 
			
			\medskip
			
			\item
			
			 The previous classification produces smooth minimal regular surfaces of general type $X$ with $$(p_g,K^2)\in\{(1,1), (1,2),  (1,4), (1,8), (2,1), (2,2), (2,4), (2,5),(2,6), (2,8), (3,8),(3,9)\}.$$
			 
			\medskip
			
			\item Finally, all the surfaces $X$ above have the following decomposition of their  transcendental lattices, Chow groups and  Chow motives:
			\[ T_X\otimes \Q\simeq (T_{Y_1}\oplus T_{Y_2}\oplus T_{Y_3})\otimes \Q, \quad \ A_0(X)\otimes \Q\simeq (A_0(Y_1)\oplus A_0(Y_2)\oplus A_0(Y_3))\otimes \Q,
			\] 
			\[ 
			h(X)_{\rm{tra}}\otimes \Q=(h(Y_1)_{\rm{tra}}\oplus h(Y_2)_{\rm{tra}}\oplus h(Y_3)_{\rm{tra}}) \otimes \Q,
			\]
 where all the transcendental Hodge structures of the surfaces $Y_i$ are either trivial or the ones of a K3 surface.
			
		\end{enumerate}
	\end{theorem}

	To complete this systematic work, we also study singular bidouble covers in case $X$ is of general type, i.e., we allow some singularities for $X$ see \cite{Cat99}. By resolving the singularities of $X$ we can build some families that are not listed in the previous Theorem \ref{thm main} and whose transcendental Hodge structures, Chow groups and Chow motives enjoy the same properties as above. In this way we construct smooth minimal regular surfaces of general type $X$ with $(p_g,K^2)\in \{(2, k), (3,h)\}$ where $1\leq k\leq 6$ and $1\leq h\leq 9$ and we verify that the Mumford--Tate and the Tate conjectures hold for these surfaces.

	Finally, we propose a new construction that we called \emph{iterated bidouble cover}. An iterated bidouble cover of $X \ra Y$ is simply another bidouble cover having a rational surface, say  $Y_1$, as a base surface for a new bidouble cover which is nested in the previous construction, i.e. $X$ is an intermediate quotient for the new bidouble cover (see Subsection \ref{subsec: Iterated general case}). With this construction we obtain a new double cover $W \ra X$ such that $p_g(W)=3,4,5$. The main results here is that the weight 2 Hodge structure of $W$ is completely determined by the cohomology of the K3 surfaces which appear in the two bidouble covers and again we can draw conclusions about Chow groups and motives, see Theorems \ref{theor: iterated bidouble GS2}, \ref{theor: iterated bidouble GS2 case 2}, Section \ref{subsec: iterated SHS1} and Theorems \ref{theor: iterated P1xP1 case 1}, \ref{theor: iterated P1xP1 case 2}. 
	
	Now, let us explain  the way in which this paper is organized.
	
	In Section \ref{preliminaries} we give all the necessarily preliminaries on bidouble cover following mainly \cite{CatM}. Moreover, we present the iterated bidouble construction (see Subsection \ref{subsec: Iterated general case}). In the last paragraph, we briefly recall the singular bidouble cover case.

In Section \ref{Sec:General}, we discuss some general results on smooth bidouble covers, singular bidouble covers and iterated bidouble covers. In particular, in Subsection \ref{subsec: Infinitesimal Torelli}, we develop the theory of Infinitesimal Torelli Property for covering surface following \cite{Par91b} and we prove the Theorem \ref{theo: Infinitesimal Torelli}. Finally in Subsection \ref{sec_K3Hodge}, we will discuss the Chow groups (and Chow motives) and the Tate and Mumford-Tate conjectures of these surfaces, in particular we prove the Theorem \ref{Theo_MTC}. All these features are independent from the base surface $Y$ as long as it is a rational surface. After this section we will study the situation accordingly to the choice of a specific $Y$. 
	
	In Section \ref{sec: P2}, we study the case with $Y=\PP^2$. We first give the classification theorem for bidouble covers, see Theorem \ref{theorem: bidouble P2}. Then we study the singular bidouble covers case. After that we work on the iterated bidouble covers with a particular care for Special Horikawa surfaces and GS2 surfaces, see Theorems \ref{theor: iterated bidouble GS2}, \ref{theor: iterated bidouble GS2 case 2}, \ref{theo: iterated SHS1}, \ref{theo: iterated SHS1 case 2}. 
	
	In Section \ref{sec: P1xP1}, we study the case with $Y=\PP^1 \times \PP^1$. We follow the same plane of the previous section giving the classification theorem \ref{theor: Y=P1xP1} and then working on singular and iterated bidouble covers, see Theorems \ref{theor: iterated P1xP1 case 1}, \ref{theor: iterated P1xP1 case 2} . 
	
	Finally in Section \ref{sec: Fn}, we study the construction when $Y=\mathbb{F}_n$, giving the classification theorem \ref{theor: Y=Fn} and discussing the iterated bidouble covers in Section \ref{subsec: iterated Fn}.
	\medskip
	
	\textbf{Notation and conventions.} 
	
	We work over the field $\mathbb{C}$ of complex numbers. 
	
	For $a,b\in \Z$, $a\equiv_n b$ means $a\equiv b\mod n$.
	
	If $G$ is an abelian group, then the set of characters (denoted usually by $\chi_i$)  forms an abelian group under point-wise multiplication,  we denote by $G^*$ its character group. 
	
	By ``\emph{surface}'' we mean a projective, connected non-singular two dimensional complex manifold $S$, and
	for such a surface $\omega_S=\oo_S(K_S)$ denotes the canonical
	class, $p_g(S)=h^0(S, \, \omega_S)$ is the \emph{geometric genus},
	$q(S)=h^1(S, \, \omega_S)$ is the \emph{irregularity} and
	$\chi(\mathcal{O}_S)=1-q(S)+p_g(S)$ is the \emph{Euler-Poincar\'e
		characteristic}. If $q(S)>0$, we call $S$ an \emph{irregular surface}.
By an abuse of notation we will call ``\emph{K3 surface}" also singular surfaces, whose minimal model is a K3 surface.
	
	Throughout the paper, we denote Cartier (or Weil) divisors on a variety by
	capital letters and the corresponding line bundles by italic
	letters, so we write for instance $\mathcal{L}=\oo_S(L)$. 
	
	\medskip
	
	\textbf{Acknowledgments}  Both authors were partially supported by GNSAGA-INdAM and PRIN 2020KKWT53 003 - Progetto \emph{Curves, Ricci flat Varieties and their Interactions}. We thank J. Commelin, B. van Geemen and R. Pignatelli for useful discussions and suggestions. 
\section{Preliminaries}\label{preliminaries} 
In this section, we collect the results on abelian covers that we will need in the sequel with a particular attention to bidouble covers. We present some constructions which will be used in the following and which are obtained modifying certain data in a bidouble cover. In particular, we consider singular bidouble covers and we introduce the notion of iterated bidouble covers.

\subsection{Bidouble Covers}\label{sec_BiDou}

Let $Y$ be a smooth complex algebraic surface. A \emph{bidouble cover} is a finite flat Galois morphism  $\pi\colon X \ra Y$ with Galois group $G:=(\ZZ/2\ZZ)^2$. 

Given such a cover, the group action on $\pi_*\mathcal{O}_X$ induces a splitting
\begin{equation*}\label{Eq_mainSplit}
	\pi_*\mathcal{O}_X=\mathcal{O}_Y \oplus \mathcal{L}^{-1}_{\chi_1} \oplus \mathcal{L}^{-1}_{\chi_2} \oplus \mathcal{L}^{-1}_{\chi_3} 
\end{equation*}

where $\mathcal{L}^{-1}_{\chi_i}$ denotes the eigensheaf on which $G$ acts via the non trivial character ${\chi_i}$ for $i=1,2,3$. 

We report here the structure Theorem for smooth Galois $(\ZZ/2\ZZ)^2$-cover given in \cite[Section 2]{CatM}.
\begin{theorem}[Proposition 2.3 \cite{CatM}] A bidouble cover $\pi\colon X \ra Y$ is uniquely determined by the data of effective divisors $D_1$, $D_2$, $D_3$ on $Y$ and divisors $L_1$, $L_2$, $L_3$ such that 
	\begin{equation}\label{eq_BiData}
		2L_k=(D_i+D_j) \textrm{ with } \{i,j,k\}=\{1,2,3\}.
	\end{equation}

If $Y$ and the effective divisors $D_i$ are	smooth and $D=\cup D_i$ is simple normal crossing, then $X$ is smooth. 
\end{theorem}	

Given a smooth complex algebraic surface $Y$ we shall refer to the set of effective divisors  $\{D_1, D_2, D_3 \}$ satisfying  \eqref{eq_BiData} as the \emph{data of a bidouble cover} $\pi$. Denote by $R$ the ramification locus of $\pi$, then  $\pi(R)=D$ is the branch locus and since we are assuming (for now) that $X$ and $Y$ are smooth, they are divisors by the purity of the branch locus. Let $\{\sigma_1, \sigma_2, \sigma_3\}=G\setminus \{0\}$ and  $R_i$ be the component of $R$ stabilized by its \emph{inertia subgroup} 
\[
H_{R_i}=\{h \in G\mbox{ such that } h|_{R_i}=1\}=\langle \sigma_i \rangle
\]
for $i=1,2,3$. We denote by $Y_i$ the quotient surface $X/\sigma_i$. This surface has some ordinary quadratic singularities at the isolated fixed points of $\sigma_i$, given by $R_j \cap R_k$ where $(i,j,k)$ is a permutation of $(1,2,3)$. 

We recall that a bidouble $X\ra Y$ is attached to the diagram \ref{eq: diagram bidouble cover}. The double cover $\pi_i\colon Y_i \longrightarrow Y$ is branched on $D_j \cup D_k$ and denoted  $\chi_i$ the character orthogonal to $\sigma_i$ we have
\[
(\pi_i)_*\mathcal{O}_{Y_i}=\mathcal{O}_Y \otimes \mathcal{L}^{-1}_{\chi_i}.
\] 

We can now give formulae for the numerical invariants following e.g., \cite[Proposition 4.2]{Par91a} 
\begin{equation}\label{eq_chiBiDoCo}
		\chi(X) =4\chi(Y)+\frac{1}{2}\sum_{i=1}^3 (L_i)^2+\frac{1}{2}\sum_{i=1}^3 L_iK_Y
		 =4\chi(Y)+\frac{1}{2}\sum_{i=1}^3 L_i(L_i+K_Y).
\end{equation}
Since $K_X=\pi^*K_Y+R$ and $2R=\pi^*D$ we have 
\begin{equation}\label{eq_K2BiDoCo}
	K^2_X=(2K_Y+D)^2=4K^2_Y+4DK_Y+D^2.
\end{equation}

Moreover, we have for the intermediate double coverings
\begin{equation}\label{equation: chi intermediate cover}
	\chi(Y_i)=2\chi(Y)+\frac{1}{2}L_i^2+\frac{1}{2}(L_iK_Y),
\end{equation}
which implies
$$\chi(X)=\chi(Y_1)+\chi(Y_2)+\chi(Y_3)-2\chi(Y).$$
To determine the properties of the intermediate double covers $Y_i$ we observe that the kth-plurigenus of these surfaces can be calculated using projection formula by
\[
\begin{split}
	\pi_*\omega_{Y_i}^{\otimes k}=\pi_*(\pi^*(\omega_Y\otimes \mathcal{L}_i)^{\otimes k}\otimes \mathcal{O}_{Y_i})= &(\omega_Y\otimes \mathcal{L}_i)^{\otimes k}\otimes\pi_*\mathcal{O}_{Y_i}= \\
	= (\omega_Y\otimes\mathcal{L}_i)^{\otimes k}\otimes(\mathcal{O}_{Y}\oplus\mathcal{L}_i^{-1})=&\big(\omega_Y\otimes\mathcal{L}_i\big)^{\otimes k}\oplus\left(\omega_Y^{\otimes k}\otimes\mathcal{L}_i^{\otimes (k-1)}\right),
\end{split}
\]
which implies \begin{equation}
	\label{eq: pulrigeberaYi}
	P_k(Y_i)=h^0\left(Y,\left(\omega_Y\otimes\mathcal{L}_i\right)^{\otimes k}\right)+h^0\left(Y,\left(\omega_Y^{\otimes k}\otimes\mathcal{L}_i^{\otimes (k-1)}\right)\right).\end{equation}

In the following we are mainly interested in the cases such that $Y$ is a rational surface and the intermediate surfaces $Y_i$ are either K3 surfaces or surfaces with $p_g=0$. To this end, we state the following easy lemma.
\begin{lemma}\label{lemma: IL LEMMA}
Let $Y$ be a rational surface and $X\ra Y$ be a smooth bidouble cover with ramification data $\{D_1,D_2,D_3\}$.
Then \begin{enumerate}
\item At most one of the divisor $D_i$'s is trivial;
	\item $Y_i$ is a K3 surface if and only if $L_i=-K_{Y}$; 
	\item $p_g(Y_i)=0$ if and only if $\dim(|K_Y+L_i|)=-1$ (here we pose $\dim(|D|)=-1$ if $|D|$ is empty);
	\item $\kappa(X)=0$ if and only if $2K_Y+D=0$.
\end{enumerate}\end{lemma}
\begin{proof} 
	$1)$ If $D_i=D_j=0$ then the cover of the rational surface $Y$ would not be connected.

	$2)$ Since $\pi_i:Y_i\ra Y$ is a double cover, $K_{Y_i}=\pi_i^*(K_Y+L_i)$ and $h^j(\mathcal{O}_{Y_i})=h^j(\mathcal{O}_{Y})+h^j(\mathcal{L}_i^{ -1})$. If $Y_i$ is a K3 surface, then its canonical bundle is trivial and then $L_i=-K_Y$. Vice versa, if $L_i=-K_Y$, then $K_{Y_i}$ is either trivial or a 2-torsion divisor. Moreover, $h^2(\mathcal{O}_{Y_i})=h^2(\mathcal{O}_{Y})+h^2(\mathcal{L}_i^{ -1})=0+h^0(\mathcal{O}_Y)=1$, so $K_{Y_i}$ is trivial. By \eqref{equation: chi intermediate cover}, one has that $\chi(Y_i)=2$, so $Y_i$ is a K3 surface.
	
	$3)$ It follows simply by \eqref{eq: pulrigeberaYi}, indeed $p_g(Y_i)=P_1(Y_i)=h^0(\omega_{Y_i})=h^0(\omega_{Y}\otimes \mathcal{L}_i)+h^0(\omega_{Y})$.
	
	$4)$ We recall that $\kappa(X)=0$ if and only if $K_X$ is a torsion divisor (possibly zero) which is equivalent to say that $2K_X=\pi^*(2K_Y+D)$ is a torsion divisor (possibly zero). If $2K_Y+D$ is a torsion divisor, then $\kappa(X)=0$. Viceversa, if $\kappa(X)=0$ then $2K_Y+D$ is a torsion divisor. Since $Y$ is a rational surface, $2K_Y+D$ must be trivial.
\end{proof}

\begin{rem}{\rm The point $(4)$ of the previous Lemma \ref{lemma: IL LEMMA} characterizes the bidouble covers of a rational surfaces whose Kodaira dimension is 0. We observe that all the classes of surfaces with trivial Kodaira dimension can be obtained in this way. In the following we will show that: the bidouble cover of $\mathbb{P}^2$ branched on a quartic and a conic is a K3 surface (cf. Theorem \ref{theorem: bidouble P2}, case c)); the bidouble cover of $\mathbb{P}^2$ branched on three conics is an Enriques surface (cf. $Z_1$ in Theorem \ref{theor: iterated bidouble GS2}); the bidouble cover of $\mathbb{P}^1\times\mathbb{P}^1$ with data $D_1=4h_1$, $D_2=4h_2$, $D_3=0$ is an Abelian surface (cf. Theorem \ref{theor: Y=P1xP1} case p)).
An example of $X$ being a bielliptic surface is given by the bidouble cover of $\mathbb{P}^1\times\mathbb{P}^1$ with data $D_1=4h_1$, $D_2=2h_2$, $D_3=2h_2$.}
\end{rem}


\subsection{Iterated bidouble cover}\label{subsec: Iterated general case}

Let $X\ra Y$ be a smooth bidouble cover and $p_i:X\ra Y_i$ one of the intermediate double cover as in diagram \eqref{eq: diagram bidouble cover}. Let us denote by $B$ its branch locus (it could be either a possibly reducible curve or the union of a curve and some singular points of $Y_i$, which are necessarily of type $A_1$). Without loss of generalities we assume $i=1$.

An iterated bidouble cover of $Y$ is a bidouble cover $W\ra Y_1$ such that $p_1:X\ra Y_1$ is one of its intermediate double cover. 

To an iterated bidouble cover is associated  the following diagram:
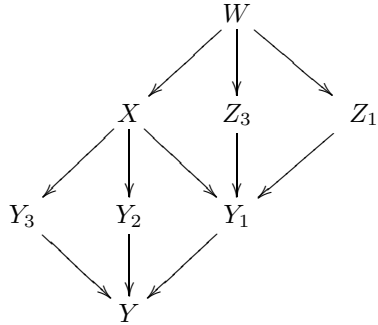
\begin{equation}
	\label{eq: diag bibidouble} 
	\xymatrix{&&W\ar[dr]\ar[dl]\ar[d]\\&X\ar[dr]\ar[dl]\ar[d]&Z_3\ar[d]&\ \ Z_1\ar[dl]\\
		Y_3\ar[dr]&Y_2\ar[d]&Y_1\ar[dl]\\
		&Y}
\end{equation}

where all the arrows in the diagram are double covers. We may refer to the bidouble cover $X\ra Y$ as the \emph{first bidouble cover}.

Since it is possible, and indeed it often happens, that $Y_1$ is singular, one often has to blow up $Y_1$ in order to obtain a smooth surface $\widetilde{Y_1}$. To this surface one can apply the theory of the bidouble covers described in Section \ref{sec_BiDou}. In particular, the double cover $p_1:X\ra Y_1$ induces a double cover between the smooth surfaces $\widetilde{X}$ and $\widetilde{Y_1}$. By the purity of the branch locus, this double cover is branched on curves, in particular on $\widetilde{B}$, which is the strict transform of $B$ on $Y_1$, and on the exceptional divisors obtained by the blow up of the singular branch point of $p_1:X\ra Y_1$. Let us denote $\widetilde{B}+E$ the branch locus of $p_1:\widetilde{X}\ra \widetilde{Y_1}$.

So if $Y_1$ is singular, the diagram attached to the iterated bidouble cover is 
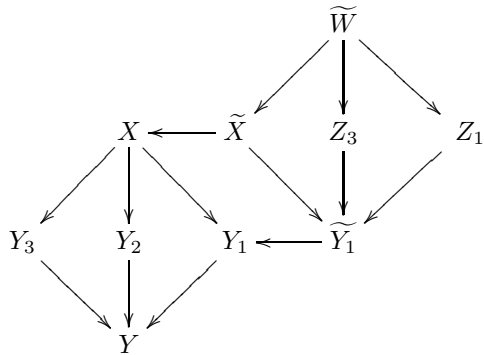
\begin{align*}\label{eq: diag bibidouble smooth}
	\xymatrix{&&&\widetilde{W}\ar[dr]\ar[dl]\ar[d]\\&X\ar[dl]\ar[d]\ar[dr]&\widetilde{X}\ar[l]\ar[dr]&Z_3\ar[d]&\ \ Z_1\ar[dl]\\
		Y_3\ar[dr]&Y_2\ar[d]&Y_1\ar[dl]&\widetilde{Y_1}\ar[l]\\
		&Y}\end{align*}
where the horizontal arrows are blow ups.

The existence of an iterated bidouble cover is  equivalent to the existence of three effective divisors $\Delta_i$, $i=1,2,3$ on $\widetilde{Y}_i$ such that \begin{itemize}
	\item $\Delta_i+\Delta_j=2\Lambda_k\in Pic(\widetilde{Y_i})$, where $(i,j,k)$ is a permutation of $(1,2,3)$
	\item $\Delta_1+\Delta_3=\widetilde{B}+E$.
\end{itemize}


\begin{rem}{\rm  By diagram \eqref{eq: diag bibidouble} one observes that $W\ra Y$ is a {\it simple iterated double cover} in the sense of \cite{Ma}, indeed $W\ra X\ra Y_1\ra Y$ is a composition of double covers.

}	\end{rem}

\subsection{The constructions}
In the following we describe the settings of the different constructions we are employing in the sequel. All of them are obtained by considering bidouble covers, which produce many different surfaces. 

\subsubsection{Construction of the bidouble cover}\label{subsubsec: construction 1} We will consider a smooth (possibly minimal) rational surface $Y$ and we will look for smooth bidouble covers of $Y$ such that the intermediate double covers $Y_i\ra Y$ %
have $p_g(Y_i)\leq 1$ and at least one of them is a K3 surface, say $Y_3$. 
So we define the following construction.
\begin{construction}\label{construction 1: first double cover}{\rm 
Consider $X\ra Y$ a bidouble cover of a rational surface $Y$ with data $\{D_1,D_2,D_3\}$ such that
\begin{itemize}
	\item $D_i$ are smooth and $D_1\cup D_2\cup D_3$ is simple normal crossing;
	\item $D_i+D_j=2L_k\in Pic(Y)$, where $(i,j,k)$ is a permutation of $(1,2,3)$;
	\item $D_1+D_2+2K_Y=0$, which assures that $Y_3$ is a K3 surface;
	\item $\dim(|K_Y+(D_i+D_3)/2|)\leq 0$, $i=1,2$, which is equivalent to have $p_g(Y_i)\leq 1$ for $i=1,2$.
\end{itemize}}
\end{construction}

We point out that this is also the first bidouble cover of an iterated bidouble cover. 

\subsubsection{Construction of the iterated bidouble cover}\label{subsubsec: construction 2} If at least one among $Y_1$ and $Y_2$ is a rational surface, we can consider an iterated bidouble cover, where the commune double cover is either $X\ra Y_1$ or $X\ra Y_2$. 

\begin{rem}\label{rem: no exceptional in Delta1,2}
{\rm We observe that the construction of an iterated bidouble cover, implies a specialization of the curves $D_1$, $D_2$, $D_3$ which are the data of the first bidouble cover. We will allow only the specializations which do not introduce singularities worst than simple normal crossing in the configuration of the curves $D_1\cup D_2\cup D_3$. This guarantees that the invariants computed for the first bidouble cover remain the same also under the specialization we are considering, see Section \ref{subsubsec: singular bidouble}. In particular, we do not allow triple points in $D_1\cup D_2\cup D_3$ and  this forces the divisors $\Delta_1$ and $\Delta_2$ (see Section \ref{subsec: Iterated general case}) to have no exceptional divisors $E_j$ as components.}\end{rem}

There are two cases of interest for us: \begin{enumerate}
	\item at least one among $Z_1$ and $Z_3$ is a K3 surface and the other one is either a K3 surface or a surface with $p_g=0$;
	\item one among $Z_1$ and $Z_3$ has the same construction of $X$ (i.e. it is obtained as bidouble cover of $Y$ with branch divisors linearly equivalent to the ones attached to $X$) and the other one is either a K3 surface or a surface with $p_g=0$.
\end{enumerate}

A priori one could be interested also in the case where both  $Z_1$ and $Z_2$ have the same construction of $X$, but we never obtain it if we require that $\Delta_1$ and $\Delta_2$ do not contain exceptional divisors as we assumed in the Remark \ref{rem: no exceptional in Delta1,2}.

Determining an iterated bidouble cover $W\ra \widetilde{Y_1}$ as in (1) or (2) is equivalent to find three effective divisors $\Delta_i$, $i=1,2,3$ on $\widetilde{Y}_i$ which satisfy certain properties.
We now write explicitly the conditions attached to a bidouble cover as in (1) and (2) respectively.

\begin{construction}\label{construction 2: iterated bidouble case 1}{\rm 
		Consider $\widetilde{W}\ra \widetilde{Y_1}$ the iterated bidouble cover with data $\{\Delta_1,\Delta_2,\Delta_3\}$ such that
		\begin{itemize}
	\item $\Delta_i+\Delta_j=2\Lambda_k\in Pic(\widetilde{Y_i})$, where $(i,j,k)$ is a permutation of $(1,2,3)$;
	\item $\Delta_1+\Delta_2+2K_{\widetilde{Y_1}}=0$, which implies that $Z_3$ is a K3 surface;
	\item $\Delta_1+\Delta_3=\widetilde{B}+E$;
	\item $\dim|K_{\widetilde{Y_1}}+(\Delta_2+\Delta_3)/2|\leq 0$;
	\item there are no components of $E$ which are components of $\Delta_1$ or $\Delta_2$ (by Remark \ref{rem: no exceptional in Delta1,2}).
\end{itemize}}
\end{construction}

\begin{construction}\label{construction 2: iterated bidouble case 2}{\rm 
		Consider $\widetilde{W}\ra \widetilde{Y_1}$ the iterated bidouble cover with data $\{\Delta_1,\Delta_2,\Delta_3\}$ such that
		\begin{itemize}
			\item $\Delta_i+\Delta_j=2\Lambda_k\in Pic(\widetilde{Y_i})$, where $(i,j,k)$ is a permutation of $(1,2,3)$;
			\item $\Delta_1+\Delta_3=\widetilde{B}+E$;
			\item $\Delta_2+\Delta_3\simeq\widetilde{B}+E$;
			\item $\dim|K_{\widetilde{Y_1}}+(\Delta_1+\Delta_2)/2|\leq 0$;
			\item there are no components of $E$ which are components of $\Delta_1$ or $\Delta_2$ (by Remark \ref{rem: no exceptional in Delta1,2}).
		\end{itemize}}
	\end{construction}
We observe that the conditions of Construction \ref{construction 2: iterated bidouble case 2} implies that $\Delta_1\simeq\Delta_2$ and the fourth condition is equivalent to $\dim|K_{\widetilde{Y_1}}+\Delta_1|\leq 0$.

\subsubsection{Singular bidouble cover}\label{subsubsec: singular bidouble}
The singular bidouble covers of smooth surfaces are studied intensively in \cite{Cat99}, where Catanese applied them to construct new surfaces of general type with prescribed invariants. In particular, he considered the case when $D_1\cup D_2\cup D_3$ is not simple normal crossing. 
To each point $P$ in the branch locus one attaches a triple $(\mu_1(P),\mu_2(P),\mu_3(P))$ where $\mu_i(P)$ is the multiplicity of $D_i$ in $P$.

Assuming that none of the curves $D_i$ contains a triple point, if $D_1\cup D_2\cup D_3$ has a triple point $P$, then $(\mu_1(P),\mu_2(P),\mu_3(P))$ is either $(2,1,0)$ or $(1,1,1)$. By \cite[Section 3]{Cat99}, if the bidouble cover surface $X$ is of general type, in the first case the self intersection of the canonical divisor of the minimal model of $X$ can be computed as in the simple normal crossing case, in the latter it decreases by 1. Therefore in the latter case, one is constructing a surface of general type with invariants which are not the same as the ones of the simple normal crossing case, i.e. one is producing a surface in a different family.

\section{General results on smooth bidouble covers, singular bidouble covers and iterated bidouble covers}\label{Sec:General}
 We discuss classical problems and conjectures for the surface $X$ obtained as smooth bidouble cover of a rational surface $Y$. More precisely, we first discuss the Infinitesimal Torelli Property. Then, we prove several results on the decomposition of the weight 2 Hodge structures, of the Chow groups of zero--cycle of degree 0 and of the transcendental motives of the surfaces $X$ (see Propositions \ref{prop: Hodge structure}, \ref{prop: chow groups}, \ref{prop: motives} respectively). We also extend some of these results to the surfaces obtained by the other constructions described in the previous sections (see Sections \ref{subsubsec: construction 2}, \ref{subsubsec: singular bidouble}), i. e. the iterated bidouble covers and certain singular bidouble covers.
 At the end of the Section, we prove that, if $X$ is a bidouble cover of a rational surface $Y$ such that $p_g(X)=2$ or $3$, under certain conditions on the intermediate double covers, the Mumford--Tate and the Tate conjectures hold.

 \subsection{Infinitesimal Torelli Property for bidouble coverings}\label{subsec: Infinitesimal Torelli}
 In this section we use the techniques develop in the papers \cite{Par91a}, \cite{Par91b} and \cite{PZ19} to prove Theorem \ref{theo: Infinitesimal Torelli} which guarantees that the Infinitesimal Torelli Property holds for bidouble covers under certain conditions. We will give examples of surfaces which satisfy this hypothesis in the following sections.
 
 Recall the following definition.
 \begin{definition} Let $D_1, \ldots ,D_k$ be divisors in a smooth manifold X and $x_1, \ldots,x_k$ equations for them.
 	Define $\Omega^1_X(\log (D_1),\ldots, \log(D_k))$ to be the subsheaf (as $\mathcal{O}_X$-module) of $\Omega^1_X(D_1+ \ldots +D_k)$ generated by $\Omega^1_X$ and by $\frac{dxj}{x_j}$ for $j=1, \ldots k$.
 \end{definition}
 Notice that if $\Delta$ is the union of the components $D_1,\ldots, D_k$, than by  $\Omega^1_X(\log \Delta)$ we mean $\Omega^1_X(\log (D_1),\ldots, \log(D_k))$, as in \cite{Par91b}.
 
 \begin{remark}{\rm  The contraction tensor $T_X(-\log D) \otimes \Omega_{X}^1(\log D) \ra \mathcal{O}_X$ is a non degenerate pairing, therefore there is a canonical duality between $T_X(-\log D)$ and $\Omega_{X}^1(\log D)$.}
 \end{remark}
 We introduce the following notation for a bidouble cover data $\{D_1, D_2, D_3\}$ as in \cite{Par91a}. First identify
 \[
 D_{\chi_i}=D_i \textrm{ for }\, i=1,2,3, \textrm{ and } D=D_1+D_2+D_3=D_{\chi_0},
 \] 
 where $\chi_0=1$.
 
 Let $\mathcal{L}_{\chi_0}=\mathcal{L}_{\chi^{-1}_0}=\mathcal{O}_{Y}$ and $\mathcal{L}_{\chi_i}=\frac{D_j+D_k}{2}$ with $(i,j,k)$ a permutation of $(1,2,3)$; we have $\mathcal{L}_{\chi_i}=\mathcal{L}_{\chi_i^{-1}}$.
 Moreover, let
 \[
 D_{\chi,\phi}=L_{\chi}+L_{\phi}-L_{\chi\phi}\]
 so that 
 \[
 D_{\chi_0,\chi_0^{-1}}=\emptyset, \quad D_{\chi_1,\chi_1^{-1}}=D_2+D_3, \quad D_{\chi_2,\chi_2^{-1}}=D_1+D_3, \quad D_{\chi_3,\chi_3^{-1}}=D_1+D_2,
 \] 
 
 If $(i,j,k)$ is a permutation of $(1,2,3)$, then $D_{\chi_i,\chi_j}=D_{\chi_k}$, indeed  
 \[
 D_{\chi_i,\chi_j}=L_{\chi_i}+L_{\chi_j}-L_{\chi_k}=L_i+L_j-L_k=\frac{D_j+D_k}{2}+\frac{D_i+D_k}{2}-\frac{D_i+D_j}{2}=D_k=D_{\chi_k}.
 \]
 The following result is a special case of \cite[Proposition 4.1]{Par91a}.
 \begin{proposition}\label{prop_InvaTang}  Let $\pi \colon X \longrightarrow Y$ be a bidouble cover with bidouble cover data $\{D_1, D_2, D_3\}$, then for every character $\chi_i$:
 	\begin{enumerate}
 		\item $\pi_*(T_X)^{\chi_i}=T_Y(-\log D_{\chi_i})\otimes \mathcal{L}^{-1}_{\chi_i}$;
 		\item $\pi_*(\Omega_X)^{\chi_i}=\Omega_Y^1(\log D_{\chi_i,\chi_i^{-1}})\otimes \mathcal{L}^{-1}_{\chi_i}$;
 		\item $\pi_*(\omega_X)^{\chi_i}=\omega_Y\otimes \mathcal{L}_{\chi_i^{-1}}$.
 	\end{enumerate}
 \end{proposition}

 We recall the following.
 \begin{definition} Let $X$ be a smooth complex surface. We say that $X$ has the \emph{Infinitesimal Torelli Property} for the periods of $2$-forms if and only if the natural map
 	\begin{equation}\label{eq_IT}
 	\rho\colon H^1(X, T_X) \longrightarrow \textrm{Hom}\big(H^0(X, \omega_X), H^1(X, \Omega^1_X) \big),
 	\end{equation}
 	given by the cup product, is injective.
 \end{definition}

 If we consider a bidouble cover $\pi \colon X \longrightarrow Y$ we can address the question when $X$ has the Infinitesimal Torelli Property in the following way. The idea is to decompose the Infinitesimal Torelli map \eqref{eq_IT} using the Galois group action.
 
 By Proposition \ref{prop_InvaTang} we have the splitting:
 {\small{
 \begin{equation*}\label{eq_Tang} H^1(X, T_X)=\bigoplus_{\chi \in G^*} H^1(Y, T_Y(-\log D_{\chi})\otimes \mathcal{L}^{-1}_{\chi})=H^1(Y,T_Y(-\log D))\oplus \bigoplus_{\chi \in G^*\setminus \{1\}} H^1(Y, T_Y(-\log D_{\chi})\otimes \mathcal{L}^{-1}_{\chi}).
 \end{equation*}
 
 \begin{equation*}\label{eq_CoTang} H^1(X, \Omega^1_X)=\bigoplus_{\chi \in G^*} H^1(Y, \Omega^1_Y(\log D_{\chi,\chi^{-1}})\otimes \mathcal{L}^{-1}_{\chi})=H^1(Y,\Omega^1_Y)\oplus \bigoplus_{\chi \in G^*\setminus \{1\}} H^1(Y, \Omega^1_Y(\log D_{\chi,\chi^{-1}})\otimes \mathcal{L}^{-1}_{\chi}).
 \end{equation*}
 
 \begin{equation*}\label{eq_Can} H^0(X, \omega_X)=\bigoplus_{\chi \in G^*} H^0(Y,\omega_Y \otimes \mathcal{L}_{\chi^{-1}})=H^0(Y,\omega_Y) \oplus \bigoplus_{\chi \in G^*\setminus \{1\}} H^0(Y,\omega_Y \otimes \mathcal{L}_{\chi^{-1}}).
 \end{equation*}
}}
 Since the cup product of the Infinitesimal Torelli map \eqref{eq_IT} is compatible with the group action, for every characters $\chi, \phi \in G^*$ we have an induced map
 \begin{equation*}\label{eq_rho}
 \rho_{\chi,\phi}\colon H^1(Y, \T_Y(-\log D_{\chi})\otimes \mathcal{L}^{-1}_{\chi}) \longrightarrow \textrm{Hom} \, (H^0(Y, \omega_Y \otimes \mathcal{L}_{\phi^{-1}}), H^1(Y, \Omega^1_Y(\log D_{\chi\phi,(\chi\phi)^{-1}})\otimes \mathcal{L}^{-1}_{\chi\phi})).
 \end{equation*}
 Then it holds:
 \[
 \rho = \bigoplus_{\chi,\phi \in G^*} \rho_{\chi,\phi}.
 \]
 To check if $\rho$ is injective (i.e. to prove the Infinitesimal Torelli Property) we use a criterion for the abelian cover which is given in \cite{Par91b}.
 
 \begin{proposition}[Proposition 2.1\cite{Par91b}]\label{prop_InfTorelliCrit} 
 	Let $Y$, $X$ be smooth surfaces and let $\pi\colon X \longrightarrow Y$ be a bidouble cover, with group $G$. Then the Infinitesimal Torelli Property for $X$ holds if and only if for every $\chi \in G^*$ we have
 	\[
 	\bigcap_{\phi \in G^*} \ker \rho_{\chi, \phi}=\{0\}.
 	\]
 \end{proposition}
 
 Let us take a closer look at the maps $\rho_{\chi,\phi}$. We would like to factorize them into two simpler maps. To this end, for every pairs of characters $\chi, \phi \in  G^*$ and every section $z \in  H^0(Y, \omega_{Y} \otimes \mathcal{L}_{\phi^{-1}})$, consider the diagram
 
 \begin{equation}\label{Diagram_Pardini}
 \xymatrix{ 
 	T_{Y}(-\log D_{\chi}) \otimes \mathcal{L}^{-1}_{\chi} \ar[rr] \ar[dd] &&  \Omega^1_Y(\log D_{\chi\phi,(\chi\phi)^{-1}})\otimes \omega^{-1}_{Y}\otimes (\mathcal{L}_{\chi\phi}\otimes \mathcal{L}_{\phi^{-1}})^{-1}\ar[dd]
 	\\ \\
 	T_{Y}(-\log D_{\chi})\otimes \mathcal{L}_{\chi}^{-1} \otimes \omega_{Y}\otimes \mathcal{L}_{\phi^{-1}} \ar[rr] & &   \Omega^1_{Y}(\log D_{\chi\phi,(\chi\phi)^{-1}})\otimes \mathcal{L}^{-1}_{\chi\phi}
 }
 \end{equation}
 where the vertical maps are given by multiplication by $z$ and the horizontal maps are defined by contraction of tensors and by the fundamental relations of the building data.
 Recall that there is a natural isomorphism for rank $2$ vector bundle $E\cong E^{\vee} \otimes \det(E)$,  which gives $T_{Y}(-\log(D_{\chi\phi,(\chi\phi)^{-1}}))\simeq \Omega^1_Y(\log(D_{\chi\phi,(\chi\phi)^{-1}}))\otimes \omega_{Y}^{-1}\otimes \mathcal{O}_Y(-D_{\chi\phi,(\chi\phi)^{-1}})$. This yields
 \begin{equation}\label{eq_EED}
 	\Omega^1_Y(\log(D_{\chi\phi,(\chi\phi)^{-1}}))\otimes \omega_{Y}^{-1}\otimes\big(\mathcal{L}_{\chi\phi}\otimes \mathcal{L}_{\phi}\big)^{-1}\cong T_{Y}(-\log(D_{\chi\phi,(\chi\phi)^{-1}}))\otimes\mathcal{L}^{-1}_{\chi}(D_{\chi\phi,(\chi\phi)^{-1}}-D_{\chi\phi,\phi}).
 \end{equation}
 This implies that the horizontal arrows in the previous diagram are natural inclusions, see \cite{Par91b}.

The previous diagram allow us to define the maps
 \begin{equation}\label{eq_q}
 q_{\chi,\phi}\colon H^1(Y, \, T_{Y}(-\log D_{\chi}) \otimes \mathcal{L}^{-1}_{\chi}) \longrightarrow H^1(Y, \, \Omega^1_Y(\log D_{\chi\phi,(\chi\phi)^{-1}})\otimes \omega^{-1}_{Y}\otimes (\mathcal{L}_{\chi\phi}\otimes \mathcal{L}_{\phi^{-1}})^{-1})
 \end{equation}
 and
 \begin{equation*}\label{eq_r}\begin{split}
 r_{\chi,\phi}\colon H^1(Y, \, \Omega^1_Y(\log D_{\chi,\chi^{-1}})\otimes &\,(\omega_{Y}\otimes \mathcal{L}_{\chi}\otimes \mathcal{L}_{\phi^{-1}})^{-1})  \longrightarrow  \\
 &  \textrm{Hom} \, \big(H^0(Y, \omega_{Y} \otimes \mathcal{L}_{\phi^{-1}}), H^1(Y, \Omega^1_{Y}(\log D_{\chi,\chi^{-1}})\otimes \mathcal{L}^{-1}_{\chi})\big).
 \end{split}
 \end{equation*}

 By construction, we have $$\rho_{\chi,\phi} =r_{\chi\phi,\phi} \circ q_{\chi,\phi}.$$ 
 
 A way to apply the Proposition \ref{prop_InfTorelliCrit}  is to study more in details the single maps $r_{\chi,\phi}$ and $q_{\chi,\phi}$.

 Here we collect some results which can be applied to this purpose.

 By the Bott's vanishing theorem, see \cite[Theorem 2.4]{Must02}: for every toric surface $Y$ and ample line bundle $\mathcal{L}$ we
 have
 \[
 H^q(Y, \Omega^p_{Y}\otimes \mathcal{L}) = 0,\ \ \ q>0.\]
 We observe that the rational surfaces are toric.
 
 To study the maps $r_{\chi,\phi}$ we define the following spaces: 
 
 \[
 H^{dom}(r,\chi,\phi):=H^1\big(Y, \, \Omega^1_Y(\log D_{\chi,\chi^{-1}})\otimes \,(\omega_Y\otimes \mathcal{L}_{\chi}\otimes \mathcal{L}_{\phi^{-1}})^{-1}\big).
 \]
 
 \[H^0_{\phi}:=H^0(Y, \omega_{Y}\otimes \mathcal{L}_{\phi^{-1}}),\ \ H^1_{\chi}:=H^1(Y, \Omega_{Y}^1(\log D_{\chi,\chi^{-1}})\otimes \mathcal{L}^{-1}_{\chi}).\]
 
 \begin{lemma}\label{lemma: rchi0chi0 injective}
 	Let $Y$ be a minimal rational surface then $H^{dom}(r,\chi_0,\chi_0)=0$, $H^0_{\chi_0}=0$. 
 	In particular, $r_{\chi,\chi_0}$ is the zero map for every $\chi$ and it is injective for $\chi=\chi_0$.
 \end{lemma}
 \proof Since $D_{\chi_0,\chi_0^{-1}}=\emptyset$ and $\mathcal{L}_{\chi_0}=\mathcal{L}_{\chi_0^{-1}}=\mathcal{O}_Y$ one obtains
 $H^{dom}(r,\chi_0,\chi_0)=H^1(Y,\Omega^1_Y\otimes \omega_{Y}^{-1})=0,$
 where the last equality follows by Bott's vanishing theorem. If $\phi=\chi_0$,  $H^0_{\chi_0}=H^0(Y,\omega_Y)=0$ and $r_{\chi,\chi_0}=H^{dom}(r,\chi,\chi_0)\ra Hom(0,H^1_{\chi})$.\endproof

 In \cite{Par98} a criterion for the injectivity of the maps $r_{\chi,\phi}$ is given by introducing the  prolongation bundle of the irreducible components of $D_{\chi,\chi^{-1}}$. The criterion is that $r_{\chi,\phi}$ is injective if the following multiplication map is surjective:
 \begin{equation*}\label{eq_Prolungation}\begin{split}
 H^0(Y, \omega_{Y}\otimes \mathcal{L}_{\phi^{-1}}) \otimes ( \mathop{\bigoplus} \limits_{
 	\substack{
 		B \, \text{irreducible} \\
 		\text{components of} \, D_{\chi,\chi^{-1}}}}  
 H^0(Y, \mathcal{O}_{Y}(B)\otimes &\omega_{Y}\otimes \mathcal{L}_{\chi})\big) \longrightarrow \\
 &
 \big( \mathop{\bigoplus} \limits_{
 	\substack{
 		B \, \text{irreducible} \\
 		\text{components of} \, D_{\chi,\chi^{-1}}}}   H^0(Y, \mathcal{O}_{Y}(B)\otimes \omega^2_{Y}\otimes\mathcal{L}_{\chi} \otimes \mathcal{L}_{\phi^{-1}})).
 \end{split}
 \end{equation*}
 
 \begin{prop}\label{prop: rchichii injective for K3}
 	Let $Y$ be a minimal rational surface and $X\ra Y$ a bidouble cover with intermediate double cover $Y_i$, $i=1,2,3$. For each $i$ such that $Y_i$ is a K3 surface, $r_{\chi,\chi_i}$ is injective for every $\chi$.\end{prop}
 \proof Since $Y_i$ is a K3 surface, $\mathcal{L}_{\chi_i}=\omega_Y^{-1}$, so the previous multiplicative map begins
 \begin{equation*}
 \mathbb{C} \otimes ( \mathop{\bigoplus} \limits_{
 	\substack{
 		B \, \text{irreducible} \\
 		\text{components of} \, D_{\chi,\chi^{-1}}}}  
 H^0(Y, \mathcal{O}_{Y}(B)\otimes \omega_{Y}\otimes \mathcal{L}_{\chi})\big) \longrightarrow 
 \big( \mathop{\bigoplus} \limits_{
 	\substack{
 		B \, \text{irreducible} \\
 		\text{components of} \, D_{\chi,\chi^{-1}}}}   H^0(Y, \mathcal{O}_{Y}(B)\otimes \omega_{Y}\otimes\mathcal{L}_{\chi})),
 \end{equation*}
 which is clearly surjective. By \cite{Par98}, $r_{\chi,\chi_i}$ is injective. \endproof

 Let us now consider some properties of the maps $q_{\chi,\phi}$.  By substituting \eqref{eq_EED} in \eqref{eq_q} one obtains that the map $q_{\chi,\phi}$ is
$$q_{\chi,\phi}\colon H^1(Y, \, T_{Y}(-\log D_{\chi}) \otimes \mathcal{L}^{-1}_{\chi}) \longrightarrow H^1(Y, \, T_{Y}(-\log(D_{\chi\phi,(\chi\phi)^{-1}}))\otimes\mathcal{L}^{-1}_{\chi}(D_{\chi\phi,(\chi\phi)^{-1}}-D_{\chi\phi,\phi}))$$

 In particular, we first consider the case $\chi=\chi_0$:
 $$q_{\chi_0,\phi}\colon H^1(Y, \, T_{Y}(-\log (D)) ) \longrightarrow H^1(Y, T_Y(-\log(D_{\phi,\phi^{-1}}))\otimes\mathcal{O}_Y(D_{\phi,\phi^{-1}})).$$
 If $\phi=\chi_i$, $i=1,2,3$, then $D_{\phi,\phi^{-1}}=D_j+D_k$, where $(i,j,k)$ is a permutation of $(1,2,3)$ and then $D_{\phi,\phi^{-1}}<D=D_1+D_2+D_3$. 
 Comparing the maps $q_{\chi_0,\phi}$ with the map $\rho_i$ defined in
 \cite[Lemma 3.1]{Par91b} one realizes that they coincide (with $L$ being trivial, where $L$ is as in \cite[Lemma 3.1]{Par91b}). To fix the notation, let us consider the two maps $q_{\chi_0,\chi_2}$ and $q_{\chi_0,\chi_3}$. The divisors $D_0$ and $D_1$ of the \cite[Lemma 3.1]{Par91b} are, in the present notation, $D_{\chi_2,\chi_2^{-1}}$ and $D_{\chi_3,\chi_3^{-1}}$, so that $\Delta_1=\Delta_0=D_{\chi_1}$. This allows to restate the \cite[Lemma 3.1]{Par91b}
 \begin{lemma}\label{lemma: intersection ker}
 	If $K_Y+D_{\chi_1}$ is ample, then 
 	$$\ker q_{\chi_0,\chi_2}\cap \ker q_{\chi_0,\chi_3}=\ker\left(\rho^{\vee}:=H^1(Y,T_{Y}(-\log D))\ra H^1(Y,T_{Y}(-\log D_{\chi_1}))\right).$$
 \end{lemma}
 \proof If $K_Y+D_{\chi_1}$ is ample, then $H^2(Y,\Omega^1_{Y}(K_Y+D_{\chi_1}))=0$ by Bott's vanishing theorem. By Serre duality  $H^0(Y, T_{Y}(-\Delta_0))=H^0(Y, T_{Y}(-D_{\chi_1}))\simeq H^2(Y,\Omega^1_{Y}(K_Y+D_{\chi_1}))=0$. Hence the hypothesis of \cite[Lemma 3.1]{Par91b} are satisfied, so the result follows. \endproof

 To determine $\ker\left(H^1(Y,T_{Y}(-\log D))\ra H^1(Y,T_{Y}(-\log D_{\chi_1}))\right)$ we need some short exact sequences, which we are going to prove for the reader convenience.
 
 \begin{prop}\label{prop: exact sequences log}
 	
 	Let $D_i$ be three smooth reduced divisors on a manifold $Y$. 
 	Let $\cup_{i}D_i$ be simple normal crossing. Then the following short sequence is exact:
 	$$0\ra \Omega_Y^1(\log D_1)\ra\Omega_Y^1(\log D_1,\log D_2,\log D_3)\ra \mathcal{O}_{D_2}\oplus\mathcal{O}_{D_3}\ra 0.$$
 \end{prop}
 \begin{proof} By \cite[Proposition 2.11]{CatM} the horizontal shorts sequences in the diagram are exact
 	\begin{equation}\label{eq: diagram sequences}\xymatrix{ &&0\ar[d]&0\ar[d]\\
 		0\ar[r]&\Omega_{Y}^1\ar[r]\ar[d]^{=}&\Omega_{Y}^1(\log D_1)\ar[r]\ar[d]^{j}&\mathcal{O}_{D_1}\ar[d]\ar[r]&0\\
 		0\ar[r]&\Omega_{Y}^1\ar[r]&\Omega_{Y}^1(\log D_1, \log D_2, \log D_3)\ar[r]\ar[d]&\bigoplus_{i=1}^3\mathcal{O}_{D_i}\ar[d]\ar[r]&0\\
 		&&\mathcal{O}_{D_2}\oplus \mathcal{O}_{D_3}\ar[r]^{=}\ar[d]&\mathcal{O}_{D_2}\oplus \mathcal{O}_{D_3}\ar[d]\\
 		&&0&0}\end{equation}
 	The last vertical sequence is trivially exact, the map $j$ is the natural sheaf inclusion and by diagram chasing we obtain the other vertical exact sequences.
 \end{proof}
 
 In \cite[Lemma 3.7]{CatM} it is proved that the image of the map $c_{\oplus}:\oplus H^0(D_i,\mathcal{O}_{D_i})\ra H^1(Y,\Omega_Y^1)$ is generated by the Chern classes of the $D_i$'s, in particular $\dim({\rm Im} c_{\oplus})\geq 1$.

 \begin{corollary}\label{corollary: D2D3 rational rhovee isomorphism}
 	Let $Y$ be a minimal rational surface, $D_1$, $D_2$, $D_3$ are connected smooth curves with genus $g(D_i)$. If $\dim(Im(c_{\oplus}))=1$ and $D_2$ and $D_3$ are rational, then the map $$\rho^{\vee}=H^1(Y,\Omega_Y^1(\log(D_i)))\ra H^1(Y,\Omega_Y^1(\log(D_i),\log(D_j),\log(D_k)))$$ is an isomorphism.
 \end{corollary}
 \proof
 By the long exact sequence induced in cohomology by the first vertical sequence in \eqref{eq: diagram sequences}, it follows that the map is surjective if $D_2$ and $D_3$ are rational. 
 By using the first horizontal exact sequence \eqref{eq: diagram sequences} we deduce 
 $$\xymatrix{&0\ar[r]&H^0(Y,\Omega_Y^1(\log (D_1)))\ar[r]&H^0(D_1,\mathcal{O}_{D_1})\ar[r]^c&\\
 	\ar[r]^c&H^1(Y,\Omega_Y^1)\ar[r]&H^1(Y,\Omega_Y^1(\log (D_1)))\ar[r]&H^1(D_1,\mathcal{O}_{D_1})\ar[r]&0}$$
 which is 
 $$\xymatrix{&0\ar[r]&H^0(Y,\Omega_Y^1(\log (D_1)))\ar[r]&\mathbb{C}\ar[r]^c&
 	\mathbb{C}^{\rho(Y)}\ar[r]&H^1(Y,\Omega_Y^1(\log (D_1)))\ar[r]&\mathbb{C}^{g(D_1)}\ar[r]&0},$$
 where $\rho(Y)$ is the Picard number of $Y$.

 The map $c$ has as image the subspace generated by the first Chern class of the divisor $D$ by \cite[Lemma 3.7]{CatM}, so $\dim(Im(c))=1$, which implies $H^0(Y,\Omega_Y^1(\log (D_1)))=0$. In follows that 
 $$h^1(Y,\Omega_Y^1(\log (D_1)))=\rho(Y)+g(D_1)-1.$$

 By the second horizontal exact sequence \eqref{eq: diagram sequences} we deduce
 $$\xymatrix{&0\ar[r]&H^0(Y,\Omega_Y^1(\log (D_1),\log(D_2),\log(D_3)))\ar[r]&\bigoplus_{i=1}^3H^0(D_i,\mathcal{O}_{D_i})\ar[r]^{\ \ \ \ \ \ \ c_{\oplus}}&\\
 	\ar[r]&H^1(Y,\Omega_Y^1)\ar[r]&H^1(Y,\Omega_Y^1(\log (D_1),\log(D_2),\log(D_3)))\ar[r]&\bigoplus_{i=1}^3H^1(D_i,\mathcal{O}_{D_i})\ar[r]&0}$$
 which is 
 \[
 0 \longrightarrow \mathbb{C}^a  \longrightarrow \mathbb{C}^3 \stackrel{c_{\oplus}}{\longrightarrow} \mathbb{C}^{\rho(Y)} \longrightarrow \mathbb{C}^b \longrightarrow \mathbb{C}^{g(D_1)} \longrightarrow 0
 \]
 By assumption $\dim({\rm Im}(c_{\oplus}))=1$, which implies $a=2$. Hence $b=2-3+\rho(Y)+g(D_1)$.
 So $h^1(Y,\Omega_Y^1(\log(D_i)))=h^1(Y,\Omega_Y^1(\log(D_i)+\log(D_j)+\log(D_k)))$. It follows then that $\rho^{\vee}$ is an isomorphism as wished.\endproof

 We observe that if $(i,j,k)$ is a permutation of $(1,2,3)$ then $\chi_i\chi_j=\chi_k$ and $D_{i}$ is a common component of $D_{\chi_i}=D_i$, $D_{\chi_k^{},\chi_k^{-1}}=D_i+D_j$ and $D_{\chi_k,\chi_j}=D_i$.
 With the notation of \cite[page 255]{Par91b}, $D_i$ is a component of type $b)$.
 Then we can apply the strategy given in \cite[page 256]{Par91b}, to discuss the injectivity of the map $q_{\chi_i,\chi_j}$ where $i,j\in\{1,2,3\}$.

 To see this use the short exact sequence at p. 256 \cite{Par91b}. If $\left((\omega_{Y}(D_{\chi_1})\otimes\mathcal{L}_{\chi_1})^{-1}\right)$ has no global sections one can apply the strategy of the proof of \cite[Theorem 3.1]{Par91b} and consider \cite[Diagram 3.4]{Par91b} then the desired injectivity of $q_{\chi_i,\chi_j}$, $i,j\in\{1,2,3\}$ follows from the injectivity of the map $i_{\chi_i,\chi_j}$ which is an isomorphism.

 \begin{theorem}\label{theo: Infinitesimal Torelli}
 	Let $X\ra Y$ be a smooth bidouble cover of a minimal rational surface such that two intermediate double cover surfaces, say $Y_2$ and $Y_3$, are K3 surfaces.
 	Assume that:
 	\begin{itemize}\item  $D_1$, $D_2$, $D_3$ are irreducible and their classes in $Pic(Y)$ are proportional;
 		\item $D_2$ and $D_3$ are rational curves;
 		\item $\left((\omega_{Y}(D_{\chi_i})\otimes\mathcal{L}_{\chi_i})^{-1}\right)$ has no global sections for $i=1,2,3$.\end{itemize}
 	Then the Infinitesimal Torelli Property holds for $X$.
 \end{theorem}
 \proof We prove that  for every $\chi \in G^*$ we have
 \(
 \bigcap_{\phi \in G^*} \ker \rho_{\chi, \phi}=\{0\},
 \) which implies the statement by Proposition \ref{prop_InfTorelliCrit}.
 
 We first consider $\chi=\chi_0$ and the maps $\rho_{\chi_0,\phi}=r_{\phi,\phi}\circ q_{\chi_0,\phi}$. 
 Since $Y_2$ and $Y_3$ are K3 surfaces, $r_{\chi_2,\chi_2}$ and $r_{\chi_3,\chi_3}$ are injective, by Proposition \ref{prop: rchichii injective for K3},  so $\ker(\rho_{\chi_0,\chi_j})=\{0\}$ if and only if $\ker(q_{\chi_0,\chi_j})=\{0\}$ for $j=2,3$. 
 By the assumptions on $D_1$, $D_2$, $D_3$, $\ker(q_{\chi_0,\chi_2})\cap \ker(q_{\chi_0,\chi_3})=\ker(\rho^{\vee})=0$ (by Corollary \ref{corollary: D2D3 rational rhovee isomorphism} and Lemma \ref{lemma: intersection ker}).
 So $\bigcap_{\phi_0 \in G^*} \ker \rho_{\chi_0, \phi}=\{0\}$.
 
 Let $\chi=\chi_i$, $i\neq 0$, and $\phi=\chi_j$ with $j\neq i$. Then $\rho_{\chi_i,\chi_j}=r_{\chi_k,\chi_j}\circ q_{\chi_i,\chi_j}$ with $(i,j,k)$ a permutation of $(1,2,3)$.
 If $j=2,3$, then $r_{\chi_k,\chi_j}$ is injective, by Proposition \ref{prop: rchichii injective for K3} so $\ker(\rho_{\chi_0,\chi_j})=\{0\}$ if and only if $\ker(q_{\chi_0,\chi_j})=\{0\}$ for $j=2,3$. Since $\left((\omega_{Y}(D_{\chi_i})\otimes\mathcal{L}_{\chi_i})^{-1}\right)$ has no global sections, the maps $q_{\chi_i,\chi_j}$ are injective for $j=1,2,3$, by \cite[page 256]{Par91b}. So $\bigcap_{\phi \in G^*} \ker \rho_{\chi_i, \phi}=\{0\}$, $i=1,2,3$. \endproof

\subsection{Hodge Structures and Chow groups}\label{sec_K3Hodge}
In this section we restate and generalize some results due, essentially, to Laterveer about the Hodge structures and certain Chow groups of surfaces which are bidouble covers of rational surfaces.

Let $X\ra Y$ be a smooth bidouble cover, we keep the notation as in Section \ref{sec_BiDou}.

The induced representation of $G=(\Z/2\Z)^2$ in cohomology yields a splitting of the second cohomology group $H^2(X,\Q)$ into a direct sum of eigenspaces. More precisely, we have the following decomposition  with respect to the actions of $\sigma_1^*$, $\sigma_2^*$ and $\sigma_3^*$
$$ H^2(X,\Q)= H^2(X,\Q)_{++}\oplus H^2(X,\Q)_{+-}\oplus H^2(X,\Q)_{-+}\oplus H^2(X,\Q)_{--}$$
where $H^2(X,\Q)_{\epsilon\delta}$ is the subspace of $H^2(X,\Q)$ where $\sigma_1^*$ acts as $\epsilon\id$ and $\sigma_2^*$ acts as $\delta\id$ and $\epsilon,\delta \in \{\pm\}$. It follows that $\sigma_3^*$ acts as $\epsilon\id\cdot\delta\id$.

Let $\widetilde{Y_i}$ be the minimal resolution of $Y_i$. The cohomology group $H^2(\widetilde{Y_i},\Q)$ is the direct sum of two contributions: one comes from the invariant part of $H^2(X,\Q)$, whereas the other one is given by  the exceptional divisors of $\widetilde{Y_i}\ra Y_i$. The first terms will be denoted by $H^2(X/\sigma_i,\Q)$ and it is isomorphic to $H^2(X,\Q)^{\sigma_i^*}$. The latter consists of  algebraic classes, hence, if we restrict our attention to the transcendental part $T_{Y_i}\otimes \Q\subset H^2(X,\Q)$ of the cohomology, they give no contribution. Summarizing, we obtain
$$H^2(Y_i,\Q)=H^2(X/\sigma_i,\Q)\simeq H^2(X,\Q)^{\sigma_i},$$
$$T_{\widetilde{Y_i}}\otimes\Q \simeq T_{\widetilde{Y_i}}\otimes\Q \simeq T_{X/\sigma_i}\otimes \Q\simeq T_X^{\sigma_i^*}\otimes\Q.$$

The previous isomorphisms hold also at the level of the weight two Hodge structure.

Therefore, they yield 

$$\begin{array}{ll}H^2(Y_1,\Q)\simeq H^2(X,\Q)_{++}\oplus H^2(X,\Q)_{+-},& H^2(Y_2,\Q)\simeq H^2(X,\Q)_{++}\oplus H^2(X,\Q)_{-+},\\H^2(Y_3,\Q)\simeq H^2(X,\Q)_{++}\oplus H^2(X,\Q)_{--},& H^2(Y,\Q)\simeq H^2(X,\Q)_{++}.\end{array}$$

If $Y$ is a rational surface, $T_Y$ is trivial and this implies that $(T_X\otimes \Q)_{++}$ is trivial too.

The previous discussion allows us to prove the following proposition.

\begin{prop}\label{prop: Hodge structure}
	Let $X\ra Y$ be a bidouble cover of a rational surface $Y$ as in Construction \ref{construction 1: first double cover}, then $$T_X\otimes \Q=p_1^*(T_{Y_1}\otimes\Q)\oplus p_2^*(T_{Y_2}\otimes\Q)\oplus p_3^*(T_{Y_3}\otimes\Q)\simeq (T_{Y_1}\otimes \Q)\oplus (T_{Y_2}\otimes \Q)\oplus (T_{Y_3}\otimes \Q).$$

 \end{prop}

The proposition follows directly by the previous computations and by the fact that $p_*$ and $p^*$ are inverse map on the cohomology with rational coefficients. So, if $Y$ is rational we have a splitting of the transcendental Hodge structure of $X$ into the direct sum of three transcendental Hodge substructures, which are the Hodge structures of the intermediate covering surfaces $Y_i$.

\begin{rem}{\rm   If we assume that $Y$ and $Y_1$ are rational, then the transcendental Hodge structure of $X$ splits into the direct sum of the transcendental Hodge structure of the remaining other two intermediate quotient surfaces $Y_2$ and $Y_3$. If these surfaces are K3 surfaces, then the transcendental Hodge structure of $X$  is the direct sum of two Hodge structures of K3-type, both geometrically associated to a K3 surface (i.e., to the coverings $X \ra Y_2$ and $X \ra Y_3$). Similarly if the surface $Y$ is rational and the surfaces $Y_1$, $Y_2$ and $Y_3$ are K3 surfaces, the Hodge structure of $X$ splits into the direct sum of three Hodge substructures of K3-type geometrically associated to K3 surfaces.
We refer to the last situation as a \emph{ triple K3 burger}, see \cite{Lat19}} \end{rem} 

\begin{rem}{\rm The result in Proposition \ref{prop: Hodge structure} depends only on the fact that the transcendental Hodge structure $T_Y\otimes \Q$ is trivial. Hence one can substitute the hypothesis ``$Y$ is a rational surface" with ``$p_g(Y)=0$", obtaining a more general result. Since in the following we always consider $Y$ to be a rational surface we prefer this condition in the hypothesis.}
\end{rem}

	\begin{corollary} \label{cor: decomposition of H2X general case}
		Let $X\ra Y$ be a bidouble cover of a rational surface as in Construction \ref{construction 1: first double cover}. Then $$H^2(X,\Q)\simeq \big(H^2(Y_1,\Z)\oplus 
		{p_2}_*(NS(X)_{-+})\oplus T_{Y_2}\oplus {p_3}_*(NS(X)_{--})\oplus T_{Y_3}\big)\otimes \Q$$
		and ${p_2}_*(NS(X)_{-+})$ (resp. ${p_3}_*(NS(X)_{--})$) is the subspace of $NS(Y_2)$ (resp. $NS(Y_3)$) orthogonal to ${p_2}_*(NS(X)_{++})\otimes \Q\simeq NS(Y)\otimes \Q$ (resp. ${p_3}_*(NS(X)_{++})\otimes \Q$).
	\end{corollary}
 \proof Let us consider the decomposition $$H^2(X,\Q)=\left(H^2(X,\Z)_{++}\oplus H^2(X,\Z)_{+-}\oplus H^2(X,\Z)_{-+}\oplus H^2(X,\Z)_{--}\right)\otimes \Q.$$ 
	In particular $H^2(Y_1,\Q)\simeq \left(H^2(X,\Z)_{++}\oplus H^2(X,\Z)_{+-}\right)\otimes \Q$.
	
	Let us now determine the term  $H^2(X,\Q)_{-+}\simeq \left(NS(X)_{-+}\oplus (T_X)_{-+}\right)\otimes\Q$. 
	
	We already observed that ${T_X}_{++}$ is trivial (because it induces the transcendental structure of $Y$, which is rational) and hence $${T_{X}}_{-+}\otimes \Q\simeq \left({T_X}_{++}\oplus {T_X}_{-+}\right)\otimes \Q\simeq T_{Y_2}\otimes \Q.$$ Moreover, we observe that $$NS(X)_{-+}\otimes \Q\simeq {p_2}_*(NS(X)_{-+})\otimes\Q.$$
So $H^2(X,\Q)_{-+}\simeq \left({p_2}_*(NS(X)_{-+})\oplus T_{Y_2}\right)\otimes \Q$ and similarly $H^2(X,\Q)_{--}\simeq \left({p_3}_*(NS(X)_{--})\oplus T_{Y_3}\right)\otimes \Q$.

Moreover, $NS(Y_2)\otimes \Q={p_2}_*(NS(X)_{++}\oplus NS(X)_{-+})\otimes \Q$, so ${p_2}_*(NS(X)_{-+})$ is orthogonal to ${p_2}_*(NS(X)_{++})$ in $NS(Y_2)$ and considering the bidouble cover $X\ra Y$ one obtains that $NS(X)_{++}\otimes \Q\simeq NS(Y)\otimes \Q$, which conclude the proof.\endproof

\begin{rem}{\rm The previous corollary allows one to decompose $H^2(X,\Q)$ into the sum of: the cohomology of a surface ($Y_1$); two terms which are algebraic (subspaces of the algebraic structures of $Y_2$ and $Y_3$) and two transcendental parts, which are the full transcendental parts of $Y_2$ and $Y_3$.}
\end{rem}

\begin{corollary}\label{corollary hodge structures and trascendental}
Let $X\ra Y$ be a bidouble cover of a rational surface as in Construction \ref{construction 1: first double cover}. If $\rho(Y)=\rho(Y_2)=\rho(Y_3)$, then $$H^2(X,\Q)=\left(H^2(Y_1,\Z)\oplus T_{Y_2}\oplus T_{Y_3}\right)\otimes \Q\mbox{ and }\rho(X)=\rho(Y_1).$$\end{corollary}
\proof  To prove the statement it suffices to show that the terms ${p_2}_*(NS(X)_{-+})$ and ${p_3}_*(NS(X)_{--})$ in the decomposition of $H^2(X,\Q)$ obtained in Corollary \ref{cor: decomposition of H2X general case} are trivial. The N\'eron--Severi of $Y_2$ is $$NS(Y_2)={p_2}_*(NS(X)_{++}\oplus NS(X)_{-+}).$$ In particular $\rho(Y_2)=\rk(NS(X)_{++})+\rk(NS(X)_{-+})$. Since $NS(X)_{++}\otimes \Q\simeq NS(Y)\otimes \Q$, we obtain  $\rho(Y_2)=\rho(Y)+\rk(NS(X)_{-+})$. By the hypothesis follows that $\rk(NS(X)_{-+})=0$. Applying the same argument to $Y_3$ one finds the decomposition of $H^2(X,\Q)$.

Moreover,
{\small
$$NS(X)\otimes \Q\simeq \left(NS(X)_{++}\oplus NS(X)_{+-}\oplus NS(X)_{-+}\oplus NS(X)_{--}\right)\otimes \Q= \left(NS(X)_{++}\oplus NS(X)_{+-}\right)\otimes \Q\simeq NS(Y_1)\otimes \Q.$$}\endproof 

One can also consider an iterated bidouble cover, see Section \ref{subsec: Iterated general case}, and apply the previous arguments to the bidouble cover $W\ra Y_1$. If $Y_1$ is rational, $T_{Y_1}$ is trivial and so $(T_W\otimes \Q)_{++}$ is trivial too.
Iterating the previous arguments to the construction of iterated bidouble cover, as in Section \ref{subsubsec: construction 2}, one obtains the following.
\begin{prop}\label{prop: iterated in general}
	Let $Y$ be a rational surface and  $X\ra Y$ be a bidouble cover. 
	
	Let $\widetilde{W}\ra \widetilde{Y_1}$ be an iterated bidouble cover as in Construction \ref{construction 2: iterated bidouble case 1}, then
	$$T_W\otimes \Q\simeq (T_{Y_2}\oplus T_{Y_3}\oplus T_{Z_1}\oplus T_{Z_3})\otimes \Q;$$
	where $T_{Y_3}$ and $T_{Z_3}$ are the transcendental lattices of K3 surfaces, $T_{Y_2}$ and $T_{Z_1}$ are either trivial or the transcendental lattice of a K3 surface.
	
	Let $\widetilde{W}\ra \widetilde{Y_1}$ be an iterated bidouble cover as in Construction \ref{construction 2: iterated bidouble case 2}, then
	$$T_W\otimes \Q\simeq (T_{Y_2}\oplus T_{Y_3}\oplus T_{Y_2}\oplus T_{Y_3}\oplus T_{Z_3})\otimes \Q;$$
	where $T_{Y_3}$ is the transcendental lattice of a K3 surface, $T_{Y_2}$ and $T_{Z_3}$ are either trivial or the transcendental lattice of a K3 surface.\end{prop}

Now we consider the construction presented in Section \ref{subsubsec: singular bidouble}, i.e., we allow some singularities to the bidouble cover $X\ra Y$. We have already observed that if $D_1$, $D_2$, $D_3$ are smooth and their intersection are transversal, $X$ is smooth and there is a formula to compute $K_X^2$. If one allows some singular points of type $(1,1,1)$ in $D_1\cup D_2\cup D_3$, then the self intersection of the canonical divisor of the minimal model of $X$ changes (accordingly to Section \ref{subsubsec: singular bidouble}). 
On the other hand the birational invariants do not change imposing triple points on $D_1\cup D_2\cup D_3$, and Proposition \ref{prop: Hodge structure} still holds for singular bidouble covers of rational surfaces whose branch locus admits triple points. To be more precise the following Proposition hold.

\begin{prop}\label{prop: singular bidouble general, Hodge}
	Let us consider two cases:\begin{itemize}
		\item let $D_1$, $D_2$, $D_3$ be three smooth curves on a rational surface $Y$ and $D:=D_1\cup D_2\cup D_3$ be simple normal crossing. Let $X\ra Y$ be the bidouble cover branched on $D$;
		\item let now change $D_3$ with a curves $D_3'$ which is a smooth curve linearly equivalent to $D_3$ passing through $k$ of the points in $D_1\cap D_2$. Let $D':=D_1\cup D_2\cup D_3'$, it has $k$ points of type $(1,1,1)$ and let $X'$ be the minimal model of the bidouble cover of $Y$ branched on $D'$. 
	\end{itemize}
	If $X$ and $X'$ are of general type, then $K_{X'}^2=K_X^2-k$ so $X$ and $X'$ lie on different families of surfaces of general type.
	
	If the intermediate double covers of $X\ra Y$ are all either K3 surfaces or surfaces with $p_g=0$, then $T_{X'}\otimes \Q\simeq T_{X}\otimes \Q$ as Hodge structure.\end{prop}
	
\begin{proof} The drops of the self intersection of the canonical bundle were observed by Catanese in \cite[Section 3]{Cat99}. The hypothesis on the intermediate double covers implies that Proposition \ref{prop: Hodge structure} holds, and the transcendental Hodge structures of $X$ and $X'$ are completely determined by the ones of the intermediate double cover. Comparing the two constructions, $D_1$ and $D_2$ do not change and $D_3'$ is linearly equivalent to $D_3$, so the properties of the intermediate double cover remain the same substituting $D_3$ with $D_3'$; indeed the branch locus of the double covers are $D_1\cup D_2$, $D_1\cup D_3\simeq D_1\cup D_3'$, $D_2\cup D_3\simeq D_2\cup D_3'$ and each of them is simple normal crossing. So the transcendental Hodge structures of the intermediate double covers do not change if one substitutes $D_3$ with $D_3'$. \end{proof}

Many of the considerations done for the Hodge-structure can be repeated for the  Chow groups of the 0 cycles of degree 0. One can argue in a similar way to decompose the Chow groups of $X$ into the direct sum of Chow groups of the surfaces $Y_i$, once one has that $Y$ is rational.

\begin{prop}\label{prop: chow groups}{\rm (c.f. \cite[Theorem 3.1]{Lat21})} Let $X\ra Y$ be a bidouble cover of a rational surface, then 
	$$A^2_{0}(X)\otimes\Q\simeq \left(A^2_{0}(Y_1)\oplus A^2_{0}(Y_2)\oplus A^2_{0}(Y_3)\right)\otimes\Q$$ 
	where $A^2_{0}(-)\otimes \Q$ is the Chow group of the 0 cycles of degree 0 with rational coefficients.\end{prop}

The proof depends again by the decomposition of $A^2_{0}(X)\otimes \Q$ in eigenspaces with respect to the action of the involutions $\sigma_i$ and is essentially given in \cite[Theorem 3.1]{Lat22}.

Iterating the previous arguments to the construction of iterated bidouble covers, as in Section \ref{subsubsec: construction 2}, one obtains
\begin{prop}\label{prop: iterated in general, Chow}
	
	Let $Y$ be a rational surface and  $X\ra Y$ be a bidouble cover. 
	
	Let $\widetilde{W}\ra \widetilde{Y_1}$ be an iterated bidouble cover as in Constructions \ref{construction 2: iterated bidouble case 1} and \ref{construction 2: iterated bidouble case 2}, then
	$$A^2_{0}(W)\otimes\Q\simeq (A^2_{0}(Y_2)\oplus A^2_{0}(Y_3)\oplus A^2_{0}(Z_1)\oplus A^2_{0}(Z_3))\otimes\Q.$$\end{prop}

\subsection{Motives and motivic Mumford--Tate conjecture}

In this subsection we will state some results about Chow motives, Andr\'e motives and motivic Mumford--Tate conjecture for bidouble covers whose intermediate quotients are K3 surfaces.  In particular, we prove the motivic Mumford--Tate conjecture for certain bidouble covers of rational surfaces, see Theorem \ref{Theo_MTC}.

We briefly recall some facts on categories of motives, and for the reader convenience we state the motivic Mumford--Tate conjecture, for a more detailed introduction on the subject see \cite{Moonen, Moonenh1}. First we recall some facts about Chow motives and Andr\'e motives of surfaces. 
We do not need full generality, so let $K$ be a subfield of~$\CC$.

Given smooth and projective varieties $X$ and $Y$ over a field $K$ (i.e., objects in the category $\SmPr_{/K}$) of dimension $d_X$ and $d_Y$ respectively, a correspondence of degree $k$ from $X$ to $Y$ is an element $\gamma$ of $A^{d_X+k}(X \times Y)$. Then $\gamma$ induces a map $A^{\cdot}(X) \ra A^{\cdot+k}(Y)$ by the formula
\[
\gamma_*(\beta) := \pi_{2*}(\gamma \cdot \pi^*_1 (\beta)),
\]
where $\pi_1\colon X\times Y \ra X$ and $\pi_2\colon X\times Y \ra Y$ denote the projections.
The category $\ChMot$ of \emph{Chow motives} (with rational coefficients) over $K$ is defined
as follows:
\begin{itemize}
	\item the objects of $\ChMot$ are triples $(X, p, n)$ such that $X \in \SmPr_{/\CC},\, p \in A^{d_X} (X \times X)$ is an idempotent correspondence (i.e. $p_* \circ p_* = p_*$) and $n$ is an integer;
	\item the morphisms in $\ChMot$ from $(X, p, n)$ to $(Y, q, m)$ are  correspondences $f\colon  X \ra Y$
of degree $n-m$, such that $f \circ p = f = q \circ f$. 
\end{itemize}
We recall that $\ChMot$ is an additive, $\QQ$-linear, pseudoabelian category, see \cite[Theorem~1.6]{Scholl}.

There exists a functor 
\[
h \colon \SmPr_{/\CC}^{\opp} \to \ChMot\ \ \mbox{such that } h:X\mapsto (X,1,0)
\]
from the category of smooth projective varieties over~$\CC$
to the category of Chow motives.

We denote also with $A^i(M)$ the $i$-th Chow group
of a motive $M \in \ChMot$.
In general, it is not known whether
the K\"unneth projectors~$\pi_i$ are algebraic,
so it does not (yet) make sense to speak of the summand $h^i(X) \subset h(X)$
for an arbitrary smooth projective variety $X/\CC$.
However, a so-called Chow--K\"unneth decomposition does exist
for curves~\cite{Manin_motive},
for surfaces~\cite{Murre_motsurf},
and for abelian varieties~\cite{DenMur}.
For algebraic surfaces there is in fact the following theorem,
which strengthens the decomposition of the Chow motive.
Statement and proof are copied from \cite[Theorem 2.2]{Lat19}.

\begin{theorem} \label{ChSurf} 
Let $S$ be a smooth projective surface over~$\CC$.
There exists a self-dual Chow--K\"unneth decomposition $\{\pi_i\}$ of~$S$,
with the further property that there is a splitting
\[
\pi_2 = \pi_2^\alg + \pi_2^\tra \quad \in A^2(S \times S)
\]
in orthogonal idempotents, defining a splitting
$h^2(S) = h^2_\alg(S) \oplus h^2_\tra(S)$ with Chow groups
\[
A^i(h^2_\alg(S)) =
\begin{cases}
					\NS(S) & \text{if $i = 1$,}\\
					0 & \text{otherwise,}
				\end{cases} \qquad
\text{and} \quad A^i(h^2_\tra(S)) =
				\begin{cases}
					A^2_{\AJ}(S) & \text{if $i = 2$,}\\
					0 & \text{otherwise.}
				\end{cases}
				\]
				Here $A^2_{\AJ}(S)$ denotes the kernel of the Abel--Jacobi map.
\begin{proof}
	The Chow--K\"unneth decomposition is given in \cite[Proposition 7.2.1]{KMP}.
	The further splitting into an algebraic and transcendental component
	is \cite[Proposition 7.2.3]{KMP}.
\end{proof}
\end{theorem}
			
We can prove the following proposition.
\begin{prop}\label{prop: motives}{\rm (c.f. \cite[Theorem 3.1]{Lat21})}
Let $X\ra Y$ be a bidouble cover branched on a simple normal crossing divisor $D_1\cup D_2\cup D_3$. Let $Y$ be a rational surface.
Then there exists a Chow--K\"unneth decomposition for $Y$ and $Y_i$ such that there is an isomorphism of Chow motives
$$h^2_{\rm tra}(X)\simeq \oplus_{i=1}^3 h^2_{\rm tra}(Y_i)\in\mathcal{M}_{\rm{rat}}$$  where $h^2_{\rm tra}$ is the transcendental part of the motive (see \cite[Theorem 2.6]{Lat21}, \cite{KMP})\end{prop}
\begin{proof} The main argument of the proof is similar to the previous ones on the decompositions of $T_X$ and of $A^2_0(X)$, since all of them are based on the decomposition induced on specific vector spaces by the action of $(\Z/2\Z)^2$. To decompose $h^2_{\rm tra}(X)$ we describe the property ``being invariant" or ``antiinvariant" with respect to a certain involution $\sigma_i$ in terms of appropriate correspondences.  In particular, we recall here the proof given in \cite{Lat21}. We define the motives $h^2(X)_{\pm, \pm}\in\mathcal{M}_{\rm rat}$ by setting
$$\begin{array}{ll}h(X)_{++}=(X,\frac{1}{4}(\Delta_X+\Gamma_1\circ\Delta_X+\Gamma_2), 0),\ &\ h(X)_{+-}=(X,\frac{1}{4}(\Delta_X+\Gamma_1\circ\Delta_X-\Gamma_2), 0)\\h(X)_{-+}=(X,\frac{1}{4}(\Delta_X-\Gamma_1\circ\Delta_X+\Gamma_2),\  0)&\ h(X)_{--}=(X,\frac{1}{4}(\Delta_X-\Gamma_1\circ\Delta_X-\Gamma_2), 0)\end{array}$$
where $\Delta_X$ is the class of the diagonal in $X\times X$ and $\Gamma_i$ is the class of the graph of $\sigma_i$ in $X\times X$.
So one obtains a decomposition $h(X)=h(X)_{++}\oplus h(X)_{+-}\oplus h(X)_{-+}\oplus h(X)_{--},$
which implies the similar decomposition for $h^2(X)$, since one can chose $h^0(X)$ and $h^4(X)$ to be invariant for $\sigma_i$.
				
Now the algebraic part of the motive and then the transcendental part of the motives are compatible with this decomposition and $(h(X)_{\rm tra})_{++}$ is isometric to $ h(Y)_{\rm tra}$, which is trivial since $Y$ is rational. \end{proof}

To speak about motivic Mumford--Tate conjecture we need to introduce the notion of {\it motivated cycles} (for a brief introduction see e.g. \cite[Section 3.1]{Moonen}).

\begin{definition}
	Let $K$ be a subfield of~$\CC$,
 and let $X$ be a smooth projective variety over~$K$.
A class $\gamma$ in $H^{2i}(X)$ is called
a \emph{motivated cycle} of degree~$i$
if there exists an auxiliary smooth projective variety~$Y$ over~$K$
such that $\gamma$ is of the form $\pi_*(\alpha \cup \star\beta)$,
where $\pi \colon X \times Y \to X$ is the projection,
$\alpha$ and~$\beta$ are algebraic cycle classes in $H^*(X \times Y)$,
and $\star\beta$ is the image of~$\beta$ under the Hodge star operation.
\end{definition}
(Alternatively, one may use the Lefschetz star operation,
see~\S1 of~\cite{A96}.)
			
Every algebraic cycle is motivated,
and under the Lefschetz standard conjecture the converse holds as well.
The set of motivated cycles naturally forms a graded $\QQ$-algebra.
The category of motives over~$K$, denoted~$\Mot_K$,
consists of objects $(X,p,m)$,
where $X$ is a smooth projective variety over~$K$,
$p$ is an idempotent motivated cycle on $X \times X$,
and $m$ is an integer.
A morphism $(X,p,m) \to (Y,q,n)$
is a motivated cycle~$\gamma$ of degree $n-m$ on $Y \times X$
such that $q \gamma p = \gamma$.
We denote with $\HH(X)$ the object $(X,\Delta,0)$,
where $\Delta$ is the class of the diagonal in $X \times X$.
The K\"unneth projectors $\pi_i$ are motivated cycles,
and we denote with $\HH^i(X)$ the object $(X,\pi_i,0)$.
Observe that $\HH(X) = \bigoplus_i \HH^i(X)$.
This gives contravariant functors $\HH(\_)$ and $\HH^i(\_)$
from the category of smooth projective varieties over~$K$ to~$\Mot_K$.
			
\begin{theorem} 
\label{mot-props}
	The category $\Mot_K$ is Tannakian over~$\QQ$,
	semisimple, graded, and polarised.
	Every classical cohomology theory of smooth projective varieties over~$K$
	factors via~$\Mot_K$.
\begin{proof}
	See th\'eor\`eme~0.4 of~\cite{A96}.
\end{proof}
\end{theorem}

\begin{definition} 
	Let $K$ be a subfield of~$\CC$.
	An \emph{abelian motive} over~$K$ is an object
	of the Tannakian subcategory of $\Mot_K$
	generated by objects of the form $\HH(X)$
	where $X$ is either an Abelian variety
	or $X = \Spec(L)$ for some finite extension
	$L/K$, with $L \subset \CC$.
										
	We denote the category of abelian motives over~$K$ with~$\AbMot_K$.
\end{definition}
									
Finally we need the following theorem
									
\begin{theorem} 
	
	\label{hodge-is-motivated}
	The Hodge realization functor $H(-)$ restricted to the subcategory of abelian motives is a full functor.
\begin{proof}
See th\'eor\`eme~0.6.2 of~\cite{A96}.
\end{proof}
\end{theorem}

By Theorem  \ref{mot-props}, the singular cohomology and $\ell$-adic cohomology functors
factor via~$\Mot_K$.
This means that if $M$ is a motive,
then we can attach to it a Hodge structure~$\HB(M)$
and an $\ell$-adic Galois representation $\Hl(M)$.
The Artin comparison isomorphism between singular cohomology
and $\ell$-adic cohomology extends to an isomorphism of vector spaces
$\Hl(M) \cong \HB(M) \otimes \QQl$ that is natural in the motive~$M$.
												
We can state the motivated Mumford--Tate conjecture following \cite[Section 3.2]{Moonen}. We shall write $\GB(M)$ for the Mumford--Tate group $\GB(\HB(M))$.
Similarly, we write $\Gl(M)$ (resp.~$\Glc(M)$)
for $\Gl(\Hl(M))$ (resp.~$(\Glc(\Hl(M)$) for the Tate group.
The Mumford--Tate conjecture extends to motives:
for the motive~$M$ it asserts that the comparison isomorphism
$\Hl(M) \cong \HB(M) \otimes \QQl$
induces an isomorphism
$$\Glc(M) \cong \GB(M) \otimes \QQl.$$
												
The results that we have shown, up to here, enable us to prove the Mumford--Tate conjecture for some of the bidouble cover surfaces we have constructed. Indeed, we can prove the following.

\begin{theorem}\label{Theo_MTC} Let $X$ be a surface with $p_g(X)=3$ or $2$ and obtained as a bidouble cover of a minimal rational surface $Y$ whose intermediate quotients $Y_i$, $i=1,2,3$, are three or two K3 surfaces and the possible remaining intermediate quotient is a rational surface. Then the Mumford--Tate and the Tate conjectures hold for $X$  if $\rho(Y_i)=\rho(Y)$ when $Y_i$ is a K3 surface.
\end{theorem}

Before giving the proof we illustrate the demonstration strategy in the realm of motives, we follow the same circle of ideas as in \cite[Section 5]{CP}. 
												
The main idea in \cite{CP} is that for surfaces $X$ with $p_g=2$ it is sometimes  possible to decompose the weight $2$ Hodge structure into two Hodge substructures of K3 type and see that  these Hodge substructures  are indeed the Hodge structures of either Abelian surfaces or K3 surfaces which are (birational) quotients of $X$. This geometric construction makes possible to consider the theory of motivated cycles introduced by Andr\'e, and to decompose the motive of $X$ into two abelian motives of K3 type. For these motives the Mumford--Tate conjecture is known. This, together with the main results of \cite{Com16} and \cite{Com19} allows to prove the Mumford-Tate conjecture for $X$. In \cite{CP} it was used the fact that the surfaces studied are of maximal Albanese dimension hence there is naturally an Abelian surface as a quotient surface, and the authors searched for a K3 surface. Here we have naturally the motives of two or three K3 surfaces coming from the bidouble cover construction. 
												
\begin{proof}[Proof. Theorem \ref{Theo_MTC}] Let $X$ be a bidouble cover of $Y$ a rational minimal surface. Suppose that $p_g(X)=2$ or $3$, that at least two of the intermediate quotients are K3 surfaces, say  $Y_2$ and $Y_3$, and that the remaining one is either rational or another K3 surface. 
Then notice that the weight 2 Hodge structure with rational coefficients of $X$ has  two (three) Hodge substructures that correspond to the transcendental lattices of the K3 surfaces $Y_2$ and $Y_3$ (and $Y_1$ too, in case $p_g(X)=3$) as explained in Proposition \ref{prop: Hodge structure}.
													
Moreover, by  \cite[Theorem 3.1]{Lat22} and by Proposition \ref{prop: motives} we know that there is a decomposition also at the level of Chow motives.  Therefore, we can use the theory of Motivated Cycles developed in \cite{A96}. 
In particular, recall by \cite[Theorem 7.1]{A96} that the motivic cohomology $\HH^2(S)$ of all algebraic K3 surfaces $S$ is an object in the category of abelian motives which we call {\it abelian motives of K3 type}.
													
 There are two cases: either $Y_1$ is a rational surface or a K3 surface. In either case $\HH^2(Y_1)$ is an abelian motives, indeed if $Y_1$ is a rational surface then  all the cohomology is generated by algebraic cycles and therefore $\HH^2(Y_1)$ is built out of Lefschetz motives $ \QQ(-1)$ (see e.g., \cite[Theorem 3.21]{GP}), which are abelian. 

By Corollary \ref{corollary hodge structures and trascendental}, the hypothesis on the Picard number of $Y_i$ implies the following decomposition as direct sum:
\[
\HH^2(X)=\HH^2(Y_1) \oplus\HH^2(Y_2)^{\textrm{tra}}\oplus \HH^2(Y_3)^{\textrm{tra}}= \QQ(-1)^{\rho(Y_1)} \oplus\HH^2(Y_2)^{\textrm{tra}}\oplus \HH^2(Y_3)^{\textrm{tra}}\, \textrm{ if } \, p_g(X)=2
\]
or 
\[
\HH^2(X)=\QQ(-1)^{\rho(Y_1)}\oplus\HH^2(Y_1)^{\textrm{tra}} \oplus\HH^2(Y_2)^{\textrm{tra}}\oplus \HH^2(Y_3)^{\textrm{tra}} \, \textrm{ if } \, p_g(X)=3
\]
here $(-)^{\textrm{tra}}$ denote the \emph{transcendental}  part of a Hodge motive. We conclude that $\HH^2(X)$ is abelian.

By what we just said we note that the motive $\HH(X)$ is an object in the Tannakian subcategory of $\Mot_K$  generated by the Lefschetz motives and $\HH(Y_2)$ and $\HH(Y_3)$ (or by $\HH(Y_1)$, $\HH(Y_2)$ and $\HH(Y_3)$). This is because $\HH(Y_i)$ is generated by Lefschtez motives and $\HH^2(Y_i)^{\textrm{tra}}$, see e.g. \cite[Section 5]{CP}.
The Mumford--Tate conjecture is known for abelian motives of K3 type (see e.g., \cite{A96}).  Hence  for $\HH(Y_i)$ with $i=2,3$ and also for $\HH(Y_1)$ if $Y_1$ is a K3. Therefore it suffices to prove the Mumford--Tate conjecture for $\HH(Y_2)\oplus \HH(Y_3)$ or  $\HH(Y_1)\oplus \HH(Y_2)\oplus \HH(Y_3)$. But this follows from the main results of \cite{Com16} and \cite{Com19}, for more details see \cite[Chapter 17]{ComT}.

Finally, recall that if the Mumford--Tate conjecture is true for~$X$,
then the Hodge conjecture for~$X$
is equivalent to the Tate conjecture for~$X$. In our case the Hodge conjecture is true for $X$, by the Lefschetz$-(1,1)$ Theorem. Therefore the Tate conjecture for $X$ is true too.
\end{proof}

\begin{corollary}\label{cor: MT for singular} Let $X$ and $X'$ be two bidouble covers of a minimal rational surface $Y$ as in Proposition \ref{prop: singular bidouble general, Hodge}. If $X$ satisfies the hypothesis of Theorem \ref{Theo_MTC} then Mumford--Tate  and the Tate conjectures hold for $X'$. 
	\end{corollary}
\proof By Proposition \ref{prop: singular bidouble general, Hodge} the transcendental Hodge structures of $X$ and of $X'$ are isomorphic. The surface $X'$ is the minimal model of a singular surface and it is obtained introducing a $(-4)$-curve for each point of type $(1,1,1)$, cf. \cite{Cat99}. Notice that the orthogonal complement to these exceptional classes in the second cohomology group of $X'$ coincides with the second cohomology group of $X$. Hence 
$$H^2(X')_{\mbox{mot}}\simeq H^2(X)_{\mbox{mot}}\oplus \Q(-1)^k$$
where $k$ is the number of points of type $(1,1,1)$. We recall that Mumford--Tate conjecture holds for the Lefschetz motives $\Q(-1)$ and for $X$ (by Theorem \ref{Theo_MTC}). We conclude because all the motives involved are abelian. \endproof
					
\section{The rational surface $Y$ is $\PP^2$}\label{sec: P2}

In this section we consider bidouble covers and in particular the constructions described in Sections \ref{subsubsec: construction 1}, \ref{subsubsec: construction 2}, \ref{subsubsec: singular bidouble} by choosing the rational surface $Y$ to be $\mathbb{P}^2$.

\subsection{The first bidouble cover}\label{subsec: first bidouble P2}
Here we consider the construction \ref{construction 1: first double cover} assuming $Y=\mathbb{P}^2$.

Let $h$ be the class of a line in $\mathbb{P}^2$, then a bidouble covering data is of the form 
\[
D_i\sim n_ih,
\] 
with $n_i \in \mathbb{N}$, for $i=1,2,3$ satisfying  the condition $n_i+n_j\equiv 0\mod 2$, indeed $L_i=\frac{n_j+n_k}{2}h.$ 
In particular the covering surface $X$ has $$\chi(X)=4+\frac{\sum_i n_i^2+n_1n_2+n_1n_3+n_2n_3-6\sum_in_i}{4}, \quad K^2_X=(-6+n_1+n_2+n_3)^2.$$

\begin{rem}\label{rem: bidoubel P2 K3}{\rm A bidouble cover of $\mathbb{P}^2$ is a K3 surface if and only if, up to a permutation of the indices, $(n_1, n_2, n_3)$ is $(2,4,0)$.}
\end{rem}

\begin{theorem}\label{theorem: bidouble P2} Let $X\ra \mathbb{P}^2$ be a smooth bidouble cover. Let us assume that each intermediate double cover $Y_i$ is either a K3 surface or with $p_g=0$ and that at least one of them is a K3 surface. Then there are the following possibilities (up to permutations of the indices)\footnote{ In the table by $rat$ we mean that the surface $Y_i$  is a rational surface}

\begin{table}[!h]
\begin{tabular}{|c||c|c|c||c|c|c||c|c|c|c|c|c|c|}\hline
 case&$n_1$&$n_2$&$n_3$&$Y_1$&$Y_2$&$Y_3$&$p_g(X)$&$q(X)$&$n_1+n_2+n_3$&$K^2_X$&Surface Name&MTC&ITP\\
  \hline
 a)&$3$&$3$&$1$&{\rm rat}&{\rm rat}&{\rm K3}&$1$&$0$&$7$&$1$& Special Kunev &$\checkmark$ & $\times$ \\
 
b)&$3$&$3$&$3$&{\rm K3}&{\rm K3}&{\rm K3}&$3$&$0$&$9$&$9$&--& $\checkmark$  & \\
 
c)&$4$&$2$&$0$&{\rm rat}&{\rm rat}& {\rm K3}&$1$&$0$&$6$&$0$ &K3 &$\checkmark$ & $\checkmark$ \\

d)&$4$&$2$&$2$&{\rm rat}&{\rm K3}&{\rm K3}&$2$&$0$&$8$&$4$ & Garbagnati Type $2$ & $\checkmark$ & $\checkmark$ \\

e)&$5$&$1$&$1$&{\rm rat}&{\rm K3}&{\rm K3}&$2$&$0$&$7$&$1$ &Special Horikawa &  $\checkmark$ &$\checkmark$ \\
\hline
\end{tabular}
\caption{} \label{Table1}
\end{table}

	The surface $X$ is minimal.

	In case $c)$, $X$ is a K3 surface. In all the other cases $X$ is of general type.
	
	The surface $X$ has the following properties:
	$$T_X\otimes \Q\simeq (T_{Y_1}\oplus T_{Y_2}\oplus T_{Y_3})\otimes \Q,\ A_0(X)\otimes \Q\simeq (A_0(Y_1)\oplus A_0(Y_2)\oplus A_0(Y_3))\otimes \Q,$$
	 $$h(X)_{\rm{tra}}\otimes \Q=\big(h(Y_1)_{\rm{tra}}\oplus h(Y_2)_{\rm{tra}}\oplus h(Y_3)_{\rm{tra}}\big) \otimes \Q.$$
	
	The Mumford--Tate conjecture (MTC) and the Tate conjecture hold in the cases $a), \ldots ,e)$.
	
	The Infinitesimal Torelli Property (ITP) holds in the cases $c)$, $d)$ and $e)$, it does not hold in case $a)$.	
\end{theorem}

We emphasize that we are unaware if the Infinitesimal Torelli Property holds for the surface of case $b)$, see also \cite[Remark 3.4]{Lat19}.

\begin{proof}
	To obtain the Table \ref{Table1}  one has to impose the conditions described in the Construction \ref{construction 1: first double cover}, in particular $n_1+n_2=6$ and $n_i+n_j\leq 6$. Moreover, we exclude the cases where two of $n_1$, $n_2$ and $n_3$ are trivial by Lemma \ref{lemma: IL LEMMA}. Up to a permutation of the indices the cases listed are the unique possible. Observe that the surfaces $Y_i$'s which are not K3 surfaces are rational by \cite[Section V.22]{BHPV}.
	
	The surface $Y_3$ is a K3 surface in all the cases, so the Kodaira dimension of $X$ is at least 0. If $K_X^2$ is positive, then the surface is of general type. It remains to analyze the case $c)$. In this case the surface $X$ is a double cover of a singular K3 surface $Y_3$ branched on the singular points of $Y_3$, indeed the branch curve of the double cover $X\ra Y_i$ is $\pi_i^*(D_i)$ and in our case $D_3=0$. The singular points of $Y_3$ are the eight points $\pi_1^{-1}(D_1\cap D_2)$, inverse image of the intersection points between a quartic and a conic. By \cite{NKummer}, this implies that $X$ is a K3 surface. 
	By Proposition \ref{prop: Hodge structure} one obtains $p_g(X)$ by the knowledge of $p_g(Y_i)$. The computation of the other geometric invariants of $X$ follows applying \eqref{eq_chiBiDoCo} and \eqref{eq_K2BiDoCo}.

	To show that $X$ is minimal in cases $a)$, $b)$, $d)$, $e)$ we observe that $K_X$ is big and nef. Indeed $K_X$ is big and nef if and only if $2K_X$ is big and nef. By Section \ref{sec_BiDou}, $K_X=\pi^*(K_{\mathbb{P}^2}+R)$ and $2R=\pi^*D$. It follows that $2K_X=\pi^*(2K_{\mathbb{P}^2}+D)$. In all the cases, except case $c)$, $2K_{\mathbb{P}^2}+D$ is an ample divisor. This implies that $2K_X$ is nef and big, and so $K_X$ is big and nef too. We already proved that in case $c)$, $X$ is a K3 surface, so it is minimal too.
	
	The decompositions of $T_X$, $A_0(X)$ and $h_{\rm tra}(X)$ follows by Propositions \ref{prop: Hodge structure}, \ref{prop: chow groups}, \ref{prop: motives}.

	The Mumford--Tate and Tate conjectures hold for K3 surfaces by Andr\'e and Tankeev, in particular they hold in case c). The case $a)$ is treated in \cite[Theorem 9.1]{Moonenh1}.
	
	To prove the Mumford--Tate and Tate conjectures for the surfaces $b),d)$ and $e)$ we show that they satisfy the hypothesis of Theorem \ref{Theo_MTC}. The surface $Y_2$ is the double cover of $\mathbb{P}^2$ branched on a sextic which is the union of two curves of degree $n_1$ and $n_3$ respectively. Since the double cover is totally determined by its branch locus, the moduli of a family of K3 surfaces, which are double cover of $\mathbb{P}^2$, equals the number of projective parameters on which the branch locus depends, up to projective transformations. So the dimension of the family of $Y_2$ is $$
	 	\left(\frac{(n_1+2)(n_1+1)}{2}-1\right)+\left(\frac{(n_3+2)(n_3+1)}{2}-1\right)-(9-1),$$ where $(n_i+2)(n_i+1)/2$ is the number of monomials of degree $n_i$ in 3 coordinates and $9-1$ is the number of a projective transformations in $\mathbb{P}^2$.
	 	
	 Denoted by $\widetilde{Y_2}$ the smooth model of $Y_2$, the following relation between the Picard number of $\widetilde{Y_2}$ and the number, $m_{\widetilde{Y_2}}$, of moduli of its family holds: $m_{\widetilde{Y_2}}=20-\rho(\widetilde{Y_2})$.\\ Hence $\rho(\widetilde{Y_2})=30-((n_1+2)(n_1+1)+(n_3+2)(n_3+1))/2$. Since $\widetilde{Y_2}$ is the blow up of $Y_2$ in $n_1n_3$ points, $$\rho(Y_2)=30-\frac{(n_1+2)(n_1+1)+(n_3+2)(n_3+1)}{2}-n_1n_3.$$
	 	
	The previous formula applied to all the cases in Table \ref{Table1} such that $Y_2$ is a K3 surface, gives $\rho(Y_2)=1=\rho(\mathbb{P}^2)$. The construction of the surface $Y_3$ is analogue to the one of the surface $Y_2$ and so $\rho(Y_2)=\rho(Y_3)=1=\rho(\mathbb{P}^2)$.
	
	 	 So the conditions of Theorem \ref{Theo_MTC} are satisfied and this proves the Mumford Tate and the Tate conjectures. 
	 	 
	 	 Notice that if $Y_i$ is a K3 surface $H^2_{\rm{tra}}(Y_i)$ (the transcendental part of the Hodge structure of $Y_i$) decomposes by its Hodge decomposition $(h^{2,0}(Y_i),h^{11}(Y_i),h^{0,2}(Y_i))_{\rm{tra}}$ into $(1,20-\rho(Y_i),1)$ parts. If $Y_i$ is rational $(h^{2,0}(Y_i),h^{11}(Y_i),h^{0,2}(Y_i))_{\rm{tra}}=(0,\rho(Y_i),0)$.
	 	 The number $h^{1,1}(X)$ is deduced using Noether's Formula to the smooth surface $X$. By Corollary \ref{corollary hodge structures and trascendental}, we deduce $\rho(X)=2p_g(X)+h^{1,1}(X)-\sum_i\rk(T_{Y_i})=\rho(Y_1)$. 	
	 	We summarize the results of the calculations in all the interesting cases in the following table.  
	\begin{table}[!h]
\begin{tabular}{|c||c|c|c|c|c|c|}\hline
	case&$h^{1,1}(X)$&$\rho(Y_1)$&$\rho(Y_2)$&$\rho(Y_3)$&$(h^{2,0}(Y_i),h^{11}(Y_i),h^{0,2}(Y_i))_{\rm{tra}}$& for $i=$\\
	\hline
	b) & $31$ &$1$ & $1$ & $1$ & $(1,10,1)$& 1,2,3 \\
	d) & $26$&$4$ & $1$ & $1$ & $(1,11,1)$ &2,3\\
	e) & $29$ &$1$ & $1$ & $1$ & $(1,14,1)$ &2,3 \\
	\hline
\end{tabular}
\end{table}

	In \cite{PZ19} it is proved that the Infinitesimal Torelli Property holds in case $e)$. The case $c)$ follows by general results on families K3 surfaces. 
	The hypothesis of Theorem \ref{theo: Infinitesimal Torelli} are satisfied in case $d)$, therefore the Infinitesimal Torelli Property holds for these surfaces.
	
	For the surfaces in a) the Infinitesimal Torelli Property fails as proven in \cite{K77} (see also \cite{Cat78},\cite{Cat80} and \cite{T80}). 
\end{proof}

For each of the previous cases we can determine the transcendental lattices $T_{Y_i}$ of the K3 surfaces involved in the construction and the rank $r_X$ of the transcendental lattice of $X$. The last column, $m_X$, contains the number of parameters up to projective transformations attached to the choice of the branch locus.
\begin{table}[!h]
\begin{tabular}{|c||c|c|c||c||c|c|c||c|c|c||c}\hline
	case&$T_{Y_1}$&$T_{Y_2}$&$T_{Y_3}$&$r_X$&$m_1$&$m_2$&$m_3$&$m_X$\\
	\hline
	a)&-&-&$U^{\oplus 2}\oplus E_8(-2)$&$12$&-&-&$10$&$12$\\
	b)&$U^{\oplus 2}\oplus E_8(-2)$&$U^{\oplus 2}\oplus E_8(-2)$&$U^{\oplus 2}\oplus E_8(-2)$&$36$&$10$&$10$&$10$&$19$\\
	c)&-&-& $U^{\oplus 2}\oplus D_4(-1)\oplus \langle -2\rangle^5$&$13$&-&-&$11$&$11$\\
	d)&-&$U^{\oplus 2}\oplus D_4(-1)\oplus \langle -2\rangle^5$&$U^{\oplus 2}\oplus D_4(-1)\oplus \langle -2\rangle^5$&$26$&-&$11$&$11$&$16$\\
	e)&-&$U^{\oplus 2}\oplus E_8\oplus \langle -2\rangle^4$&$U^{\oplus 2}\oplus E_8\oplus \langle -2\rangle^4$&$32$&-&$14$&$14$&$16$\\
	\hline
\end{tabular}
\caption{} \label{Table2}
\end{table}

If $Y_i$ is a K3 surface, to compute the lattice $T_{Y_i}$ given in the previous table, we observe that its rank equals the one of $T_{\widetilde{Y_i}}$, which is $22-\rho(\widetilde{Y_i})$. Since $\rho(\widetilde{Y_i})$ is computed in the proof of Theorem \ref{theorem: bidouble P2}, one knows $\rk(T_{Y_i})$. Since $Y_i$ is a double cover of $\mathbb{P}^2$, the cover involution fixes curves which are isomorphic to the branch locus. In particular, one can check that $\widetilde{Y_i}$ is a very general member of a family of K3 surfaces with a non-symplectic involution with a prescribed fixed locus. This allows to determines its transcendental lattices by using the results in \cite{N}, see also \cite{AST}, since the branch locus coincides with the fixed locus of the cover involution and so the transcendental lattice is the sublattice of $H^2(\widetilde{Y_i},\Z)$ on which the cover involution acts as minus the identity.

\begin{remark}{\rm 
		In case $c)$ the surface $X$ is a K3 surface, which is a double cover of the K3 surface  $Y_3$. 
		By the knowledge of the transcendental lattice of $Y_3$, we deduce the N\'eron--Severi group of the minimal model $\widetilde{Y_3}$ of $Y_3$, which is $NS(\widetilde{Y_3})=\langle 2\rangle \oplus N$ (where $N$ is the Nikulin lattice,  which characterises the K3 surfaces admitting a double cover with another K3 surface, see e.g. \cite{CG}). By \cite{CG}, we deduce that the N\'eron--Severi group of $X$, which is the double cover of $\widetilde{Y_3}$, is an overlattice of index 2 of $\langle 4\rangle\oplus N$. This implies that $T_X\simeq U(2)\oplus U(2)\oplus \langle -4\rangle\oplus E_8(-1)$.
		
		By the Theorem \ref{theorem: bidouble P2} (and the fact that both $Y_1$ and $Y_2$ are rational), $T_X\otimes \Q\simeq T_{Y_3}\otimes \Q$ (and indeed $\rk(T_X)=\rk(T_Y)=13$), but the isomorphism is not compatible with the lattice structure over $\Z$, since $$U^{\oplus 2}\oplus D_4(-1)\oplus \langle -2\rangle^5\simeq T_{Y_3}\not\simeq T_{X}\simeq U(2)\oplus U(2)\oplus \langle -4\rangle\oplus E_8(-1).$$
		
		In particular, the decomposition in Proposition \ref{prop: Hodge structure} holds on $\Q$, but one can not extend this result to the lattice structure over $\Z$.}\end{remark}
\begin{remark}\label{rem GS2}{\rm 
	In the sequel the most interesting cases for us will be those with $p_g(X)=2$, i.e., the last two cases in Table \ref{Table1}. Since these are regular surfaces and for case e)  we have $p_g=K^2/2+3/2$ we refer to the surfaces of case e) as Special Horikawa Surfaces 1  (SHS1 for short). Moreover, the surfaces in case d) are known as surfaces of type GS2b  meaning \emph{Garbagnati surfaces} of type 2b following \cite{Lat22}. }
\end{remark}

\medskip
\subsection{Singular bidouble}\label{subsect singular P2}
We now consider the construction described in Section \ref{subsubsec: singular bidouble}, changing the curve $D_3$ in Theorem \ref{theorem: bidouble P2} to obtain triple points in the branch locus.

\medskip

{\bf Specializations of case $b)$ to $(1,1,1)$ points}.
The curves $D_1$, $D_2$ and $D_3$ are cubics. One can specialize $D_3$ requiring that it passes through $k$ of the 9 points in $D_1\cap D_2$. Observe that $1\leq k\leq 9$, but $k\neq 8$, since if $D_3$ has 8 points  in commune with $D_1$ and $D_2$, then also the ninth point of $D_1\cap D_2$ is forced to be in $D_3$.

By considering the the bidouble cover branched on these three curves, one obtains a singular surface, whose minimal model will be denoted by $X'$.

\begin{prop}\label{prop_previous}
	
	Let $X'$ be the minimal model of a bidouble cover of $\mathbb{P}^2$ branched on the divisors $D_i$, $i=1,2,3$ as above and assume that there are $k=1,\ldots, 7$ points of type $(1,1,1)$.
	Then $X'$ is of general type, $$p_g(X')=3,\ \ K_{X'}^2=9-k,$$
	$$T_{X'}\otimes \Q\simeq \left(T_{Y_1}\oplus T_{Y_2}\oplus T_{Y_3}\right)\otimes \Q\mbox{ and }
	A_0^2(X')\otimes \Q\simeq \left(A_0^2(Y_1)\oplus A_0^2(Y_2)\oplus A_0^2(Y_3)\right)\otimes \Q.$$
	The Mumford--Tate and the Tate conjecture holds for $X'$.
\end{prop}
\begin{proof} The result follows from Propositions \ref{prop: singular bidouble general, Hodge}, \ref{prop: chow groups} and Corollary \ref{cor: MT for singular}.\end{proof}

\begin{rem}
{\rm Imposing a point of type $(1,1,1)$ to the configuration of the three cubic decrease by one the projective parameters on which the branch locus depends.
	In particular, we obtain that the families in the previous Proposition \ref{prop_previous} depends on $m_X-\#(\mbox{triple points})$ parameters, i.e. on $10+K_X^2$ parameters.}\end{rem}
\begin{rem}{\rm The transcendental lattices of the K3 surfaces $Y_i$ in the previous Proposition \ref{prop_previous} are all isometric to $U\oplus U\oplus E_8(-2)$ as the ones obtained in the case $D_1\cup D_2\cup D_3$ is simple normal crossing, see Proposition \ref{prop: singular bidouble general, Hodge}
}\end{rem}

\medskip

{\bf Specializations of case $d)$ to $(1,1,1)$ points}.
Let $q$ be a smooth quartic and $c_1$ and $c_2$ two conics in $\mathbb{P}^2$ such that $\{P_k\}=q\cap c_1\cap c_2$ for $k=1,2,3$.
Then the points $P_k$ are points of type $(1,1,1)$ for the data $D_1=q$, $D_2=c_1$, $D_3=c_2$. The minimal model of the covering surface $X$ has $K_X^2=4-i$. So we obtain different families of surfaces of general type for which we have analogues results as before.

\subsection{Iterated bidouble covers over $\mathbb{P}^2$}\label{subsec: iterated P2}

The aim of this section is to construct an iterated bidouble cover as described in Section \ref{subsubsec: construction 2} considering as first bidouble cover either the case $d)$ or the case $e)$ of Table \ref{Table1} in  Theorem \ref{theorem: bidouble P2}. Thanks to this construction, one is able to obtain a surface $W$ of general type whose transcendental Hodge structure splits into the direct sum of transcendental Hodge structures of K3-type geometrically associated to K3 surfaces.

As remarked in Section \ref{sec_BiDou} the surface $Y_1$ is singular. To construct a smooth model of $Y_1$, one blows up the  $n_2\cdot n_3$ intersection points $D_2\cap D_3$ in $\mathbb{P}^2$ and then one constructs the double cover $\widetilde{\pi_1}\colon\widetilde{Y_1}\ra\widetilde{\mathbb{P}}^2$ of the blown up surface branched on the union of the strict transforms, $\widetilde{D}_2$ and $\widetilde{D}_3$, of $D_2$ and $D_3$. The Picard group of $\widetilde{Y_1}$ contains the classes $A:=\pi_1^*h$ and $E_j$, $j=1,\ldots, n_2n_3$ (where $E_j=\widetilde{\pi_1^*(e_j)}$ and $e_j$ are the exceptional divisors of the blow up of $\mathbb{P}^2$). 

The intersection properties of such divisors are: $A^2=2$, $E_j^2=-2$ and the other intersections are trivial.

The branch divisor of $\widetilde{\pi_1}$ is $$\widetilde{D_2}+\widetilde{D_3}=n_2h-\sum_j e_j+n_3h-\sum_je_j=2n_2h-2\sum_je_j,$$ where we used that $n_2=n_3$. So the canonical bundle of $\widetilde{Y_1}$ is 
$$K_{\widetilde{Y_1}}=\widetilde{\pi_1}^*\left(-3h+\sum_j e_j+n_2h-\sum_je_j\right)=(n_2-3)A.$$ 
Let us now consider the smooth double cover $\widetilde{X}\ra\widetilde{Y_1}$ induced by $X\ra Y_1$.
Its the branch divisor is, with the notation of Section \ref{subsubsec: construction 2}, $\widetilde{B}+E$, where $\widetilde{B}:=n_1A$ and $E:=\sum_jE_j$.

To perform Construction \ref{construction 2: iterated bidouble case 1}, we looks for three effective divisors $\Delta_i$, $i=1,2,3$ such that:\\
$\Delta_i+\Delta_j=2\Lambda_k$ for a certain divisors $\Lambda_k$ in $NS(\widetilde{Y_1})$;
$\Delta_1+\Delta_2=(6-2n_2)A;$ $\Delta_1+\Delta_3=n_1A+\sum_jE_j$. This will be done in Theorems \ref{theor: iterated bidouble GS2} and \ref{theo: iterated SHS1}.

To perform Construction \ref{construction 2: iterated bidouble case 2}, we looks for three effective divisors $\Delta_i$, $i=1,2,3$ such that:\\
$\Delta_i+\Delta_j=2\Lambda_k$ for a certain divisors $\Lambda_k$ in $NS(\widetilde{Y_1})$;
$\Delta_1+\Delta_3=\Delta_2+\Delta_3=n_1A+\sum_jE_j$. This will be done in Theorems \ref{theor: iterated bidouble GS2 case 2} and \ref{theo: iterated SHS1 case 2}.

\subsubsection{Iterated bidouble cover of GS2b}\label{subsec: iterated GS2b}
\begin{theorem}\label{theor: iterated bidouble GS2}
There exists exactly one iterated bidouble cover as in Construction \ref{construction 2: iterated bidouble case 1} of the surface GS2b (i.e. case $d)$ of Theorem \ref{theorem: bidouble P2}).
The data of this bidouble cover are 
$$\Delta_1=2A,\ \ \ \Delta_2=0,\ \ \ \Delta_3=2A+\sum_j E_j.$$

The surface $X$ is a bidouble cover of $\mathbb{P}^2$ branched on 4 conics $c_1$, $c_2$, $c_3$, $c_4$.

The surfaces $Y_2$ and $Y_3$ are K3 surfaces obtained as double cover of $\mathbb{P}^2$ each branched on the union of three conics, $c_1\cup c_2\cup c_3$ and $c_1\cup c_2\cup c_4$ respectively.

The surface $Z_3$ is a K3 surface bidouble cover $\mathbb{P}^2$ branched on the union of the reducible quartic $c_3\cup c_4$ and the conic $c_1$.

The surface $Z_1$ is an Enriques surface  obtained as bidouble cover of  $\mathbb{P}^2$ branched on the union of three conics, which are $c_3$, $c_4$ and another conic $c_5$.

The surface $W$ is a regular surface of general type with $p_g(W)=3$:
$$T_W\otimes \Q\simeq (T_{Y_2}\oplus T_{Y_3}\oplus T_{Z_3})\otimes \Q\mbox{ where }\rk(T_{Y_2})=\rk(T_{Y_3})=\rk(T_{Y_3})=9,\mbox{ so }\rk(T_{W})=27;$$
$$A^2_{0}(W)\otimes\Q\simeq (A^2_{0}(Y_2)\oplus A^2_{0}(Y_3)\oplus A^2_{0}(Z_3))\otimes\Q.$$ 	

\end{theorem}
\begin{proof}
Looking at Table \ref{Table1}, in this case we have $n_3=n_2=2$, $n_1=4$, $K_{\widetilde{Y_1}}=-A$, $B=4A+\sum_{j=1}^4E_j$.

Hence we are looking for effective divisors $\Delta_i\in Pic(\widetilde{Y_i})$ such that $$\Delta_1=xA,\ \ \Delta_2=(2-x)A,\ \ \Delta_3=(4-x)A+\sum_{j=1}^4E_j.$$

The effectivity of $\Delta_j$  implies $0\leq x\leq 2$ and $|\Delta_2+\Delta_3+2K_{\widetilde{Y_1}}|\leq 0$ implies that $x\geq 2$. Hence $x=2$ and the unique possible choice is $\Delta_1=2A$, $\Delta_2=0$, $\Delta_3=2A+\sum_j E_j$.

The double cover $\widetilde{X}\ra \widetilde{Y_1}$ is now branched on the union of $\Delta_1$ and $\Delta_3$, which means that the branch curve $q_1^*(D_1)$  splits in the union of two curves, which are the pullback of two conics in $\mathbb{P}^2$. This forces $D_1$ to be the union of two conics, say $c_1$ and $c_2$. Notices that this produces a subfamily of the coverings $Y_2$, $Y_3$ and $X$. In particular $Y_2$ and $Y_3$ are double cover of $\mathbb{P}^2$ branched along three conics (the union of two of them is $D_1$, the third is $D_2=c_3$ or $D_3=c_4$).
Their transcendental lattice is $T_{Y_2}\simeq T_{Y_3}\simeq \langle 2\rangle^2\oplus\langle -2\rangle^7$, computed as above, since these K3 surfaces are double cover of $\mathbb{P}^2$.	

The surface $Z_1$ is an Enriques surface. Indeed we apply \eqref{eq: pulrigeberaYi} to the cover $q_1:Z_1\ra\widetilde{Y}_1$ and we obtain $$P_n(Z_1)=h^0(\widetilde{Y_1}, n\sum_jE_j/2)+h^0(\widetilde{Y_1}, -A+(n-1)\sum_jE_j/2).$$
Since $$h^0(\widetilde{Y_1}, -A+(n-1)\sum_jE_j/2)=0\mbox{ and }h^0(\widetilde{Y_1}, n\sum_jE_j/2)=\left\{\begin{array}{ll} 0&\mbox{ if }n=1\\1&\mbox{ if }n>1\end{array}\right.$$ $Z_1$ is either an Enriques surface or a bielliptic surface. By \eqref{eq_chiBiDoCo} one obtains $\chi(Z_1)=1$, so $Z_1$ is regular and then an Enriques surface.

The surface $Z_3$ which is a $4:1$ cover $\mathbb{P}^2$. The branch of $\pi_1:Y_1\ra\mathbb{P}^2$ consists of two conics $c_3$ and $c_4$ and the branch of $\gamma:Z_3\ra Y_1$ is a curve $\Delta_1$ whose class is $2A=\pi_1^*(2h)$, hence it is the pullback of a conic in $\mathbb{P}^2$. 
Recall that $\Delta_1\cup \Delta_3$ is the branch of $X\ra Y_1$, so it is the pullback of the quartic curve $c_1\cup c_2$. We can assume that $\Delta_1$ is the pullback of $c_1$.
Hence the branch of $\gamma:Z_3\ra Y_1$ is the pullback of $c_1$. So the $4:1$ cover $Z_3\ra \mathbb{P}^2$ splits in the $2:1$ covers $Z_3\ra Y_1$ and $Y_1\ra\mathbb{P}^2$ where the former is branched on the pullback of the conic $c_1$ and the latter of the quartic $c_3\cup c_4$. The cover $Z_3\ra \mathbb{P}^2$ can not be cyclic since there is a branch for $Y_1\ra \mathbb{P}^2$ which is not a branch for $Z_3\ra Y_1$. So $Z_3$ is a bidouble cover of $\mathbb{P}^2$, whose branch data consists of the conic $c_1$, the quartic $c_3\cup c_4$ and an empty divisor. 

To compute the transcendental lattice of $Z_3$, we observe that $(\Z/2\Z)^2\subset \Aut(Z_3)$ is generated by a symplectic involution $\sigma$ and a non--symplectic one, $\iota$. Indeed one of the intermediate double covers of $Z_3\ra\mathbb{P}^2$ is a K3 surface, the others are rational surfaces.
Therefore the second cohomology of $Z_3$ splits in 4 subspaces according to the characters of the group.
We observe that $\rk((H^2(Z_3,\Z)^{\sigma})^{\perp})=8$, since $\sigma$ is symplectic, \cite[Section 7]{NikSym} and that  $\rk(H^2(Z_3,\Z)^{\iota})$ and $\rk((H^2(Z_3,\Z)^{\iota\sigma})$ depends on the fixed locus of the relative involutions. This allows to compute the dimension of all the 4 subspaces. 

The fixed locus of $\iota$ (resp. $\iota\sigma$) is isomorphic to the branch locus of $Z_3\ra Z_3/\iota$  (resp. $Z_3\ra Z_3/\iota$). These branches are the pullbacks of $c_1$ and $c_3\cup c_4$ that are one curve of genus 3, and the union two curves of genus 1 respectively.

This implies that $\rk(H^2(Z_3,\Z)^{\iota})=8$ and $\rk((H^2(Z_3,\Z)^{\iota\sigma})=10$ (\cite{N}, see also \cite{AST}).

By a direct computation one obtains that $\rk(H^2(Z_3,\Z)_{+--})=9$, where  $H^2(Z_3,\Z)_{+--}$ is understood as the space invariant under the symplectic involution $\sigma$ and antiinvariant for the other involutions. This space coincides with the transcendental lattice of a K3 surface generic among the ones constructed as a bidouble cover as above and so determines the family of such a K3 surfaces. This is due to the fact that generically the transcendental lattice of a K3 surface is the space invariant for a symplectic involution and anti invariant for a non--symplectic one.

The surface $W$ is a cover of a surface of general type, so it is of general type. All the other data are obtained applying the results of Section \ref{sec_BiDou}

\end{proof}

\begin{rem}{\rm In Proposition \ref{theor: iterated bidouble GS2} we give only the birational invariants of $W$, because the smooth double cover of $\widetilde{Z_3}$ coming from the construction is not minimal and so  we cannot guarantee the minimality of $W$.}\end{rem}

\begin{rem}\label{rem: BH}{\rm
The K3 surface $Z_3$ appearing in the Theorem \ref{theor: iterated bidouble GS2} admits the group $(\Z/2\Z)^2$ as subgroup of the automorphism group: it can be generated with a symplectic involution and a non-symplectic one. These surfaces are among the ones studied in \cite{BH}, where the authors classify the finite subgroups $G$ of automorphisms of K3 surfaces by considering cohomological lattice theoretic properties. In particular, they describe $G$ as generated by a finite group of symplectic automorphisms and a non-symplectic generator.}\end{rem}

\begin{theorem}\label{theor: iterated bidouble GS2 case 2}
	There exists exactly one iterated bidouble cover as in Construction \ref{construction 2: iterated bidouble case 2} of the surface GS2b (i.e. case $d)$ of Theorem \ref{theorem: bidouble P2}).
	The data of the bidouble cover are 
	$$\Delta_1=A,\ \ \ \Delta_2=A,\ \ \ \Delta_3=3A+\sum_j E_j.$$
	
	The surface $X$ is a bidouble cover of $\mathbb{P}^2$ branched on a cubic $b$, two conics $c_1$, $c_2$, and a line $\ell$.
	
	The surfaces $Y_2$ and $Y_3$ are K3 surfaces obtained as double cover of $\mathbb{P}^2$ branched on the union $b\cup c_1\cup \ell$ and $b\cup c_2\cup \ell$ respectively;
	
	The surface $Z_1$ is a bidouble cover of $\mathbb{P}^2$ branched on $b\cup c_1\cup c_2\cup m$, where $m$ is a line in $\mathbb{P}^2$.

   The surface $Z_3$ is a K3 surface obtained as bidouble cover of $\mathbb{P}^2$ with branch data the divisors $\delta_1=c_1\cup c_2$, $\delta_2=\ell\cup m$, $\delta_3=\emptyset$.

	The surface $W$ is a regular surface of general type with $p_g(W)=5$;
	$$T_W\otimes \Q\simeq (T_{Y_2}\oplus T_{Y_3}\oplus T_{Y_2}\oplus T_{Y_3}\oplus T_{Z_3})\otimes \Q\mbox{ where }\rk(T_{Y_2})=\rk(T_{Y_3})=10,\ \rk(T_{Z_3})=8\mbox{ so }\rk(T_{W})=48;$$
	$$A^2_{0}(W)\otimes\Q\simeq (A^2_{0}(Y_2)\oplus A^2_{0}(Y_3)\oplus A^2_{0}(Y_2)\oplus A^2_{0}(Y_3)\oplus A^2_{0}(Z_3))\otimes\Q.$$ 	

\end{theorem}
\proof The proof is analogue to the one of Theorem \ref{theor: iterated bidouble GS2}. \endproof

\subsubsection{Iterated bidouble cover of SHS1}\label{subsec: iterated SHS1}
\begin{theorem}\label{theo: iterated SHS1}
There exist exactly two iterated bidouble covers as in Construction \ref{construction 2: iterated bidouble case 1} of the surface SHS1  (i.e. case $e)$ of Theorem \ref{theorem: bidouble P2}).
The data of the bidouble cover are 
$$\begin{array}{c||lll}
	
	\mbox{case 1:}&\Delta_1=3A,&\Delta_2=A,&\Delta_3=2A+E\\
	
	\mbox{case 2:}&\Delta_1=4A,&\Delta_2=0,&\Delta_3=A+E.
\end{array}$$

In case 1) the bidouble cover $X\ra \mathbb{P}^2$ is branched on a cubic $e$ a conic $c$ and two lines, $\ell$ and $m$; in case 2) the bidouble cover $X\ra \mathbb{P}^2$ is branched on a quartic $q$ and three lines $\ell$, $m$, $n$.

In case 1) the surfaces $Y_2$ and $Y_3$ are double covers of $\mathbb{P}^2$ branched on the union of a cubic, a conic and a line, in case 2) the branch locus consists of the union of a quartic and two lines.

 In case $1)$ the surface $Z_3$ is a K3 surface, bidouble cover of $\mathbb{P}^2$ branched on an empty divisor, a reducible conic ($\ell\cup m$) and a reducible quartic, which consists of the cubic $c$ and another line.
	
	In case $2)$ the surface $Z_3$ is a K3 surface, bidouble cover of $\mathbb{P}^2$ branched on an empty divisor, a reducible conic ($\ell\cup m$) and the quartic $q$.

In both the cases, the surface $W$ is a regular surface of general type with $p_g(W)=3$;
$$T_W\otimes \Q\simeq (T_{Y_2}\oplus T_{Y_3}\oplus T_{Z_3})\otimes \Q$$ where  $(\rk(T_{Y_2}), \rk(T_{Y_3}),\rk(T_{Z_3}), \rk(T_W))=(10,10,9,29), \ (12,12,12,36)$ in cases 1) and 2) respectively; 	
$$A^2_{0}(W)\otimes\Q\simeq (A^2_{0}(Y_2)\oplus A^2_{0}(Y_3)\oplus A^2_{0}(Z_3))\otimes\Q.$$ 	

\end{theorem}
\begin{proof}
In this case $n_3=n_2=1$, $n_1=5$, $K_{\widetilde{Y_1}}=-2A$ and $B=5A+E$.
Hence we are looking for effective divisors $\Delta_i\in Pic(\widetilde{Y_i})$ $$\Delta_1=xA,\ \ \Delta_2=(4-x)A,\ \ \Delta_3=(5-x)A+E.$$
The effectivity of $\Delta_j$  implies $0\leq x\leq 4$. The condition $\dim|K_{\widetilde{Y_1}}+(\Delta_2+\Delta_3)/2|\leq 0$, implies $x\geq 3$, so the admissible cases are $x=3,4$.

In case 1) the double cover $\widetilde{X}\ra \widetilde{Y_1}$ is branched on the union of $\Delta_1$ and $\Delta_3$, which means that the branch curve $q_1^*(D_1)$  splits in the union of two curves, which are the pullback of a conic and a cubic in $\mathbb{P}^2$. This forces $D_1$ to be the union of a conic and a cubic. So $Y_2$ and $Y_3$ are double covers of $\mathbb{P}^2$ branched along a line, a conic and a cubic and their transcendental lattices are isometric to $U\oplus U\oplus D_4\oplus \langle -2\rangle^4$. 

In case 2) the double cover $\widetilde{X}\ra \widetilde{Y_1}$ is branched on the union of $\Delta_1$ and $\Delta_3$, which means that the branch curve $q_1^*(D_1)$  splits in the union of two curves, which are the pullback of a line and a quartic in $\mathbb{P}^2$. This forces $D_1$ to be the union of a line and a quartic. So $Y_2$ and $Y_3$ are double covers of $\mathbb{P}^2$ branched along two lines and a quartic and their transcendental lattices are isometric to $U\oplus U(2)\oplus\langle -2\rangle^6$. 

The surface $Z_1$ is rational, since its canonical divisor is the pullback of $(-A+E)/2$ and $(-3A+E)/2$ in cases 1) and 2) respectively. 

In case 1) the surface $Z_3$ is the double cover of $\widetilde{Y_1}$ branched on a divisor linearly equivalent to $\Delta_1+ \Delta_2=3A+A=4A$. Similarly to the proof of Theorem \ref{theor: iterated bidouble GS2}, we observe that $\Delta_1$ is the pullback of the cubic $e$ and $\Delta_2$ is the pullback of a line. The $4:1$ cover $Z_3\ra \mathbb{P}^2$, splits in $Z_3\ra Y_1$, branched on the pullback of $e$ and on the pullback of a line, and $Y_1\ra\mathbb{P}^2$, which is branched on $\ell\cup m$. So it cannot be cyclic and indeed it is a bidouble cover whose branch divisors are: an empty set, a reducible quartic, union of $e$ and a line, and a conic, union of $\ell$ and $m$.
As in proof of  Theorem \ref{theor: iterated bidouble GS2}, the transcendental lattice coincides with $H^2(Z_3,\Z)_{+--}$, the subspace invariant for the symplectic involution and antiinvariant for the antisymplectic ones of the bidouble cover $Z_3\ra\mathbb{P}^2$. The fixed loci of the non--symplectic ones are the union of two genus 1 curve and the union of a genus 4 curve and a rational curve.
This implies that $\rk(H^2(Z_3,\Z)_{+--})=9$. As in the Remark \ref{rem: BH}, we observe that $Z_3$ is one of the surfaces considered in \cite{BH} for which both the symplectic and the non--symplectic generators of the finite subgroup of automorphisms have order 2.

The case 2) is similar: the 4:1 cover $Z_3\ra \mathbb{P}^2$, splits in the double covers $Z_3\ra Y_1$, branched on the pull back of $q$, and $Y_1\ra\mathbb{P}^2$ branched on the $\ell\cup m$ and so $Z_3$ is as in the statement.
As above one computes that $\rk(H^2(Z_3,\Z)_{+--})=12$ since the fixed loci of the non--symplectic involutions are a curve of genus 9 and the union of tow genus 1 curves. 
 \end{proof}

\begin{rem}{\rm A priori an iterated bidouble cover can be a $(\Z/2\Z)^3$ cover of $Y$ or not. The surfaces $W$ constructed in Theorems \ref{theor: iterated bidouble GS2}, \ref{theo: iterated SHS1} are a $(\Z/2\Z)^3$ cover of $\mathbb{P}^2$ with $p_g(W)=3$. Hence, they are (possibly special) members of the families of $(\Z/2\Z)^n$ covers of $\mathbb{P}^2$ having $p_g=3$, constructed in \cite{FP}. In particular, the surface constructed in Theorems \ref{theor: iterated bidouble GS2} is in the family C3 in \cite{FP} as shown by the description of the intermediate covers. The surface in Theorem \ref{theo: iterated SHS1}, case 2 is in the family D3 in \cite{FP} and finally the surface in Theorem \ref{theo: iterated SHS1}, case 1 is in the family E3 in \cite{FP}}\end{rem}

	\begin{theorem}\label{theo: iterated SHS1 case 2} 
There exist exactly two iterated bidouble covers as in Construction \ref{construction 2: iterated bidouble case 2} of the surface SHS1  (i.e. case $e)$ of Theorem \ref{theorem: bidouble P2})
The data of the bidouble covers are 
$$\begin{array}{c||lll}

\mbox{case 1:}&\Delta_1=A,&\Delta_2=A,&\Delta_3=4A+E\\

\mbox{case 2:}&\Delta_1=2A,&\Delta_2=2A,&\Delta_3=3A+E.
\end{array}$$

In case 1) the bidouble cover $X\ra \mathbb{P}^2$ is branched on a quartic $q$ and three lines $\ell$, $m$, $n$; in case 2) the bidouble cover $X\ra \mathbb{P}^2$ is branched on a cubic $e$ a conic $c$ and two lines, $\ell$, $m$.

In case 1) the surfaces $Y_2$ and $Y_3$ are double covers of $\mathbb{P}^2$ each branched on the union of the quartic $q$ and two lines, in the case 2) the branch locus consists of the union of the cubic $e$, the conic $c$ and one line.

In case 1) the surface $Z_3$ is a rational surface. In case 2) $Z_3$ is a K3 surface, obtained as bidouble cover of $\mathbb{P}^2$ whose branch data are: an empty set, the union of the two lines $\ell$ and $m$, the union of $c$ and another conic.

The surface $W$ is a regular surface of general type with $p_g(W)=4$ in case 1) and $p_g(W)=5$ in case 2);
$$T_W\otimes \Q\simeq (T_{Y_2}\oplus T_{Y_3}\oplus T_{Y_2}\oplus T_{Y_3}\oplus T_{Z_3})\otimes \Q$$ where  $(\rk(T_{Y_2}), \rk(T_{Y_3}),\rk(T_{Z_3}), \rk(T_W))=(10,10,0,40), \ (12,12,8,56)$ in cases 1) and 2) respectively; 	
$$A^2_{0}(W)\otimes\Q\simeq (A^2_{0}(Y_2)\oplus A^2_{0}(Y_3)\oplus A^2_{0}(Y_2)\oplus A^2_{0}(Y_3)\oplus A^2_{0}(Z_3))\otimes\Q.$$ 

	\end{theorem}
	\proof The proof is analogue to the ones of the previous theorems. \endproof

To recap, the iterated bidouble covers of the surfaces of type GS2b or SHS1 produce regular surfaces $W$ of general type. In Table \ref{Table4} we list for every surfaces constructed: the construction used, the value of $p_g(W)$, the rank $r_W$ of the transcendental lattice $T_W$ of $W$ and the dimension of the families of K3 surfaces $Y_2$, $Y_3$, $Z_3$ (when it is K3 surface). In the last column we give the reference to the theorem where the results are proved. 

\begin{table}[!h]
\begin{tabular}{|c|c||c|c||c|c|c|c|c|c|}\hline
	case&\mbox{Construction}&$p_g(W)$&$r_W$&$m_{Y_2}$&$m_{Y_3}$&$m_{Z_3}$& Thm.\\
	\hline
	GS2b&\ref{construction 2: iterated bidouble case 1}&$3$&$27$&$7$&$7$&$7$&\ref{theor: iterated bidouble GS2}\\
	
	GS2b&\ref{construction 2: iterated bidouble case 2}&$5$&$48$&$8$&$8$&$6$&\ref{theor: iterated bidouble GS2 case 2}\\
	
	SHS1\mbox{ case 1)}&\ref{construction 2: iterated bidouble case 1}&$3$&$29$&$8$&$8$&$7$&\ref{theo: iterated SHS1}\\
	SHS1\mbox{ case 2)}&\ref{construction 2: iterated bidouble case 1}&$3$&$36$&$10$&$10$&$10$&\ref{theo: iterated SHS1}\\
	
	SHS1\mbox{ case 1)}&\ref{construction 2: iterated bidouble case 2}&$4$&$40$&$8$&$8$&\mbox{not K3}&\ref{theo: iterated SHS1 case 2}\\
SHS1\mbox{ case 2)}&\ref{construction 2: iterated bidouble case 2}&$5$&$56$&$10$&$10$&$6$&\ref{theo: iterated SHS1 case 2}\\
	\hline
\end{tabular}

\caption{} \label{Table4}
\end{table}
\section{The rational surface $Y$ is $\PP^1\times\PP^1$}\label{sec: P1xP1}

In this section we consider the constructions described in Sections \ref{subsubsec: construction 1}, \ref{subsubsec: construction 2}, \ref{subsubsec: singular bidouble} by choosing the rational surface $Y$ to be $\mathbb{P}^1\times \mathbb{P}^1$.

\subsection{The first bidouble cover}\label{subsec: first bidouble P1XP1}

Let $h_1$, $h_2$ be the two generators of $\Pic(\mathbb{P}^1\times \mathbb{P}^1)$ such that $h_i^2=0$ and $h_1h_2=1$.
We denote by $D_i\sim n_ih_1+m_ih_2$, $i=1,2,3$ the data of a bidouble cover $X\ra \mathbb{P}^1\times\mathbb{P}^1$. The existence of the bidouble cover implies that $n_i+n_j\equiv 0\mod 2$ and $m_i+m_j\equiv 0 \mod 2$ if $i\neq j$ and $i,j\in\{1,2,3\}$.

For $(i,j,k)$ a permutation of $(1,2,3)$, $L_i=\frac{n_j+n_k}{2}h_1+\frac{m_j+m_k}{2}h_2$ and we denote  $a_i=\frac{n_j+n_k}{2}$, $b_i=\frac{m_j+m_k}{2}$,
$n=\sum n_i$ and  $m=\sum m_i$.
If $D_i$ are smooth and $D_1\cup D_2\cup D_3$ is simple normal crossing, then $X$ is smooth and by \cite{Cat99} (see also Section \ref{sec_BiDou}) $$\chi(X)=\frac{1}{4}\left( (n-4)(m-4)+\sum_i n_im_i\right),\ \ p_g(X)=\sum_i\max(0,a_i-1)\max(0,b_i-1),\ \ K_X^2=2(n-4)(m-4).$$

As in Section \ref{subsubsec: construction 1} we require that $Y_3$ is a K3 surface and that $Y_i$ are either K3 surfaces or surfaces with $p_g=0$, for $i=1,2$. Up to switching $h_1$ and $h_2$ and to change the numbering of the pair $(n_i,m_i)$, the previous conditions on $Y_j$ imply:\begin{itemize}\item $n_2=4-n_1$, $m_2=4-m_1$;
	\item if $n_1+n_3>4$ (resp. $n_1+n_3\geq 4$) then $m_1+m_3<4$, (resp. $m_1+m_3\leq 4$);
	\item if $n_2+n_3>4$ (resp. $n_2+n_3\geq 4$) then $m_2+m_3<4$, (resp. $m_2+m_3\geq 4$). This condition is equivalent to: if $n_3>n_1$ (resp. $n_3\geq n_1$) then $m_3<m_1$, (resp. $m_3\leq m_1$). \end{itemize} 

We observe that, under the condition $n_2=4-n_1$, $m_2=4-m_1$, which is equivalent to $Y_3$ being a K3 surface, it holds 
$$K_X^2=2n_3m_3,\ \ \chi(X)=4-(n_1+m_1)+\frac{n_3m_3+n_1m_1}{2}.$$

So we have the following possibilities (up to switching possibly the role of $(n_1,m_1)$ with the one of $(n_2,m_2)$):
\begin{enumerate}
	\item $n_1+n_3>4$ and $n_3>n_1$: in this case $n_3$ can be arbitrarily big, $m_1=2$, $m_3=0$. This corresponds to an infinite series of examples. Since $m_3=0$, $K_X^2=0$.	
	\item $n_1+n_3>4$ and $n_3=n_1$: in this case $n_1=n_3$ is either 3 or 4. Since $m_1+m_3<4$, $m_3\leq m_1$ and $m_1\equiv_2 m_3$, the possibilities for $(m_1,m_3)$ are $(0,0)$, $(1,1)$, $(2,0)$. This gives a finite list of examples, some with $K_X^2=0$ and some other of general type. The ones of general type are obtained by the choices $(n_1,m_1,m_3)=(3,1,1), (4,1,1)$. Up to permutation of the indices and symmetries of the rulings of $\mathbb{P}^1\times\mathbb{P}^1$, these are the cases $d)$ and $g)$ respectively in Table \ref{TableP1xP1 general} below.
	\item $n_1+n_3>4$ and $n_3<n_1$: in this case $n_1=4$ and $n_2=2$; moreover $m_1+m_3<4$, so we have a finite list of cases. Again some of them have $K_X^2=0$ and others correspond to surfaces $X$ of general type. 
	\item $n_1+n_3\leq 4$ and $n_3\leq n_1$: in this case we may assume that also $m_1+m_3\leq 4$, $m_3\leq m_1$, otherwise it suffices to switch the role of $h_1$ and $h_2$ (i.e. to switch $n_i$ with $m_i$) to obtain one of the previous case. Also these conditions give a finite list of examples.
\end{enumerate}
In the following theorem we provide the classification obtained by the previous conditions. In each Table the column ``case", corresponds to the previous enumerated possibilities for the sum $n_1+n_3$.

\begin{theorem}\label{theor: Y=P1xP1}
	Let $X\ra \mathbb{P}^1\times\mathbb{P}^1$ be a smooth bidouble cover whose branch divisor is simple normal crossing. Let us assume that each intermediate double cover $Y_i$ is either a K3 surface or a surface with $p_g=0$ and that at least one of them is a K3 surface.
	
	If $K_X^2>0$, then $X$ is a minimal surface of general type and the possibilities are listed in Table \ref{TableP1xP1 general}
	\begin{table}[!h]
		\begin{tabular}{|c||c|c|c||c|c|c||c|c|c|c|c|c|c|}\hline
			case&$(n_1,m_1)$&$(n_2,m_2)$&$(n_3,m_3)$&$Y_1$&$Y_2$&$Y_3$&$p_g(X)$&$q(X)$&$(n,m)$&$K^2_X$&case&MTC&ITP\\
			\hline
			a)&$(4,1$)&$(0,3)$&$(2,1)$&{\rm rat}&{\rm rat}&{\rm K3}&$1$&$0$&$(6,5)$&$4$&$(3)$&$\checkmark$&$\times$ \\
			b)&$(4,0)$&$(0,4)$&$(2,2)$&{\rm rat}&{\rm rat}&{\rm K3}&$1$&$0$&$(6,6)$&$8$&$(3)$&$\checkmark$&$\times$ \\
			c)&$(3,1)$&$(1,3)$&$(1,1)$&{\rm rat}&{\rm rat}& {\rm K3}&$1$&$0$&$(5,5)$&$2$&$(4)$&$\checkmark$&$\times$ \\
			d)&$(3,1)$&$(1,3)$&$(1,3)$&{\rm rat}&{\rm K3}&{\rm K3}&$2$&$0$&$(5,7)$&$6$&$(4)$&$\checkmark$&\\
			e)&$(3,2)$&$(1,2)$&$(1,2)$&{\rm rat}&{\rm K3}&{\rm K3}&$2$&$0$&$(5,6)$&$4$&$(4)$&$\checkmark$&\\
			f)&$(3,3)$&$(1,1)$&$(1,1)$&{\rm rat}&{\rm K3}&{\rm K3}&$2$&$0$&$(5,5)$&$2$&$(4)$&$\checkmark$&$\checkmark$\\
			g)&$(3,0)$&$(1,4)$&$(1,4)$&{\rm rat}&{\rm K3}&{\rm K3}&$2$&$0$&$(5,8)$&$8$&$(4)$&$\checkmark$&\\
			h)&$(2,2)$&$(2,2)$&$(2,2)$&{\rm K3}&{\rm K3}&{\rm K3}&$3$&$0$&$(6,6)$&$8$&$(4)$&$\checkmark$ &\\			
			\hline
		\end{tabular} 
		\caption{}\label{TableP1xP1 general}
	\end{table}

	If $K_X^2=0$ the surface $X$ is minimal, admits a genus 1 fibration and the possibilities are listed in Table \ref{TableP1xP1 elliptic}
	
	\begin{table}[!h]
		\begin{tabular}{|c||c|c|c||c|c|c||c|c|c|c|c|c|c|}\hline
			case&$(n_1,m_1)$&$(n_2,m_2)$&$(n_3,m_3)$&$Y_1$&$Y_2$&$Y_3$&$p_g(X)$&$q(X)$&$(n,m)$&$K^2_X$&case&MTC&ITP\\
			\hline
			i)&$(4,2)$&$(0,2)$&$(\geq 4,0)$&{\rm rat}&{\rm rat}&{\rm K3}&$1$&$0$&$(\geq 8,4)$&$0$&$(1)$, $(2)$&$\checkmark$&\\
			j)&$(3,2)$&$(1,2)$&$(\geq 3,0)$&{\rm rat}&{\rm rat}&{\rm K3}&$1$&$0$&$(\geq 7,4)$&$0$&$(1)$, $(2)$&$\checkmark$&\\
			k)&$(2,2)$&$(2,2)$&$(\geq 2,0)$&{\rm rat}&{\rm rat}& {\rm K3}&$1$&$0$&$(\geq 6,4)$&$0$&$(1)$, $(2)$&$\checkmark$&\\
			l)&$(4,0)$&$(0,4)$&$(4,0)$&{\rm K3}&{\rm rul}&{\rm K3}&$2$&$3$&$(8,4)$&$0$&$(2)$&&\\
			m)&$(3,0)$&$(1,4)$&$(3,0)$&{\rm K3}&{\rm rul}&{\rm K3}&$2$&$2$&$(8,4)$&$0$&$(2)$&&\\
			n)&$(4,0)$&$(0,4)$&$(2,0)$&{\rm rat}&{\rm rul}&{\rm K3}&$1$&$2$&$(6,4)$&$0$&$(3)$&$\checkmark$&\\
			o)&$(4,2)$&$(0,2)$&$(2,0)$&{\rm rat}&{\rm rat}&{\rm K3}&$1$&$0$&$(6,4)$&$0$&$(3)$&$\checkmark$&\\
			p)&$(4,0)$&$(0,4)$&$(0,0)$&{\rm rul}&{\rm rul}&{\rm K3}&$1$&$2$&$(4,4)$&$0$&$(4)$&$\checkmark$&$\checkmark$\\
			q)&$(4,1)$&$(0,3)$&$(0,1)$&{\rm rul}&{\rm rat}&{\rm K3}&$1$&$1$&$(4,5)$&$0$&$(4)$&$\checkmark$&\\
			r)&$(4,2)$&$(0,2)$&$(0,0)$&{\rm rat}&{\rm rat}&{\rm K3}&$1$&$0$&$(4,4)$&$0$&$(4)$&$\checkmark$&$\checkmark$\\
			s)&$(4,2)$&$(0,2)$&$(0,2)$&{\rm rul}&{\rm K3}&{\rm K3}&$2$&$1$&$(4,6)$&$0$&$(4)$&&\\
			t)&$(4,3)$&$(0,1)$&$(0,1)$&{\rm rul}&{\rm K3}&{\rm K3}&$2$&$1$&$(4,5)$&$0$&$(4)$&&\\
			u)&$(3,2)$&$(1,2)$&$(1,0)$&{\rm rat}&{\rm rat}&{\rm K3}&$1$&$0$&$(5,4)$&$0$&$(4)$&$\checkmark$&\\
			v)&$(2,2)$&$(2,2)$&$(2,0)$&{\rm rat}&{\rm rat}&{\rm K3}&$1$&$0$&$(6,4)$&$0$&$(4)$&$\checkmark$&\\
			w)&$(2,2)$&$(2,2)$&$(0,0)$&{\rm rat}&{\rm rat}&{\rm K3}&$1$&$0$&$(4,4)$&$0$&$(4)$&$\checkmark$&$\checkmark$\\
			\hline
		\end{tabular} 
		\caption{}\label{TableP1xP1 elliptic}
	\end{table}

	In case $p)$, $X$ is an Abelian surface. In cases $r)$ and $w)$,  $X$ is a K3 surface. In cases $a),\ldots h)$, $X$ is of general type, in all the other cases $X$ is a properly elliptic surface.
	
	The surface $X$ has the following properties:
	$$T_X\otimes \Q\simeq (T_{Y_1}\oplus T_{Y_2}\oplus T_{Y_3})\otimes \Q,\ A_0(X)\otimes \Q\simeq (A_0(Y_1)\oplus A_0(Y_2)\oplus A_0(Y_3))\otimes \Q,$$
	$$h(X)_{\rm{tra}}\otimes \Q=(h(Y_1)_{\rm{tra}}\oplus h(Y_2)_{\rm{tra}}\oplus h(Y_3)_{\rm{tra}}) \otimes \Q.$$
	
	The Mumford--Tate conjecture and the Tate conjecture hold in the cases $a), \ldots, h)$.
	
	The Infinitesimal Torelli Property holds in cases $f)$, $p)$, $r)$ and $w)$ and does not hold in cases $a)$, $b)$ and $c)$.
\end{theorem}
\begin{proof} The proof is analogue to the one of Theorem \ref{theorem: bidouble P2}. 
	To show that all the surfaces $Y_i$ that are not K3 surfaces have a negative Kodaira dimension, we compute the plurigenera $P_k(Y_i)$ by using \eqref{eq: pulrigeberaYi}: \[
	P_k(Y_i)=h^0\left(\mathbb{P}^1 \times \mathbb{P}^1, \, \big( \omega_Y\otimes\mathcal{L}_i\big)^{\otimes k}\right) + h^0\left(\mathbb{P}^1 \times \mathbb{P}^1, \,  \omega_Y^{\otimes k}\otimes\mathcal{L}_i^{\otimes (k-1)}\right),
	\]		
	where 
	$ \mathcal{L}_i=\mathcal{O}_{\mathbb{P}^1\times\mathbb{P}^1}\left(\frac{n_j+n_k}{2},\frac{m_j+m_k}{2}\right).
	$
	Recalling that
	\[
	h^0\left(\mathbb{P}^1 \times \mathbb{P}^1, \mathcal{O}_{\mathbb{P}^1\times\mathbb{P}^1}(a,b)\right)= \begin{cases} 0 & \textrm{ if } b\geq 0, a\leq -2 \textrm{ or } b\leq -2, a\geq 0 \\
		(a+1)(b+1) & \textrm{ if } a\geq -1, b\geq -1 \\
		0 & \textrm{ if } a< 0, b< 0, \\
	\end{cases}
	\]
	one sees that for the choices of $n_1, n_2 ,n_3,m_1,m_2, m_3$ in the tables, $P_k=0$ for all $k$ if $Y_i$ is not a K3 surface. By \eqref{equation: chi intermediate cover} one computes $\chi(Y_i)$ and one deduces $q(Y_i)$. This allows to determine if $Y_i$ is rational or ruled.
	
	We take a closer look to the Table \ref{TableP1xP1 elliptic} now. The K3 surface $Y_3$ admits a genus 1 fibration induced by the projection of $\mathbb{P}^1\times \mathbb{P}^1\ra\mathbb{P}^1$. The surface $X$ is the double cover of the K3 surface branched on $n_1m_2+m_1n_2$ singular points and on the pullback of $D_3$. One can check that in Table \ref{TableP1xP1 elliptic} $D_3$ (if non empty) consists of disjoint the union of fibers of the fibration $\mathbb{P}^1\times \mathbb{P}^1\ra\mathbb{P}^1$. So the pullback of $D_3$ on $Y_3$ consists of the disjoint union of fibers of the genus 1 fibration $Y_3\ra\mathbb{P}^1$. This implies that $X$ admits a genus 1 fibration.  The minimal model of a non rational elliptic surface has $K_X^2=0$, then $X$ is minimal.
	
	The minimality in case $K_X^2>0$ is proved as in Theorem \ref{theorem: bidouble P2}, by proving that $2K_X$ is big and nef.
	
	The statement about the Mumford--Tate conjecture follows from Theorem \ref{Theo_MTC} for the surfaces with $p_g\geq2$ such that the intermediate double covers are either K3 surfaces or rational surfaces and from \cite[Theorem 9.1]{Moonenh1} for those with $p_g=1$. For the former one, we still have to check that $\rho(Y_i)=\rho(Y)$ when $Y_i$ is a K3 surface. 
		We do that as in the proof of Theorem \ref{theorem: bidouble P2}. The surface $Y_2$ is a double cover of $\mathbb{P}^1\times \mathbb{P}^1$ and hence it is uniquely determined by the branch locus. The branch locus is the union of two curves of bidegree $(n_1,m_1)$ and $(n_3,m_3)$.
		The equation of a curve $C\subset \mathbb{P}^1\times \mathbb{P}^1$ of bidegree $(a,b)$ depends on $(a+1)(b+1)-1=ab+a+b$ parameters, since it is $$\sum_{i=0}^{a}x_0^ix_1^{a-i}f_{b}^{(i)}(y_0:y_1)$$
		where $f_b$ is a homogeneous polynomial of degree $b$.
		
		Hence the choice of the branch locus of $Y_2\ra\mathbb{P}^1\times\mathbb{P}^1$ depends on 
		$$(n_1+1)(m_1+1)-1+(n_3+1)(m_3+1)-1$$ projective parameters.
		The projective transformations of $\mathbb{P}^1\times\mathbb{P}^1$ are 6 dimensional. So the choice of the branch locus of $Y_2\ra\mathbb{P}^1\times\mathbb{P}^1$ depends on $(n_1+1)(m_1+1)+(n_3+1)(m_3+1)-8$ 
		parameters up to projectivity and $\rho(\widetilde{Y_2})=20-[(n_1+1)(m_1+1)+(n_3+1)(m_3+1)-8]=28-(n_1+1)(m_1+1)-(n_3+1)(m_3+1)$. Since $\widetilde{Y_2}$ is the blow up of $Y_2$ in $n_1m_3+n_3m_1$, $$\rho(Y_2)=28-(n_1+1)(m_1+1)-(n_3+1)(m_3+1)-n_1m_3-n_3m_1.$$
		Recalling that in our case $Y=\mathbb{P}^1\times\mathbb{P}^1$, we can check that $2=\rho(Y)=\rho(Y_2)$ is all the cases $d),\ldots,h)$. The computations for $\rho(Y_3)$ and $\rho(Y_1)$ in case $Y_1$ is a K3 surface are analogue.

		We summarize the results on the Hodge decomposition of the transcendental part of $Y_i$ and $X$ in the following table.  
		\begin{table}[!h]
			\begin{tabular}{|c||c|c|c|c|c|c|}\hline
				case&$h^{1,1}(X)$&$\rho(Y_1)$&$\rho(Y_2)$&$\rho(Y_3)$&$(h^{2,0}(Y_i),h^{11}(Y_i),h^{0,2}(Y_i))_{\rm{tra}}$& for $i=$\\
				\hline
				d) & $24$ &$8$ & $2$ & $2$ & $(1,8,1)$& 2,3 \\
				e) &$26$ &$6$ & $2$ & $2$ & $(1,10,1)$ &2,3\\
				f) & $28$ &$4$ & $2$ & $2$ & $(1,12,1)$ &2,3 \\
				g) &$22$& $10$ & $2$ & $2$ & $(1,6,1)$ &2,3 \\
				h) &$32$&$2$ & $2$ & $2$ & $(1,10,1)$& 1,2,3 \\
				\hline
			\end{tabular}
		\end{table}
	\newpage

	The hypothesis of Theorem \ref{theo: Infinitesimal Torelli} are satisfied in case $f)$, therefore the Infinitesimal Torelli Property holds for these surfaces. The cases $r)$ and $w)$ (resp. $p)$) follow by general results on families K3 surfaces (resp. Abelian surfaces).
	
	For the surfaces in a), b) and c) the Infinitesimal Torelli Property  fails as proven in \cite{T81}. We do not know if the Infinitesimal Torelli Property holds in the other cases.\end{proof}

The surfaces in cases $a)$, $b)$ and $c)$ are known as Todorov surfaces. In particular surfaces of case c) are studied in \cite{CD} where the authors prove that these surfaces form an unirational family of dimension $11$, moreover Todorov \cite{T81} proved that for these surfaces the Global Torelli Property fails. 

The surfaces in case $h)$ are triple K3 burger. This family is different from the family of triple K3 burgers described in Theorem \ref{theorem: bidouble P2} since the self intersections of the canonical bundles in the two cases are not the same. Similarly the cases $d)$, $f)$ and $g)$ are surely different families from the ones appearing in Theorem \ref{theorem: bidouble P2} with the same geometric genus.

The case $e)$ is associated to a surface $X$ of general type with $p_g=2$ and $K_X^2=4$. This could be a subfamily of the surface GS2b constructed in Theorem \ref{theorem: bidouble P2}, and we do not know if the surfaces in case $e)$ and  the surfaces GS2b are members of the same family.

\begin{rem}{\rm If $D_3=0$, then the surface $X$ is a $2:1$ cover of a K3 surface branched on its singular point and $Y_3$ is a singular model of a K3 surface. This implies that either $Y_3$ has 16 singularities and $X$ is an Abelian surface or $Y_3$ has 8 singularities and $X$ is a K3 surface. If $((n_1,m_1),(n_2,m_2))=((4,0),(0,4))$, $X$ is an Abelian surface (in particular the product of two elliptic curves) and $Y_3$ is its Kummer surface. In the other two cases $X$ is a K3 surface. By considering the elliptic fibration induced on $Y_3$ by the projection $\mathbb{P}^1\times \mathbb{P}^1\ra \mathbb{P}^1$ one deduces the N\'eron Severi group of the resolution of $Y_3$: it is $U\oplus N$ if $((n_1,m_1),(n_2,m_2))=((4,2),(0,2))$ and $U\oplus E_8(-2)$ if $((n_1,m_1),(n_2,m_2))=((2,2),(2,2))$.
		If $NS(Y_3)\simeq U\oplus N$ and one considers the double cover of $Y_3$ branched on the curves generating $N$, it is well known, \cite[Proposition 4.2]{vGS}, that the N\'eron--Severi of the cover surface is isometric to the one of $Y_3$. In particular, in this case $NS(X)\simeq NS(Y_3)$ and $T_X\simeq T_{Y_3}\simeq U\oplus U\oplus N$. 
		
		Since $Y_1$ and $Y_2$ are rational in all these cases, Proposition \ref{prop: Hodge structure} implies that $T_{X}\otimes \Q\simeq T_{Y_3}\otimes \Q$. By the previous results, summarized in the following table, we can observe that it is possible that the equality holds also over $\Z$ or not 
		
		$$\begin{array}{|c|c|c|c|c|c|c|c|}
			\hline
			case&(n_1,m_1)&(n_2,m_2)&(n_3,m_3)&T_{Y_3}&T_X\\
			\hline
			p)&(4,0)&(0,4)&(0,0)&U(2)\oplus U(2)&U\oplus U\\
			r)&(4,2)&(0,2)&(0,0)&U\oplus U\oplus N&U\oplus U\oplus N\\
			w)&(2,2)&(2,2)&(0,0)&U\oplus U\oplus E_8(-2)&U(-4)\oplus U\oplus D_4\oplus D_4\\\hline
		\end{array}$$
		The computation of the transcendental lattice $T_X$ in the last case is done as follows: it is the orthogonal to the N\'eron--Severi group, which is necessarily an overlattice of $U(4)\oplus E_8(-2)$. Indeed the pullback of the two generators of $NS(\mathbb{P}^1\times\mathbb{P}^1)$ are two divisors $H_1,H_2\in NS(X)$, whose intersection form is $U(4)$. Moreover, $X$ is a K3 surface admitting a symplectic involution, $\sigma_3$, which acts trivially on $H_1$ and $H_2$ and as the opposite of the identity on a lattice isometric to $E_8(-2)$ (this follows by the general theory of the symplectic involutions on K3 surfaces, see e.g. \cite{vGS}). The sublattice $\langle H_1,H_2\rangle$ is invariant for all the involutions $\sigma_i$ (since it is the pullback of divisors on $\mathbb{P}^1\times\mathbb{P}^1=X/\langle \sigma_1,\sigma_2\rangle$). The involution $\sigma_i$, for $i=1,2$, fixes a curve of genus 5, hence $NS(X)^{\sigma_i}\simeq \langle 2\rangle\oplus\langle-2\rangle^5$. This implies that $NS(X)$ is a finite index overlattice of $U(4)\oplus\Lambda_1\oplus \Lambda_2$ such that: $\Lambda_1\simeq \Lambda_2$; $NS(X)^{\sigma_i}$ is a finite index overlattice of $U(4)\oplus \Lambda_i$ for $i=1,2$ ; $(NS(X)^{\sigma_3})^{\perp}\simeq E_8(-2)$ is an overlattice of $\Lambda_1\oplus \Lambda_2$. This allows to determine first $\Lambda_i$ (which is a lattice of rank 4 and discriminant group $(\Z/2\Z)^2\times(\Z/4\Z)^2$) and then $NS(X)$. One finds the discriminant form of $NS(X)$ and realizes that it is unique in its genus. In particular $NS(X)$ is isometric to $U(4)\oplus D_4\oplus D_4$ so $T_X\simeq U(-4)\oplus U\oplus D_4\oplus D_4$. }	
\end{rem}
For each of the cases in Table \ref{TableP1xP1 general}, we determine the transcendental lattices $T_{Y_i}$ of the K3 surfaces involved in the construction. We summarize the results in the following table, whose last column contains the number $m_X$ of parameters up to projective transformation attached to the choice of the branch locus.
A priori it is possible that some deformations of $X$ are not obtained as bidouble cover. 
\begin{table}[!h]
	\begin{tabular}{|c||c|c|c||c|c|c||c|c|c||c}\hline
		case&$T_{Y_1}$&$T_{Y_2}$&$T_{Y_3}$&$m_1$&$m_2$&$m_3$&$m_X$\\
		\hline
		a)&$-$&$-$&$\langle 2\rangle^{\oplus 2}\oplus \langle-2\rangle^{\oplus 6}$&$-$&$-$&$6$&$11$\\
		b)&$-$&$-$&$U(2)^{\oplus 2}$&$-$&$-$&$2$&$10$\\
		c)&$-$&$-$&$\langle2\rangle^{\oplus 2}\oplus\langle-2\rangle^{\oplus 8}$&$-$&$-$&$8$&$11$\\
		d)&$-$&$U(2)^{\oplus 2}\oplus \langle -2\rangle^{\oplus 6}$&$U(2)^{\oplus 2}\oplus \langle -2\rangle^{\oplus 6}$&$-$&$8$&$8$&$15$\\
		e)&$-$&$U^{\oplus 2}\oplus\langle-2\rangle^{\oplus 8}$&$U^{\oplus 2}\oplus\langle-2\rangle^{\oplus 8}$&$-$&$10$&$10$&$15$\\
		f)&$-$&$U^{\oplus 2}\oplus D_6\oplus\langle -2\rangle^{\oplus 4}$&$U^{\oplus 2}\oplus D_6\oplus\langle -2\rangle^{\oplus 4}$&$-$&$12$&$12$&$15$\\
		g)&$-$&$U(2)^{\oplus 2}\oplus \langle -2\rangle^{\oplus 4}$&$U(2)^{\oplus 2}\oplus \langle -2\rangle^{\oplus 4}$&$-$&$6$&$6$&$15$\\
		h)&$U^{\oplus 2}\oplus E_8(-2)$&$U^{\oplus 2}\oplus E_8(-2)$&$U^{\oplus 2}\oplus E_8(-2)$&$10$&$10$&$10$&$18$\\
		\hline
	\end{tabular}
	\caption{} \label{Table Tra P1xP1}
\end{table}

\subsection{Singular double covers} 
In Section \ref{subsect singular P2} we constructed surfaces $X$ of general type with $p_g(X)=3$ and $1\leq K_X^2\leq 9$ and with $p_g(X)=2$ and $1\leq K_x^2\leq 4$ by allowing singular bidouble covers of $\mathbb{P}^2$.
By allowing points of type $(1,1,1)$ in the configurations of bidouble cover data listed in Theorem \ref{theor: Y=P1xP1} we obtain the same invariants with a unique exception, that we now describe.
Let $D_i$, $i,1,2,3$, be three smooth divisors such that $D_1\simeq 3h_1+h_2$, $D_2\simeq D_3\simeq h_1+3h_2$, $D_1\cap D_2\cap D_3$ consists of exactly one point and the other singular points of $D_1\cup D_2\cup D_3$ are simple normal crossing. Then the minimal model $X$ of the bidouble cover of $\mathbb{P}^1\times \mathbb{P}^1$ branched on  $D_1\cup D_2\cup D_3$ is such that $p_g(X)=2$, $K_X^2=5$ and the Mumford--Tate and Tate conjectures holds.
\subsection{Iterated bidouble covers of $\mathbb{P}^1\times\mathbb{P}^1$}\label{subsec: iterated P1xP1}

\medskip

\medskip
We restrict ourself to  the construction of iterated bidouble covers of $\mathbb{P}^1 \times \mathbb{P}^1$ only in the case when $Y_2$ is a K3 surface and $Y_1$ is rational, i.e. in the case $p_g(X)=2$ and $X$ is of general type. These are the cases $d)$, $e)$, $f)$, $g)$ of Table \ref{TableP1xP1 general}.

The surface $\widetilde{Y_1}$ is the blow up of $Y_1$ in $n_2m_3+n_3m_2$ points. We denote by $E_j$ the exceptional divisors and $E$ their sum, by $A_1$ and $A_2$ the pull back of the generators $h_1$ and $h_2$ of $\textrm{Pic}(\mathbb{P}^1\times\mathbb{P}^1)$, by $\widetilde{B}$ the strict transform on $\widetilde{Y_1}$ of the branch curve of $X\ra Y_1$ so that $\widetilde{B}+E$ is the branch locus of $\widetilde{X}\ra \widetilde{Y_1}$. Hence we are in the following situation:
\begin{table}[!h]
\begin{tabular}{|c||c|c|c|c|}
	\hline
	case&$n_2m_3+n_3m_2$&$K_{\widetilde{Y_1}}$&$\widetilde{B}$&$E$\\
	\hline
	d)&$6$&$-A_1+A_2$&$3A_1+A_2$&$\sum_{j=1}^6 E_j$\\
	e)&$4$&$-A_1$&$3A_1+2A_2$&$\sum_{j=1}^4 E_j$\\
	f)&$2$&$-A_1-A_2$&$3A_1+3A_2$&$\sum_{j=1}^2E_j$\\
	g)&$8$&$-A_1+2A_2$&$3A_1$&$\sum_{j=1}^8E_j$\\
	\hline
	\end{tabular} 
		\caption{}\label{Table7}
	\end{table}

\begin{theorem}\label{theor: iterated P1xP1 case 1} It holds the following, where we refer to the cases in Theorem \ref{theor: Y=P1xP1}.
\begin{enumerate}
	\item It is not possible to construct an iterated bidouble cover as in Construction \ref{construction 2: iterated bidouble case 1}  for the bidouble cover of the cases $d)$ and $g)$.

	\item In case $e)$ there exists an iterated bidouble cover  W as in Construction \ref{construction 2: iterated bidouble case 1}  whose data are
	$$\Delta_1=2A_1,\ \ \Delta_2=0,\ \ \Delta_3=A_1+2A_2+\sum_jE_j.$$
	In this case $X$ is a bidouble cover of $\mathbb{P}^1\times \mathbb{P}^1$ branched along the union of two curves of bidegree $(1,0)$ and three curves of bidegree $(1,2)$.
	The surface $Z_3$ is a K3 surface which is a bidouble cover of $\mathbb{P}^1\times \mathbb{P}^1$ whose data are  $\Gamma_1=2h_1$,  $\Gamma_2=D_2+D_3$ ($D_2$ and $D_3$ linearly equivalent to $h_1+2h_2$)  and $\Gamma_3=0$.

	\item Up to possibly switching $A_1$ and $A_2$, in case $f)$ there exist three iterated bidouble covers W as in Construction \ref{construction 2: iterated bidouble case 1} whose data are
	$$\Delta_1=2A_1+yA_2,\ \ \Delta_2=(2-y)A_2,\ \ \Delta_3=A_1+(3-y)A_2+\sum_jE_j,\ \ \mbox{ with } y=0,1,2.$$
	In these cases $X$ is a bidouble cover of $\mathbb{P}^1\times \mathbb{P}^1$ branched along the union of one curve of bidegree $(2,y)$, one of curve of bidegree $(1,3-y)$ and two curves of bidegree $(1,1)$.
	The surface $Z_3$ is a bidouble cover of $\mathbb{P}^1\times \mathbb{P}^1$ whose data are  $\Gamma_1=C_1+C_2$ with $C_1=2h_1+yh_2$ and $C_2=yh_2$,  $\Gamma_2=D_2+D_3$ (where $D_2$ and $D_3$ linearly equivalent to $h_1+h_2$)  and $\Gamma_3=0$.
	\end{enumerate}
	
\end{theorem}
\begin{proof}
	We apply the construction in Section \ref{subsubsec: construction 2}, and we look for effective divisors $\Delta_1$, $\Delta_2$, $\Delta_3$ such that
	$\Delta_1+\Delta_2= -2K_{\widetilde{Y_1}}$, $\Delta_1+\Delta_3=  \widetilde{B}+E$, $\dim(|\Delta_1+\Delta_2|)\leq 0$ where $K_{\widetilde{Y_1}}$ and $\widetilde{B}$ and $D$ are as in the Table \ref{Table7}.

	{\bf (1)} In the cases $d)$ and $g)$ it is not possible to find two effective divisors $\Delta_1$ and $\Delta_2$ such that $\Delta_1+\Delta=-2K_{\widetilde{Y_1}}$.
	
	\medskip
	
	{\bf (2)} In case $e)$ one can find the effective divisors $\Delta_1$, $\Delta_2$ and $\Delta_3$ in such a way that $\Delta_1+\Delta_2=-2K_{\widetilde{Y_1}}$ and $\Delta_1+\Delta_3=\widetilde{B}+E$, which are 
	$$\Delta_1=xA_1,\ \ \Delta_2=(2-x)A_1,\ \ \Delta_3=(3-x)A_1+2A_2+\sum_jE_j,\ \ \mbox{ with } 0\leq x\leq 2.$$
The request that $\dim|K_{\widetilde{Y_1}}+(\Delta_2+\Delta_3)/2|\leq 0$  implies $x=2$.
	The curve $\widetilde{B}$ splits into the sum of $\Delta_1$ and $\Delta_3$. This means that the curve $D_1=3h_1+2h_2$ attached on the first bidouble cover $X\ra \mathbb{P}^1\times \mathbb{P}^1$ splits in the union of the two curves $2h_1$ and $h_1+2h_2$. So $Y_1$ and $Y_2$ are the double covers of $\mathbb{P}^1\times\mathbb{P}^1$ branched on the union of three curves, of bidegree $(2,0)$, $(1,2)$ and $(1,2)$ respectively. Since the curves of bidegree $(2,0)$ are necessarily the union of two irreducible curves of bidegree $(1,0)$, we obtain that $Y_2$ and $Y_3$ are bidouble covers of $\mathbb{P}^1\times\mathbb{P}^1$ branched on the union of four rational curves, two of bidegree $(1,0)$ and two of bidegree $(1,2)$. Hence there exists a non-symplectic involution of $Y_2$ and $Y_3$ which fixes 4 rational curves. So their transcendental lattices are isometric $U(2)^{\oplus 2}\oplus\langle -2\rangle^4$. 
	
	To describe the geometry of $Z_3$, let us reconsider more closely also the construction of $Y_1$. The surface $Y_1$ is the double cover of $\mathbb{P}^1\times \mathbb{P}^1$ branched on $D_2$ and $D_3$. Let us now consider the bidouble cover $U\ra \mathbb{P}^1\times \mathbb{P}^1$ whose data are the three effective divisors $\Gamma_1$, which is a curve of bidegree $(2,0)$, $\Gamma_2=D_2+D_3$ and $\Gamma_3=0$. The intermediate double covers are $Y_1$ (whose branch locus is $\Gamma_2+\Gamma_3$), a rational surface $V_2$ (whose branch locus is $\Gamma_1+\Gamma_3$) and a singular model of a K3 surface $V_3$ whose branch locus is $\Gamma_1+\Gamma_2$. The bidouble cover $U$ is a K3 surface, double cover of $V_3$ branched on its eight singular points. By construction $U$ is also a double cover of $Y_1$ branched on the pullback of $\Gamma_1$, so its construction induces on $\widetilde{Y_1}$ the double cover $Z_3\ra \widetilde{Y_1}$.
	In particular $U$ is a birational model of $Z_3$, 
	i.e. the following diagram holds 
	$$\xymatrix {&Z_3\ar[r]\ar[dl]&U\ar[dl]\ar[d]\ar[dr]&\\
		\widetilde{Y_1}\ar[r]&Y_1\ar[dr]&V_2\ar[d]&V_3\ar[dl]\\&&\mathbb{P}^1\times\mathbb{P}^1&&}$$
	where the horizontal arrows are birational transformations, the others are $2:1$ covers. 
	The surface $V_3$ is a K3 surface, double cover of $\mathbb{P}^1\times \mathbb{P}^1$ branched on 4 rational curves (two of type $(1,0)$, two of type $(1,2)$). It follows that it admits a non-symplectic involution fixing 4 rational curves and the N\'eron--Severi group of its resolution is isometric to $U\oplus M_{(2,2)}$ (where $M_{(2,2)}$ is described in \cite[Sections 6 and 7]{NikSym}). The K3 surface $Z_3$ is branched on 8 of the 12 singular points resolved by the curves in $M_{(2,2)}$. To compute the N\'eron--Severi group of $Z_3$ we start by considering the one of $V_3$. The surface $V_3$ admits an elliptic fibration with 12 fibers of type $I_2$, induced by the projection of $\mathbb{P}^1\times \mathbb{P}^1\ra\mathbb{P}^1$. The Mordell Weil group of this fibration is $(\Z/2\Z)^2$ and to consider the cover branched on 8 of the 12 exceptional divisor is equivalent (by \cite[Proposition 2.3]{G13})
 to consider the quotient of the elliptic fibration by the translation by a 2-torsion section. The quotient surface has an elliptic fibration with four fibers of type $I_4$ and Mordell--Weil group $\Z/2\Z$. This allows to determine the N\'eron--Severi group and hence the transcendental lattice: it is the unique (up to isometry) lattice with signature $(2,6)$, discriminant group $(\Z/2\Z)^2\times(\Z/4\Z)^2$ and discriminant form generated by $b_1$, $b_2$, $b_3$, $b_4$ with $b_1^2=b_2^2=0$, $b_1b_i=b_3^2=b_4^2=-1/2$ for $i=2,3,4$, $b_2b_i=0$, for $i=2,3,4$ and $b_3b_4=-1/4$.
	
	\medskip
	
	{\bf (3)} The case $f)$ is analogue;
	\begin{itemize}
		\item 
		if the data are
		$$\Delta_1=2A_1+2A_2,\ \ \Delta_2=0,\ \ \Delta_3=A_1+A_2+\sum_jE_j$$
		then $D_1$ splits in the union of a $(2,2)$ curve and a $(1,1)$ curve, the surfaces $Y_2$ and $Y_3$ are double cover of $\mathbb{P}^1\times\mathbb{P}^1$ branched on the union of a $(2,2)$ curve and two $(1,1)$ curves, so the transcendental lattices of $Y_2$ and $Y_3$ are $T_{Y_2}\simeq T_{Y_3}\simeq U\oplus\langle 2\rangle\oplus \langle -2\rangle^{\oplus 7}$ and $Z_3$ is birational to a bidouble cover of $\mathbb{P}^1\times \mathbb{P}^1$ whose data are $\Gamma_1=2h_1+2h_2$, $\Gamma_2=D_2+D_3$, $\Gamma_3=0$. The K3 surface $V_3$ (defined as above as an intermediate double cover of the bidouble cover $Z_3\ra\mathbb{P}^1\times\mathbb{P}^1$) has the same geometric properties as $Y_2$ and $Y_3$, so that its transcendental lattice is $T_{V_3}\simeq U\oplus\langle 2\rangle\oplus \langle -2\rangle^{\oplus 7}$ and its N\'eron--Severi group is $U\oplus N\oplus \langle -2\rangle^2$. Arguing as above, one obtains that the cover surface $Z_3$ is birational to the quotient of $V_3$ by a symplectic involution which is the translation by a 2-torsion section. In this way one obtains that $Z_3$ admits an elliptic fibration with $2I_4+4I_2$ as reducible fibers and $\Z/2\Z$ as Mordell--Weil group. So $T_{Z_3}$ is the unique even lattice with signature $(2,8)$, discriminant group $(\Z/2\Z)^2\times(\Z/4\Z)^2$ and discriminant form generated by $b_i$, $i=1,\ldots,4$ with the following intersection properties $b_1^2=b_2^2=1$, $b_1b_2=b_3b_4=1/2$, $b_3^2b_4^2=3/4$, the others intersections are zero.
		\item 
		if the data are
		$$\Delta_1=2A_1+A_2,\ \ \Delta_2=A_2,\ \ \Delta_3=A_1+2A_2+\sum_jE_j$$
		then $D_1$ splits in the union of a $(2,1)$ curve and a $(1,2)$ curve, the surfaces $Y_2$ and $Y_3$ are double cover of $\mathbb{P}^1\times\mathbb{P}^1$ branched on the union of a $(2,1)$ curve, a $(1,2)$ curve and a $(1,1)$ curve, so the transcendental lattices of $Y_2$ and $Y_3$ are $T_{Y_2}\simeq T_{Y_3}\simeq U(2)^{\oplus 2}\oplus \langle -2\rangle^{\oplus 5}$ and $Z_3$ is birational to a bidouble cover of $\mathbb{P}^1\times \mathbb{P}^1$ whose data are $\Gamma_1=C_1+C_2$, with $C_1=2h_1+h_2$, $C_2=h_2$, $\Gamma_2=D_2+D_3$, $\Gamma_3=0$. The N\'eron--Severi of the surface $V_3$ is $U\oplus M_{(\Z/2\Z)^2}$ and the considerations of the case $e)$ holds also in this case, so $T_{Z_3}$ is isometric to the one computed in case $e)$.
		\item 
		if the data are
		$$\Delta_1=2A_1,\ \ \Delta_2=2A_2,\ \ \Delta_3=A_1+3A_2+\sum_jE_j$$
		then $D_1$ splits in the union of two $(1,0)$ curves and a $(1,3)$ curve, the surfaces $Y_2$ and $Y_3$ are double cover of $\mathbb{P}^1\times\mathbb{P}^1$ branched on the union of two $(1,0)$ curves, a $(1,3)$ curve and a $(1,1)$ curve, so the transcendental lattices of $Y_2$ and $Y_3$ are $T_{Y_2}\simeq T_{Y_3}\simeq U(2)^{\oplus 2}\oplus \langle -2\rangle^{\oplus 4}$. The K3 surface $Z_3$ is a bidouble cover of $\mathbb{P}^1\times\mathbb{P}^1$ whose data are $\Gamma_1=C_1+C_2$, with $C_1=2h_1$, $C_2=2h_2$, $\Gamma_2=D_2+D_3$, $\Gamma_3=0$.  The K3 surface $V_3$ admits an elliptic fibration with $2I_0^*+6I_2$ as reducible fibers and $(\Z/2\Z)^2$ as Mordell--Weil group. The surface $Z_3$ is birational to the K3 surface obtained as quotient of $V_3$ by the translation by a 2-torsion section and hence admits an elliptic fibration with $2I_0^*+2I_4$ as reducible fibers and $(\Z/2\Z)$ as Mordell--Weil group.  So $T_{Z_3}$ is the unique even lattice with signature $(2,4)$, discriminant group $(\Z/2\Z)^2\times(\Z/4\Z)^2$ and discriminant form generated by $b_i$, $i=1,\ldots,4$ with the following intersection properties $b_2^2=1$, $b_1b_2=b_4^2=-1/2$, $b_3^2=-1/4$, $b_3b_4=-3/4$, the others intersections are zero.
	\end{itemize}

	 \end{proof}
We observe that the surfaces constructed in the previous Theorem satisfy the hypothesis of Proposition \ref{prop: iterated in general} and \ref{prop: iterated in general, Chow} on the decomposition of the transcendental Hodge structures and of the Chow groups.

\begin{theorem}\label{theor: iterated P1xP1 case 2} It holds the following, where we refer to the cases in the tables of Theorem \ref{theor: Y=P1xP1}.
	\begin{enumerate}
		\item It is not possible to construct an iterated bidouble cover as in Construction \ref{construction 2: iterated bidouble case 2} for the bidouble covers of the cases $d)$ and $g)$.
		
		\item In case $e)$ there exists an iterated bidouble cover  W as in Construction \ref{construction 2: iterated bidouble case 2}  whose data are
		$$\Delta_1=A_1,\ \ \Delta_2=A_1,\ \ \Delta_3=2A_1+2A_2+\sum_jE_j.$$
	
		In this case $X$ is a bidouble cover of $\mathbb{P}^1\times \mathbb{P}^1$ branched along the union of two curves of bidegree $(1,2)$, a curve of bidegree $(2,2)$ and a curve of bidegree $(1,0)$.
		The surface $Z_3$ is a K3 surface which is a bidouble cover of $\mathbb{P}^1\times \mathbb{P}^1$ whose data are  $\Gamma_1=2h_1$,  $\Gamma_2=D_2+D_3$ ($D_2$ and $D_3$ linearly equivalent to $h_1+2h_2$)  and $\Gamma_3=0$. In particular $p_g(W)=5$.

		\item In case $f)$ there exist two (up to a permutation of $A_1$ and $A_2$) iterated bidouble covers  W as in Construction \ref{construction 2: iterated bidouble case 2} whose data are
		$$\Delta_1=A_1+A_2,\ \ \Delta_2=A_1+A_2,\ \ \Delta_3=2A_1+2A_2+\sum_jE_j,$$
		or $$\Delta_1=A_1,\ \ \Delta_2=A_1,\ \ \Delta_3=2A_1+3A_2+\sum_jE_j.$$
		
		In the first case $X$ is branched on the union of three curves of bidegree $(1,1)$ and one curve of bidegree $(2,2)$. The surface $Z_3$ is a K3 surface which is a bidouble cover of $\mathbb{P}^1\times \mathbb{P}^1$ branched on the union of 4 curves of bidegree $(1,1)$. In particular $p_g(W)=5$.
		
		In the second case $X$ is branched on the union of two curves of bidegree $(1,1)$, one curve of bidegree $(2,3)$ and one of bidegree $(1,0)$ and $Z_3$ is a rational surface.  In particular $p_g(W)=4$.
		
	\end{enumerate}

\end{theorem}
\proof The proof is analogue to the one of Theorem \ref{theor: iterated P1xP1 case 1}.\endproof

\section{The rational surface is $Y=\mathbb{F}_n$}\label{sec: Fn}

In this last section we consider bidouble covers and in particular the constructions described in Sections \ref{subsubsec: construction 1}, \ref{subsubsec: construction 2}, \ref{subsubsec: singular bidouble} by choosing the rational surface $Y$ to be a Hirzebruch--Segre surface $\mathbb{F}_n$ with $n\geq 2$.

Let us write the classes in $\textrm{Pic}(\mathbb{F}_n)$ as $aC_n+bF_n$ where $F_n$ is a class of a fibre and $C_n$ is the unique curve such that $C_n^2=-n$. Recall that the canonical bundle of $\mathbb{F}_n$ is $$K_{\mathbb{F}_n}=-2C_n-(n+2)F_n.$$

Using the generators above  we can  write the bidouble cover divisor data as
\[ D_1=a_1C_n+b_1F_n, \quad  D_2=a_2C_n+b_2F_n, \quad D_3=a_3C_n+b_3F_n.
\]

We recall the following

\begin{lemma}\label{lemma: irredu on Fn}{\rm (See \cite[Chapter V, Corollary 2.18]{H77})} The class $aC_n+bF_n\in NS(\mathbb{F}_n)$ represents an irreducible curve on $\mathbb{F}_n$ if and only if one of the followings hold: either $a=0$, $b=1$ or $a=1, b=0$ or $b\geq an$.\end{lemma} 
We observe that $C_n$ is rigid, and $F_n$ is not. In particular, the divisors $bF_n$ represents a smooth divisor which is the union of $b$ different fibres, but $aC_n$, $a>1$ represents necessarily a non reduced divisor.

\begin{prop}\label{prop_k321fn} Let $S$ be a surface with at worst ADE singularities, which is a double cover of $\mathbb{F}_n$ with branch divisor $B$. The minimal model of $S$ is a K3 surface if and only if $0\leq n\leq 4$ and $B=4C_n+2(n+2)F_n$. 
	
	If the minimal model of $S$ is a K3 surface and $n\geq 3$, then $B$  is the union of $C_n$ and a curve (possibly reducible) whose class is $3C_n+2(n+2)F_n$.
\end{prop}
\proof Let us assume that $S$ is a K3 surface, then $B\simeq -2K_{\mathbb{F}_n}\simeq 4C_n+2(n+2)F_n$. The divisor $B$ has to be an effective divisor, and $BC_n=-2n+4$.  If $B$ does not contain $C_n$ as component then $-2n+4\geq 0$ and so $n\leq 2$. If $B$ is the union of $C_n$ and another curve, then the intersection of this curve with $C_n$ has to be positive and so $-n+4\geq 0$, i.e. $n\leq 4$. We observe that $C_n$ can not be a multiple component of the branch locus since the branch index is necessarily 1.

Viceversa, notice that if $S\ra \mathbb{F}_n$ is a double cover branched along a simple normal crossing  divisor $B\simeq 4C_n+2(n+2)F_n$ then $S$ is a K3 surface. Indeed, by the knowledge of the branch divisor one obtains that two times the canonical divisor is trivial, moreover one can calculate that $p_g(S)=1$ and $q(S)=0$ by standard formulae on the double cover. 
\endproof

From now on we assume $n>1$, since $\mathbb{F}_0\simeq \mathbb{P}^1\times \mathbb{P}^1$ and $\mathbb{F}_1$ is non minimal.

\begin{prop}\label{prop: Fn effective curves} Let $X\ra \mathbb{F}_n$ be a smooth bidouble cover as in Construction \ref{construction 1: first double cover}. Then $n\leq 4$ and 
	\begin{itemize}
		\item[a)] if both $D_1$  and $D_2$ do not have $C_n$ as component then $n=2$ and there are the following possibilities (up to possibly switching $D_1$ and $D_2$): 
 		$$(a_1,b_1,a_2,b_2)=(3,6,1,2), (2,4,2,4);$$	
		\item[b)]
	    if $D_1$ is the union of $C_n$ and other curves, $D_2$ does not contain $C_n$, then (up to possibly switching $D_1$ and $D_2$) $$2\leq n\leq 4,\ \ 1\leq a_1\leq 4\mbox{ and }b_1=na_1-n.$$

	\end{itemize}
\end{prop}
\proof We recall that by Construction \ref{construction 1: first double cover}, $Y_3$ is a K3 surface. So the condition on $n$ follows by the previous Proposition \ref{prop_k321fn}. Moreover, since 
$D_1=a_1C_n+b_1F_n$ and $D_2=a_2C_n+b_2F_n$, it hold $a_2=(4-a_1)$, $b_2=(2n+4-b_1)$. If $D_1$ and $D_2$ are irreducible by Lemma \ref{lemma: irredu on Fn}, $b_1\geq na_1$ and $2n+4-b_1\geq 4n-a_1$ which implies $n\leq 2$ and $b_1=2a_1$. We exclude the case $a_1=8$, because it would implies $D_2=D_3=0$.

Similarly, if one assumes that $D_1=C_n\cup B$ where $B=(a_1-1)C_n+b_1F_n$ and $D_2=(4-a_1)C_n+(2n+4-b_1)F_n$, one has to require that $BC_n= 0$ because we are assuming that the divisor $D_1$ is smooth. So
$BC_n=n-a_1n+b_1= 0\mbox{ then }b_1=a_1n-n.$\endproof

\begin{theorem}\label{theor: Y=Fn}
	Let $n>1$ and $X\ra \mathbb{F}_n$ be a smooth bidouble cover as in Construction \ref{construction 1: first double cover}.
	
	If all the divisors $D_i$ are irreducible, then $n=2$ and the  possibilities are listed in Table \ref{TableF2 irrid}.	
	\begin{table}[!h]
	\begin{tabular}{|c||c|c|c||c|c|c||c|c|c|c|c|c|c|}\hline
		case&$(a_1,b_1)$&$(a_2,b_2)$&$(a_3,b_3)$&$Y_1$&$Y_2$&$Y_3$&$p_g(X)$&$q(X)$&$K^2_X$&MTC&ITP\\
		\hline
		a)&$(3,6)$&$(1,2)$&$(1,2)$&{\rm rat}&{\rm K3}&{\rm K3}&$2$&$0$&$2$&&$\checkmark$\\
		b)&$(2,4)$&$(2,4)$&$(0,0)$&{\rm rat}&{\rm rat}&{\rm K3}&$1$&$0$&$0$&$\checkmark$&$\checkmark$\\
		c)&$(2,4)$&$(2,4)$&$(0,\geq 2)$&{\rm rat}&{\rm rat}&{\rm K3}&$1$&$0$&$0$&&\\
		d)&$(2,4)$&$(2,4)$&$(2,2)$&{\rm rat}&{\rm rat}&{\rm K3}&$1$&$0$&$0$&&\\
		e)&$(2,4)$&$(2,4)$&$(2,4)$&{\rm K3}&{\rm K3}&{\rm K3}&$3$&$0$&$ 8$&&\\
	\hline
	\end{tabular} 
		\caption{Bidouble cover of $\mathbb{F}_2$ with irreducible $D_i$'s}\label{TableF2 irrid}
	\end{table}
	
	Otherwise $D_1$ splits in the union of $C_n$ and an irreducible smooth curve, $n\leq 4$ and there are the following possibilities:
		
	$n=2$ and the admissible cases are listed in Table \ref{TableF2 rid}
	\begin{table}[!h]
	\begin{tabular}{|c||c|c|c||c|c|c||c|c|c|c|c|c|c|}\hline
		case&$(a_1,b_1)$&$(a_2,b_2)$&$(a_3,b_3)$&$Y_1$&$Y_2$&$Y_3$&$p_g(X)$&$q(X)$&$K^2_X$&MTC&ITP\\
		\hline
		f)&$(4,6)$&$(0,2)$&$(0,0)$&{\rm rat}&{\rm rat}&{\rm K3}&$1$&$0$&$0$&$\checkmark$&$\checkmark$\\
		g)&$(4,6)$&$(0,2)$&$(0,2)$&{\rm rul}&{\rm K3}&{\rm K3}&{\rm 2}&{\rm 1}&{\rm 0}&&\\
		h)&$(3,4)$&$(1,4)$&$(1,2)$&{\rm rat}&{\rm rat}&{\rm K3}&$1$&$0$&$2$&$\checkmark$&\\
		i)&$(3,4)$&$(1,4)$&$(1,4)$&{\rm rat}&{\rm K3}&{\rm K3}&$2$&$0$&$6$&$\checkmark$&\\
		j)&$(2,2)$&$(2,6)$&$(0,0)$&{\rm rat}&{\rm rat}&{\rm K3}&$1$&$0$&$0$&$\checkmark$&$\checkmark$\\
		k)&$(2,2)$&$(2,6)$&$(0,\geq 2)$&{\rm rat}&{\rm rat}&{\rm K3}&$1$&$0$&$0$&&\\
		\hline
	\end{tabular} 
		\caption{Bidouble cover of $\mathbb{F}_2$ with $D_1$ reducible}\label{TableF2 rid}
	\end{table}

$n=3$ and the admissible cases are listed in Table \ref{TableF3}	
\begin{table}[!h]
	\begin{tabular}{|c||c|c|c||c|c|c||c|c|c|c|c|c|c|}\hline
		case&$(a_1,b_1)$&$(a_2,b_2)$&$(a_3,b_3)$&$Y_1$&$Y_2$&$Y_3$&$p_g(X)$&$q(X)$&$K^2_X$&MTC&ITP\\
		\hline
		l)&$(4,9)$&$(0,1)$&$(0,1)$&{\rm rat}&{\rm K3}&{\rm K3}&$2$&$0$&$0$&$\checkmark$&\\
		m)&$(3,6)$&$(1,4)$&$(1,4)$&{\rm rat}&{\rm K3}&{\rm K3}&$2$&$0$&$5$&$\checkmark$&\\
		n)&$(2,3)$&$(2,7)$&$(0,\geq 1)$&{\rm rat}&{\rm rat}&{\rm K3}&$1$&$0$&$0$&&\\
		\hline
	\end{tabular} 
		\caption{Bidouble cover of $\mathbb{F}_3$ with $D_1$ reducible}\label{TableF3}
	\end{table}

$n=4$ and the admissible cases are listed in Table \ref{TableF4}.	
	\begin{table}[!h]
	\begin{tabular}{|c||c|c|c||c|c|c||c|c|c|c|c|c|}\hline
		case&$(a_1,b_1)$&$(a_2,b_2)$&$(a_3,b_3)$&$Y_1$&$Y_2$&$Y_3$&$p_g(X)$&$q(X)$&$K^2_X$&MTC&ITP\\
		\hline
		o)&$(3,8)$&$(1,4)$&$(1,4)$&{\rm rat}&{\rm K3}&{\rm K3}&$2$&$0$&$4$&$\checkmark$&\\
		p)&$(2,4)$&$(2,8)$&$(0,0)$&{\rm rat}&{\rm rat}&{\rm K3}&$1$&$0$&$0$&$\checkmark$&$\checkmark$\\
		q)&$(2,4)$&$(2,8)$&$(0,\geq 2)$&{\rm rat}&{\rm rat}&{\rm K3}&$1$&$0$&$0$&&\\
		\hline
	\end{tabular} 
		\caption{Bidouble cover of $\mathbb{F}_4$ with $D_1$ reducible}\label{TableF4}
	\end{table}
	
	The surface $X$ is minimal and in the cases  $b)$, $f)$, $j)$ and $p)$, it is a K3 surface; in cases $a)$, $e)$, $h)$, $i)$, $m)$, $o)$ it is of general type; otherwise it is properly elliptic.

	The surface $X$ has the following properties:
	$$T_X\otimes \Q\simeq (T_{Y_1}\oplus T_{Y_2}\oplus T_{Y_3})\otimes \Q,\ A_0(X)\otimes \Q\simeq (A_0(Y_1)\oplus A_0(Y_2)\oplus A_0(Y_3))\otimes \Q,$$
	$$h(X)_{\rm{tra}}\otimes \Q=(h(Y_1)_{\rm{tra}}\oplus h(Y_2)_{\rm{tra}}\oplus h(Y_3)_{\rm{tra}}) \otimes \Q.$$
	
	The Mumford--Tate conjecture and the Tate conjecture hold in the cases $b)$, $f)$, $h)$, $i)$, $j)$, $m)$, $o)$, $p)$.

	The Infinitesimal Torelli Property holds in the cases $a)$, $b)$, $f)$, $j)$ and $p)$.
\end{theorem}
\proof The proof is similar to the ones of Theorems \ref{theorem: bidouble P2} and \ref{theor: Y=P1xP1} and the tables are constructed by Proposition \ref{prop_k321fn}, observing that $C_n$ can be a component of $D=D_1\cup D_2\cup D_3$ with multiplicity at most 1 and so cannot be a component of $D_2$ or $D_3$ if it is already a component of $D_1$. To prove the minimality of $X$ we argue as in Theorems \ref{theorem: bidouble P2} and \ref{theor: Y=P1xP1} and we recall that the nef cone of $\mathbb{F}_n$ is spanned by $F_n$ and $C_n+nF_n$, see \cite{BS}. Moreover, applying \eqref{eq_K2BiDoCo} and \eqref{eq_chiBiDoCo} to our particular context, we obtain
\[
K^2_X=(2K_{\mathbb{F}_n}+D)^2=2\left(\sum^3_{i=1} a_i -4\right)\left(\sum^3_{i=1} b_i-2n-4\right)-n\left(\sum^3_{i=1} a_i -4\right)^2,
\]
and
\[
\begin{split}
\chi(X)=4+\frac{1}{8} & \big[(2a_2b_2+2a_2b_3+2a_2n-4a_2+2a_3b_2+2a_3b_3+2a_3n-4a_3-4b_2-4b_3-na_2^2-2na_2a_3-na_3^2)+\\
& (2a_1b_1+2a_1b_3+2a_1n-4a_1+2a_3b_1+2a_3b_3+2a_3n-4a_3-4b_1-4b_3-na_1^2-2na_1a_3-na_3^2)+ \\
& (2a_2b_2+2a_2b_1+2a_2n-4a_2+2a_1b_2+2a_1b_1+2a_1n-4a_1-4b_2-4b_1-na_2^2-2na_2a_1-na_1^2)\big].
\end{split}
\]
Notice that the rational fibration on $\mathbb{F}_n$, induces a genus 1 fibration on the K3 surface $Y_3$. In cases $c)$, $j)$, $k)$, $m)$ and $p)$ such a fibration induces a genus 1 fibration on $X$. Therefore $X$ is not of general type. Since some fibers of the genus 1 fibration on $Y_3$ are in the branch locus of the double cover $X\ra Y_3$, $X$ can not be of trivial Kodaira dimension. 

In cases $b)$, $f)$, $j)$, $p)$, $X$ is a K3 surface hence the Mumford-Tate conjecture holds. Similarly it holds in case $h)$ by \cite{Moonenh1}. In cases $i)$, $m)$, $o)$ it follows by Theorem \ref{Theo_MTC} because the hypothesis holds. Indeed, to compare $\rho(Y_i)$ and $\rho(Y)$, we compute the dimension of the family of $Y_i$, $i=2,3$. By Riemann-Roch theorem, if $b\geq an$, 
$$h^0(aC_n+bF_n)=1+\frac{1}{2}(aC_n+bF_n)((a+2)C_n+(b+n+2)F_n)=1+\frac{1}{2}(-na(a+1)+2ab+2a+2b).$$
The curve $C_n$ is rigid. The branch locus of $\pi_2\colon Y_2\ra Y$ consists of three curves whose classes are $C_n$, $(a_1-1)C_n+b_1F_n$, $a_3C_n+b_3F_n$. So the choice of the branch locus of $Y_2$ depends on $$\frac{-n(a_1-1)a_1+2(a_1-1)b_1+2(a_1-1)+2b_1}{2}+\frac{-na_3(a_3+1)+2a_3b_3+2a_3+2b_3}{2}$$
projective parameters. Since the dimension of the automorphisms of $\mathbb{F}_n$ is $n+5$, see e.g \cite{B12}, the family of $Y_2$ depends on $$\frac{-n(a_1-1)a_1+2(a_1-1)b_1+2(a_1-1)+2b_1}{2}+\frac{-na_3(a_3+1)+2a_3b_3+2a_3+2b_3}{2}-(n+5)$$ 
parameters. This allows to compute the Picard number of $Y_2$ and similarly of $Y_3$ as in proof of Theorem \ref{theorem: bidouble P2}.
We summarize the results on the Hodge decompositions of the surfaces for which the Mumford--Tate conjecture hold in the following table.
\begin{table}[!h]
	\begin{tabular}{|c||c|c|c|c|c|c|}\hline
		case&$h^{1,1}(X)$&$\rho(Y_1)$&$\rho(Y_2)$&$\rho(Y_3)$&$(h^{2,0}(Y_i),h^{11}(Y_i),h^{0,2}(Y_i))_{\rm{tra}}$& for $i=$\\
		\hline
		i) & 24 & 8 & 2 & 2 & $(1,8,1)$& 2,3 \\
		m) & 25 & 7 & 2 & 2 & $(1,9,1)$& 2,3 \\
		o) & 26 & 6 & 2 & 2 & $(1,10,1)$& 2,3 \\
		
		\hline
	\end{tabular}
\end{table}
 \endproof
\begin{rem}{\rm The surfaces constructed in Table \ref{TableF2 irrid} do not satisfy the hypothesis of Corollary \ref{corollary hodge structures and trascendental}, indeed $\rho(Y_2)>\rho(Y)$ for all these surfaces. The curve $C_2\subset \mathbb{F}_2$ is a rational curve not contained in the branch locus of the double cover $\pi_2:Y_2\ra Y$ and not intersecting the branch locus. Hence it splits in the double cover, i.e. $\pi_2^{-1}(C_2)$ consists of two disjoint rational curves on $Y_2$. These curves are not linearly equivalent to each other and this implies that $\rho(Y_2)\geq \rho(Y)+1$.}\end{rem}

\subsection{Iterated bidouble covers of $\mathbb{F}_n$}\label{subsec: iterated Fn}
Finally, we consider the iterated bidouble covers. Moreover, we restrict ourselves to the case where $X$ is of general type and $p_g(X)=2$. 

We set $\Gamma:=p_1^*(C_n)$ and $\Phi:=p_1^*F_n$ and we denote by the same letter the divisors on $Y_1$ and their strict transform on $\widetilde{Y_1}$. Then it holds 
$$\begin{array}{|c||c|c|c|c|}
	\hline
	\mbox{case}&-na_2a_3+a_2b_3+a_3b_2&K_{\widetilde{Y_1}}&\widetilde{B}&E\\
	\hline
	a)&2&-\Gamma-2\Phi&3\Gamma+6\Phi&E_1+E_2\\
	i)&6&-\Gamma&3\Gamma+4\Phi&\sum_{i=1}^6E_i\\	
	m)&5&-\Gamma-\Phi&3\Gamma+6\Phi&\sum_{i=1}^5E_i\\
	o)&4&-\Gamma-2\Phi&3\Gamma+8\Phi&\sum_{i=1}^4E_i\\
	\hline
	\end{array}
$$

The possible data for the iterated bidouble covers as in Construction \ref{construction 2: iterated bidouble case 1} and \ref{construction 2: iterated bidouble case 2} are listed in the following table.
$$\begin{array}{|c||c|c|c|c|c|c|c|c|c|c}
	\hline
\mbox{ case }&\Delta_1&\Delta_2&\Delta_3&\mbox{Construction}&p_g(W)\\
\hline
a)&2\Gamma+3\Phi&\Phi&\Gamma+3\Phi+E&\ref{construction 2: iterated bidouble case 1}&3\\
a)&2\Gamma+4\Phi&\emptyset&\Gamma+2\Phi+E&\ref{construction 2: iterated bidouble case 1}&3\\
a)&\Phi&\Phi&3\Gamma+5\Phi+E&\ref{construction 2: iterated bidouble case 2}&4\\
a)&2\Phi&2\Phi&3\Gamma+4\Phi+E&\ref{construction 2: iterated bidouble case 2}&4\\
a)&\Gamma+2\Phi&\Gamma+2\Phi&2\Gamma+4\Phi+E&\ref{construction 2: iterated bidouble case 2}&5\\
o)&2\Gamma+4\Phi&\emptyset&\Gamma+4\Phi+E&\ref{construction 2: iterated bidouble case 1}&3\\
\hline
\end{array}$$

We observe that there are no iterated bidouble covers as in Constructions \ref{construction 2: iterated bidouble case 1} and \ref{construction 2: iterated bidouble case 2} in cases $i)$ and $m)$. The list is obtained by a tedious but elementary case by case analysis, requiring that $\Delta_i$, $i=1,2,3$, satisfy the conditions given in Constructs \ref{construction 2: iterated bidouble case 1} or \ref{construction 2: iterated bidouble case 2} and that their image in $\mathbb{F}_n$ corresponds either to irreducible curves or to the union of a certain number of fibers (cf. proof of Proposition \ref{prop: Fn effective curves}).



\smallskip

Alice Garbagnati  Universit\`a degli Studi di Milano,  Dipartimento di Matematica \emph{''Federigo Enriques"}, I-20133 Milano, Italy \\
\emph{e-mail} \verb|alice.garbagnati@unimi.it|

\smallskip

Matteo Penegini, Universit\`a degli Studi di Genova, DIMA Dipartimento di Matematica, I-16146 Genova, Italy \\
\emph{e-mail} \verb|penegini@dima.unige.it|

\end{document}